\documentclass[11pt,letter]{amsart}
\usepackage{amsfonts,amsmath,amssymb}
\usepackage{hyperref}
\usepackage{latexsym,fullpage,amsfonts,amssymb,amsmath,amscd,graphics,epic}
\usepackage[all]{xy}
\usepackage{amssymb,amsthm,amsxtra}
\usepackage[usenames]{color}
\usepackage{amscd}
\usepackage{amsthm}
\usepackage{amsfonts}
\usepackage{amssymb}
\usepackage{mathrsfs}
\usepackage{url}
\usepackage{bbm}
\usepackage{wasysym}
\usepackage{enumerate}
\usepackage{enumitem}

\theoremstyle{plain}
\newtheorem*{theorem*}{Theorem}
\newtheorem*{remark*}{Remark}
\newtheorem*{example*}{Example}
\newtheorem{lemma}{Lemma}
\newtheorem{proposition}[lemma]{Proposition}
\newtheorem{corollary}[lemma]{Corollary}
\newtheorem{theorem}[lemma]{Theorem}

\newtheorem*{conjecture*}{Conjecture}

\newtheorem*{necessarycondition*}{Necessary Condition}

\numberwithin{lemma}{section}

\theoremstyle{definition}
\newtheorem{definition}[lemma]{Definition}

\newtheorem{example}[lemma]{Example}

\theoremstyle{remark}
\newtheorem{remark}[lemma]{Remark}

\newtheorem{notation}[lemma]{Notation}

\newcommand{\bR}{{\mathbb R}}

\newcommand{\abs}[1]{\left|{#1}\right|}

\def\quotient#1#2{%
    \raise1ex\hbox{$#1$}\Big/\lower1ex\hbox{$#2$}%
}

\usepackage[normalem]{ulem}

\begin{document}

\date{November 23, 2022}
 \title[Classification of implication-closed ideals in certain
rings of jets]{Classification of implication-closed ideals in certain rings of jets}
  \author{Charles Fefferman$^\dag$ and Ary Shaviv$^\dag$$^\dag$}
 \address{Department of Mathematics, Princeton University, Princeton, New
 Jersey 08544}
 \email{cf@math.princeton.edu; ashaviv@math.princeton.edu}
\thanks{$^\dag$C.F. was supported by the Air Force Office of Scientific Research
(AFOSR),
under award FA9550-18-1-0069 and the National Science Foundation (NSF), under
grant DMS-1700180.}
\thanks{$^\dag$$^\dag$A.S.   was supported
by the Air Force Office of Scientific Research (AFOSR), under award FA9550-18-1-0069.
 }

\maketitle

\begin{abstract}
For a set $E\subset\mathbb{R}^n$ that contains the origin we consider $I^m(E)$
-- the set of all $m^{\text{th}}$ degree Taylor approximations (at the origin)
of $C^m$ functions on $\mathbb{R}^n$ that vanish on $E$. This set is a proper ideal in $\mathcal{P}^m(\mathbb{R}^n)$ -- the ring of all $m^{\text{th}}$ degree
Taylor approximations of $C^m$ functions on $\mathbb{R}^n$. In \cite{FS2} we introduced the notion of a \textit{closed} ideal in  $\mathcal{P}^m(\mathbb{R}^n)$,
and proved that any ideal of the form $I^m(E)$ is closed. In this paper we classify (up to a natural equivalence relation) all closed ideals in  $\mathcal{P}^m(\mathbb{R}^n)$ in all cases in which $m+n\leq5$. We also show that in these cases the converse also holds -- all closed proper ideals in $\mathcal{P}^m(\mathbb{R}^n)$
arise as $I^m(E)$ when $m+n\leq5$. In addition, we prove that in these cases any ideal of the form $I^m(E)$ for some $E\subset\mathbb{R}^n$ that contains the origin already arises as $I^m(V)$ for some semi-algebraic $V\subset\mathbb{R}^n$ that contains
the origin. By doing so we prove that a conjecture by N. Zobin holds true in these cases.
\end{abstract}

\tableofcontents

\section{Introduction}

Fix natural numbers $m$ and $n$, and consider the ring $\mathcal{P}^{m}(\mathbb{R}^n)$ (respectively $\mathcal{P}_{0}^{m}(\mathbb{R}^n)$)
of $m$-jets at the origin of $C^m$ functions on $\mathbb{R}^n$ (that vanish
at the origin). We write $J^m(F)$ to denote the $m$-jet at $\vec 0$ of a function
$F\in C^m(\mathbb{R}^n)$, i.e., its $m^{\text{th}}$ degree Taylor approximation about
the origin. Given a subset $E\subset\mathbb{R}^n$ that contains the origin,
define an ideal $I^{m}(E)$ in $\mathcal{P}_{0}^{m}(\mathbb{R}^n)$ by $$I^{m}(E):=\{J^{m}(F)\text{
}|\text{ } F \in C^m(\mathbb{R}^n)\text{ and } F=0 \text{ on
} E\}.$$ Which ideals in $\mathcal{P}_0^m(\mathbb{R}^n)$ arise as $I^m(E)$ for some $E$? In \cite{FS2} we introduced the notion of a \textit{closed} ideal in $\mathcal{P}^m_0(\mathbb{R}^n)$, and we showed that for any $m,n$ and and $E$, the ideal $I^m(E)$ is always closed (Theorem \ref{main_theorem_on_necessary_condition_SECOND_PAPER} below). We do not know whether in general the converse also holds, namely whether for arbitrary $m$ and $n$ every
closed ideal in $\mathcal{P}_0^m(\mathbb{R}^n)$ arises as $I^{m}(E)$ for
some $E$. However, in this paper we provide some evidence that this might be the case: we prove that if either $m=1$, $n=1$ or $m+n\leq5$, then indeed every
closed ideal in $\mathcal{P}_0^m(\mathbb{R}^n)$ arises as $I^{m}(E)$ for
some $E$. We moreover show that for arbitrary $n$, a principal ideal in $\mathcal{P}^2_0(\mathbb{R}^n)$ is closed if and only if it is of the form $I^2(E)$ for some $E$. 
\subsection*{Constructive classification approach} To prove the above: in the rings $\mathcal{P}_{0}^{1}(\mathbb{R}^n),\mathcal{P}_{0}^{m}(\mathbb{R}^1),$ $\mathcal{P}_{0}^{2}(\mathbb{R}^2),\mathcal{P}_{0}^{3}(\mathbb{R}^2)$
and $\mathcal{P}_{0}^{2}(\mathbb{R}^3)$ we provide lists of closed ideals, and prove that under a natural equivalence relation, any closed ideal in these rings is equivalent to one (and only one) ideal on those lists  (we do the same in $\mathcal{P}^2_0(\mathbb{R}^n)$ for principal ideals and some other ideals with few generators). For each ideal $I$ on these lists we exhibit a set $E$ such that
$I=I^m(E)$. The natural equivalence relation arises from the
fact that any $C^m$-diffeomorphism $\phi:\mathbb{R}^{n}\to\mathbb{R}^n$
that
fixes the origin induces an automorphism of $\mathcal{P}_0^m(\mathbb{R}^n)$
defined by $p\mapsto J^{m}(p\circ\phi)$. We say that two ideals $I$ and $I'$
in $\mathcal{P}_0^m(\mathbb{R}^n)$
are \textit{equivalent} if $I'$ is the image of $I$ under such an automorphism.
We in particular prove that any closed ideal in the
rings $\mathcal{P}_{0}^{1}(\mathbb{R}^n),\mathcal{P}_{0}^{m}(\mathbb{R}^1),\mathcal{P}_{0}^{2}(\mathbb{R}^2),\mathcal{P}_{0}^{3}(\mathbb{R}^2)$
and $\mathcal{P}_{0}^{2}(\mathbb{R}^3)$  arises as
$I^m(E)$
for some closed $E$ that contains the origin.

\subsection*{Semi-algebraicity}Recall that semi-algebraic sets
and
maps are those that can be described by finitely many polynomial equations
and inequalities
and boolean operations (for detailed exposition on semi-algebraic geometry
see, for instance, \cite{BCR}).

Note that $I^m(E)$ only depends on the local behaviour of $E$ about the origin. We remark that if $I'$ is the image of $I$ under the automorphism induced by a given $C^m$-diffeomorphism $\phi:\mathbb{R}^{n}\to\mathbb{R}^n$
that
fixes the origin, then $I'$ is also the image of $I$ under the $m^{th}$ degree approximation of $\phi$, which is a \textit{semi-algebraic} $C^m$-diffeomorphism of some neighborhood of the origin. 

Since the sets  $E$ constructed above are all \textit{semi-algebraic},  we in fact prove that each closed ideal in the lists above arises
as $I^{m}(E)$ for a semi-algebraic set $E$ that contains the origin. A conjecture
of N. Zobin (see \cite[Problem 5]{F}) asserts that every ideal $I^{m}(E)$
in $\mathcal{P}_0^m(\mathbb{R}^n)$
already arises as the ideal $I^{m}(V)$ for a semi-algebraic set $V$ that
contains the origin. Since the image of any semi-algebraic set under a semi-algebraic
map is also a semi-algebraic set, the remark above tells us that we in fact proved that the conjecture is true for the cases
in which either $m=1$, $n=1$ or $m+n\leq 5$, and in the case of principal ideals in $\mathcal{P}^2_0(\mathbb{R}^n)$.

\subsection*{Main results in a nutshell}In view of Theorem \ref{main_theorem_on_necessary_condition_SECOND_PAPER}
below, the main results of this paper can be summarized by the following:

\begin{theorem}\label{second_paper_theorem} Let $I\lhd\mathcal{P}_0^m(\mathbb{R})$
be an ideal, and assume one of the following holds:

\begin{enumerate}
\item Either $m=1$, $n=1$ or $m+n\leq5$.
\item $m=2$ and $I$ is principal.
\item $m=2$ and there
exist a partition $n=n_1+n_2$
and two homogenous polynomials of degree 2, $p_1\in\bR[x_1,\dots,x_{n_1}]$
and $p_2\in\bR[x_{n_1+1},\dots,x_{n}]$, such that $I=\langle p_1,p_2\rangle_2$.
\end{enumerate}

Then, the following are equivalent:
\begin{itemize}
\item $I$ is closed.
\item There exists a closed set $\vec 0 \in E\subset\mathbb{R}^n$ such that
$I=I^m(E)$.
\item There exists a closed semi-algebraic set $\vec 0 \in E\subset\mathbb{R}^n$
such that $I=I^m(E)$.
\end{itemize}\end{theorem}

We prove Case (1) of Theorem \ref{second_paper_theorem} by proving the following classification results (below $\langle p_1,\dots p_k\rangle_m$ stands for the ideal in $\mathcal{P}_0^m(\mathbb{R}^n)$ generated by the jets $p_1,\dots,p_k$):

\begin{theorem}[Classifications in $\mathcal{P}_0^1(\mathbb{R}^n)$ and $\mathcal{P}_0^m(\mathbb{R}^1)$]\label{theorem_on_almost_trivial_cases} \

\begin{enumerate}
\item Any ideal in $\mathcal{P}_0^1(\mathbb{R}^n)$
is both closed and of the form $I^1(E)$ for some semi-algebraic $E\subset\mathbb{R}^n$  that
contains the origin. 
\item The only closed ideals in $\mathcal{P}_0^m(\mathbb{R})$ are $\{0\}=I^m(\mathbb{R})$
and $\mathcal{P}_0^m(\mathbb{R})=I^m(\{\vec 0\})$.
\end{enumerate}

\end{theorem}

\begin{theorem}[Classification in $\mathcal{P}_0^2(\mathbb{R}^2)$]\label{theorem_on_2_2} Let $(x,y)$ be a standard coordinate system on
$\mathbb{R}^2$. Then, any closed ideal in $\mathcal{P}_0^2(\mathbb{R}^2)$
is equivalent to
one of the ideals in Table \ref{tab:table2}: 
\begin{table}[h!]
  \begin{center}
    \caption{Closed ideals in $\mathcal{P}_0^2(\mathbb{R}^2)$.}
    \label{tab:table2}
    \begin{tabular}{c|c}    
      \hline
      \textbf{\# of generators} & {\textbf{Ideals}} \\ 
      \hline
      1 & $\{0\},\langle
x\rangle_2,\langle
x^2\rangle_2,\langle
xy\rangle_2$ \\
      \hline
      2 & $\langle
x^2,xy\rangle_2,\mathcal{P}_0^2(\mathbb{R}^{2})$ \\
      \hline

    \end{tabular}
  \end{center}
\end{table}

Moreover, any ideal in Table \ref{tab:table2} arises as $I^2(E)$ for
some semi-algebraic set $E\subset\mathbb{R}^2$ that contains the origin,
and any two different
ideals in Table \ref{tab:table2} are non-equivalent.
\end{theorem}

\begin{theorem}[Classification in $\mathcal{P}_0^3(\mathbb{R}^2)$]\label{theorem_on_3_2} Let $(x,y)$
be a standard coordinate system on
$\mathbb{R}^2$. Then, any closed ideal in $\mathcal{P}_0^3(\mathbb{R}^2)$
is equivalent to
one of the ideals in Table \ref{tab:table3}: 
\begin{table}[h!]
  \begin{center}
    \caption{Closed ideals in $\mathcal{P}_0^3(\mathbb{R}^2)$.}
    \label{tab:table3}
    \begin{tabular}{c|c}    
      \hline
      \textbf{\# of generators} & {\textbf{Ideals}} \\ 
      \hline
      1 & $\{0\},\langle
x\rangle_3,\langle
x^{2}\rangle_3,\langle
x^3\rangle_3,\langle
xy\rangle_3,\langle
yx^2\rangle_3,\langle
xy(x+y)\rangle_3,\langle
x^2-y^3\rangle_3$ \\
      \hline
      2 & $\langle
x^2,xy\rangle_3,\langle
x^2,xy^2\rangle_3,\langle x^2y,xy^2\rangle_3,\langle
x^3,x^2y\rangle_3,\langle
x^2-y^3,xy^2\rangle_3,\mathcal{P}_0^3(\mathbb{R}^{2})$

\\
      \hline

      3 & $\langle
x^3,x^{2}y,xy^{2}\rangle_3$ \\
      \hline

    \end{tabular}
  \end{center}
\end{table}

Moreover, any ideal in Table \ref{tab:table3} arises as $I^3(E)$ for
some semi-algebraic set $E\subset\mathbb{R}^2$ that contains the origin,
and any two different
ideals in Table \ref{tab:table3} are non-equivalent.
\end{theorem}

\begin{theorem}[Classification in $\mathcal{P}_0^2(\mathbb{R}^3)$]\label{theorem_on_2_3} Let $(x,y,z)$
be a standard coordinate system on
$\mathbb{R}^3$. Then, any closed ideal in $\mathcal{P}_0^2(\mathbb{R}^3)$
is equivalent to
one of the ideals in Table \ref{tab:table4}: 
\begin{table}[h!]
  \begin{center}
    \caption{Closed ideals in $\mathcal{P}_0^2(\mathbb{R}^3)$.}
    \label{tab:table4}
    \begin{tabular}{c|c}    
      \hline
      \textbf{\# of generators} & {\textbf{Ideals}} \\ 
      \hline
      1 & $\{0\},\langle x \rangle_2,\langle x^2 \rangle_2,\langle x^2+y^2 \rangle_2,\langle
x^2-y^2 \rangle_2,\langle x^2+y^2-z^2 \rangle_2$ \\
      \hline
      2 & $\langle
x,y \rangle_2,\langle x,y^2 \rangle_2,\langle x^2,xy \rangle_2,\langle x^{2} ,y^{2}\rangle_2,\langle xy,y^{2} -x^{2}\rangle_2$
\\

             & $\langle x,yz
\rangle_2,\langle
xy,yz \rangle_2,\langle x^2,yz \rangle_2,\langle
xy,yz-x^2 \rangle_2,$
\\

             & $\langle x^2,xy+xz+yz \rangle_2,\langle xy,z(x+y) \rangle_2,\langle xy+yz,xz+yz \rangle_2$

\\
      \hline 
      3 & $\langle x,y^2,yz
\rangle_2,\langle x^2,y^2,xy \rangle_2,\langle xy,xz,yz \rangle_2,$
\\
   
       & $\langle x^2,xy,yz \rangle_2,\langle
x^2,xy,xz \rangle_2,\langle
x^2,y^2+xz,xy
\rangle_2,\mathcal{P}_0^2(\mathbb{R}^3)$

\\

      \hline 
            4 & $\langle x^2,xy,xz,yz \rangle_2,\langle x^2,y^2,xy,xz \rangle_2$
\\
      \hline 
            5 & $\langle x^2,y^2,xy,xz,yz
\rangle_2$
\\
      \hline

    \end{tabular}
  \end{center}
\end{table}

Moreover, any ideal in Table \ref{tab:table4} arises as $I^2(E)$ for
some semi-algebraic set $E\subset\mathbb{R}^3$ that contains the origin,
and any two different
ideals in Table \ref{tab:table4} are non-equivalent.
\end{theorem}

\subsection*{The phenomenon of implied jets -- motivating example}The notion of a closed ideal is motivated by the following phenomenon: given jets $p_1,\dots,p_k,p_*\in\mathcal{P}_0^n(\mathbb{R}^n)$, it may happen that \begin{gather}\label{label_for_motivation_part_1}\text{For any set }E\subset\mathbb{R}^n,\text{ if }p_1,\dots,p_k\in I^m(E)\text{ then also }p_*\in I^m(E).\end{gather} In Section \ref{section_settings} we introduce the exact definition of a jet $p_*$ being \textit{implied} by an ideal $I\lhd\mathcal{P}_0^m(\mathbb{R}^n)$. We say that an ideal $I\lhd\mathcal{P}_0^m(\mathbb{R}^n)$ is \textit{implication-closed}, or simply \textit{closed}, provided that any jet that is implied by $I$ also belongs to $I$. In \cite{FS2} we prove that if a jet $p_*$ is implied by an ideal $I\lhd\mathcal{P}_0^m(\mathbb{R}^n)$, then there exist jets $p_1,\dots,p_k\in I$ such that (\ref{label_for_motivation_part_1}) holds. Hence, $I^m(E)$ is always closed, for any set $E$ that contains the origin. In this introduction, we content ourselves with a simple example to illustrate (\ref{label_for_motivation_part_1}).

We work in the ring of 2-jets of $C^2$-smooth functions on $\mathbb{R}^3$, where $(x,y,z)$
is a standard coordinate system on
$\mathbb{R}^3$. We will show that the following holds:\begin{multline}\label{label_for_motivation_part_2}\text{For
any set }E\subset\mathbb{R}^3,\\ \text{ if }p_1:=x^{2}+y^{2}\in I^m(E)\text{ and }p_2:=xz\in I^m(E)\text{ then
also }p_*:=x^{2}\in I^m(E).\end{multline}Note that (\ref{label_for_motivation_part_2}) implies, in particular, that the ideal in $\mathcal{P}_0^2(\mathbb{R}^3)$ generated by $x^2+y^2$ and $xz$ does not arise as $I^2(E)$ for any set $E$ that contains the origin. Here is an elementary proof of (\ref{label_for_motivation_part_2}):

Let $E\subset \mathbb{R}^3$ be such that the assumptions of (\ref{label_for_motivation_part_2}) hold. We first use the jet $p_1$ to control the geometry of $E$. Since $p_1:=x^{2}+y^{2}\in I^m(E)$ there exists $F\in C^2(\mathbb{R}^3)$ such that $J^2(F)=x^2+y^2$ and $F|_E=0$. Explicitly, we have $F(x,y,z)=x^2+y^2+o(x^2+y^2+z^2)$. This implies that on $E$ we have $\abs{(x,y)}=o(\abs{z})$, i.e., $E$ is "tangent" to the $z$-axis. We introduce 2 "cones": $$\Gamma_1:=\{(x,y,z)\in\mathbb{R}^3:\abs{(x,y)}<\abs{z}\}\text{ and }\Gamma_2:=\{(x,y,z)\in\mathbb{R}^3:\abs{(x,y)}<2\abs{z}\}.$$Since $E$ is "tangent" to the $z$-axis, we have \begin{gather}\label{random_label_for_intro_example_pp}E\cap B\subset \Gamma_1\cup\{\vec 0\}\text{ for some small ball }B\text{ about the origin}.\end{gather} Now we use $p_2=xz$. Since $p_2\in
I^m(E)$ there exists $\tilde F\in C^2(\mathbb{R}^3)$ such that $J^2(\tilde F)=0$
and \begin{gather}\label{random_label_for_intro_example_qq}xz+\tilde F=0\text{ on }E.\end{gather}Note that since $J^2(\tilde
F)=0$ we have \begin{gather}\label{random_label_for_intro_example_a}\abs{\partial^\alpha\tilde F(\vec x)}=o(\abs{\vec x}^{2-\abs{\alpha}})\text{ for any multi-index }\abs{\alpha}\leq2.\end{gather}Let $\theta\in C^2(\mathbb{R}^3\setminus\{\vec 0\})$ be a homogeneous function of degree 0 such that $\theta=1$ on $\Gamma_1$ and $\theta=0$ outside $\Gamma_2$. Define $S(\vec x):=\theta(\vec x)\cdot\frac{x}{z}$. One easily checks that \begin{gather}\label{random_label_for_intro_example_b}\abs{\partial^\alpha S(\vec x)}=O(\abs{\vec
x}^{-\abs{\alpha}})\text{ on }\mathbb{R}^3\setminus\{\vec 0\}\text{ for any multi-index }\abs{\alpha}\leq2.\end{gather}Consequently, (\ref{random_label_for_intro_example_a}) and (\ref{random_label_for_intro_example_b}) imply that \begin{gather}\label{random_label_for_intro_example_c}\abs{\partial^\alpha(
S\tilde F)(\vec x)}=o(\abs{\vec
x}^{2-\abs{\alpha}})\text{ on }\mathbb{R}^3\setminus\{\vec 0\}\text{ for any
multi-index }\abs{\alpha}\leq2.\end{gather} Setting $(S\tilde F)(\vec 0):=0$ we get \begin{gather}\label{random_label_for_intro_example_t}S\tilde F\in C^2(\mathbb{R}^3)\text{ and }J^2(S\tilde F)=0.\end{gather}We now multiply (\ref{random_label_for_intro_example_qq}) by $S$. Since $\theta=1$ on $\Gamma_1$ we have that $S=\frac{x}{z}$ on $\Gamma_1$, and so \begin{gather}0=S\cdot(xz+\tilde F)=x^2+S\tilde F\text{ on }E\cap\Gamma_1.\end{gather} Recalling (\ref{random_label_for_intro_example_pp}) and noting that $x^2+S\tilde F=0$ at the origin we conclude that \begin{gather}\label{random_label_for_intro_example_t_2}x^{2}+S\tilde F=0\text{ on }E\cap B\text{ for some small ball }B\text{ about
the origin}.\end{gather}Finally, we set $F^\#=\chi\cdot(x^2+S\tilde F)$, where $\chi\in C^\infty(\mathbb{R}^3)$ is supported in $B$ and $\chi=1$ in a neighborhood of the origin. From (\ref{random_label_for_intro_example_t}) and (\ref{random_label_for_intro_example_t_2}) we conclude that $F^\#\in C^2(\mathbb{R}^3)$, $J^2(F^\#)=x^2$ and $F^\#=0$ on $E$. Therefore, $x^2\in I^2(E)$, and so (\ref{label_for_motivation_part_2}) holds.

We stress that in this proof $p_1$ and $p_2$ play very different roles: we use the fact that $p_1\in I^2(E)$ in order to prove that $E$ lives away from the singularities of $\frac{x}{z}$. Thus, we know that there exists a cutoff function $\theta$ that is supported away from the singularities of $\frac{x}{z}$ and equals 1 on $E$. Then, we use the fact that $p_2\in I^2(E)$ and the existence of $\theta$ in order to construct a function that vanishes on $E$ and has a 2-jet $x^2$. In the terminology we develop in \cite{FS2} and formally introduce below in Section \ref{section_settings}, we only use $p_1$ in order to determine the possible \textit{allowed directions} of $I^2(E)$.

\subsection*{Structure of this paper} In \textbf{Section \ref{section_settings}} we set notation and state all the results about the theory of implication-closed ideals from  \cite{FS2} that we are using in this paper. 

In \textbf{Section \ref{section_easy_examples}} we calculate some simple examples of ideals of the form $I^m(E)$ for concrete sets $E$; these examples will be useful in later sections.

 \textbf{Section \ref{section_m_or_n_is_1}} is dedicated to the rings $\mathcal{P}_0^m(\mathbb{R})$ and $\mathcal{P}_0^1(\mathbb{R}^n)$, namely to the proof of Theorem \ref{theorem_on_almost_trivial_cases}.

\textbf{Section \ref{section_2_2}}
is dedicated to the ring $\mathcal{P}_0^2(\mathbb{R}^2)$,
namely to the proof of Theorem \ref{theorem_on_2_2}. 

In \textbf{Section \ref{section_special_examples_with_m_is_2}} we study principal ideals in $\mathcal{P}_0^2(\mathbb{R}^n)$ and some other ideals in this ring with few generators; the results of this section are used in Section \ref{section_3_2} and Section \ref{section_2_3}.

 \textbf{Section \ref{section_3_2}}
is dedicated to the ring $\mathcal{P}_0^3(\mathbb{R}^2)$,
namely to the proof of Theorem \ref{theorem_on_3_2}. 

Finally, \textbf{Section
\ref{section_2_3}}
is dedicated to the ring $\mathcal{P}_0^2(\mathbb{R}^3)$,
namely to the proof of Theorem \ref{theorem_on_2_3}.

\subsection*{Acknowledgments}We are grateful to Nahum Zobin for his beautiful conjecture. We thank the participants
of the Whitney Extension Problems Workshops over the years, the institutions
that hosted and supported these workshops -- American Institute of Mathematics,
Banff International Research Station, the College of William and Marry, the
Fields Institute, the Technion, Trinity College Dublin and the  University
of Texas at Austin -- and the NSF, ONR, AFOSR and BSF for generously supporting
these workshops. 

\section{The setting}\label{section_settings}

\subsection*{Notation}We work in
$\mathbb{R}^n$ with Euclidean
metric, and most of the notation we use is standard. 

\

\textbf{Functions.} For an open subset $U\subset\mathbb{R}^n$ and $m\in\mathbb{N}\cup\{0\}$
we denote by $C^m(U)$ the space of real valued $m$-times continuously differentiable
functions. 
We use multi-index notation for derivatives: for a multi-index $\alpha:=(\alpha_1,\alpha_2,\dots,\alpha_n)\in(\mathbb{N}\cup\{0\})^n$
we set $\abs{\alpha}:=\alpha_1+\dots+\alpha_n$ and $\alpha!:=\alpha_1!\alpha_2!\dots\alpha_n!$.
For $\vec x=(x_1,x_2,\dots, x_n)\in\mathbb{R}^n$ we set $\vec x^{\alpha}:=x_1^{\alpha_1}x_2^{\alpha_2}\dots
x_n^{\alpha_n}$. If $\abs{\alpha}\leq m$ and $f\in C^m(U)$ we write $$f^{(\alpha)}:=\frac{\partial^{|\alpha|}f}{\partial
x_1^{\alpha_1}\partial x_2^{\alpha_2}\cdots\partial
x_n^{\alpha_n}}$$ when $\abs{\alpha}\neq0$, and $f^{(\alpha)}:=f$ when $\abs{\alpha}=0$.
We sometimes write $\partial^\alpha f$ or $\partial_\alpha f$ instead of
$f^{(\alpha)}$. When it is clear from the context we will sometimes write $f_{xy}$
or $\partial^2_{xy}f$ when $\alpha=(1,1,0,0,\dots,0)$ and $(x,y,z,\dots)$
is a coordinate system on $\mathbb{R}^n$, and other such similar conventional
notation. 

\

\textbf{Asymptotic behaviors.} For $f,g\in C^m(U)$ we
write $$f(\vec x)=o(g(\vec x))$$ if$$\frac{f(\vec
x)}{g(\vec x)}\to0\text{ as }\vec x\to\vec 0.$$ We also write
$$f(\vec x)=O(g(\vec
x))$$ if$$\abs{\frac{f(\vec x)}{g(\vec x)}}\text{
is bounded in some punctured neighborhood of }\vec 0.$$  Finally, we write $$f(\vec
x)=\Theta(g(\vec
x))$$ if$$f(\vec x)=O(g(\vec
x))\text{ and } g(\vec x)=O(f(\vec
x)).$$ 

\

\textbf{Balls, cones, annuli and other geometric objects.}  For $r>0$ we
set $B(r):=\{\vec
x\in\mathbb{R}^n:\abs{\vec
x}<r\}$ and set $B^{\times}(r):=B(r)\setminus\{\vec 0\}$, where $n$ should
be
clear from the context. For two sets $X,Y\subset\mathbb{R}^n$ we
set $$\text{dist}(X,Y):=\inf\{\abs{\vec x-\vec y}:\vec x\in X,\vec y\in Y\}\text{
if }X\neq\emptyset\text{ and }Y\neq\emptyset\text{ };\text{ }\text{dist}(X,Y)=+\infty\text{
otherwise,}$$
and if one of them is a singleton we write $\text{dist}(\vec x,Y)$ instead
of $\text{dist}(\{\vec x\},Y)$, and similarly if the other is a singleton.
We denote
as usual $\bR^n\supset S^{n-1}:=\{\abs{\vec x}=1\}$ and refer to points in
$S^{n-1}$ as \textit{directions}. For $\Omega\subset S^{n-1}$
and $\delta>0$ we denote (the
dome around $\Omega$ of opening $\delta$) $$D(\Omega,\delta):=\{\vec
\omega\in S^{n-1}: \text{dist}(\omega,\Omega)<\delta\}.$$ Note that in particular
if $\Omega=\emptyset$ then $D(\Omega,\delta)=\emptyset$. Given (radius)
$r>0$  we set $$\Gamma(\Omega,\delta,r):=\bigcup\limits_{r'\in(0,r)}r' D(\Omega,\delta)=\{\vec
x\in\mathbb{R}^n:0<\abs{\vec x}<r,\text{dist}(\frac{\vec x}{\abs{\vec x}},\Omega)<\delta\}.$$Note
that in particular
if $\Omega=\emptyset$ then $\Gamma(\Omega,\delta,r)=\emptyset$. 

For
a singleton $\omega\in S^{n-1}$ we write $D(\omega,\delta)$ and $\Gamma(\omega,\delta,r)$
instead of  $D(\{\omega\},\delta)$ and $\Gamma(\{\omega\},\delta,r)$ (respectively).
We call a set of the form $\Gamma(\omega,\delta,r)$ (resp. $\Gamma(\Omega,\delta,r)$)
a cone in the direction $\omega$ (around the set of directions $\Omega\subset
S^{n-1}$), or a conic neighborhood of $\omega$ (of $\Omega$). Note that some
cones are non-convex. Also note that for any $\Omega\subset
S^{n-1}$, $\delta>0$ and $r>0$, we have that $\Gamma(\Omega,\delta,\epsilon)$
is open in $\bR^n$, and $D(\Omega,\delta)$ is open in $S^{n-1}$ (in the
restricted topology). Finally (when $n$ is clear from the context), for $\mathbb{R}\ni
K\geq 1$ and $\mathbb{R}\ni r>0$ we define the annulus$$\text{Ann}_K(r):=\{\vec
x\in\mathbb{R}^n:\frac{r}{K}<\abs{\vec x}<Kr\}.$$

\subsection*{Preliminaries}Below we collect the main definitions and results from \cite{FS2} that we will use.

\begin{definition}[allowed and forbidden directions of jets and ideals; see {\cite[Definition 2.6]{FS2}}]\label{definition_allowed_and_forbidden_directions_SECOND_PAPER}
\
\begin{itemize}
\item Let $p_{1},p_2,\dots,p_L\in\mathcal{P}_0^m(\bR^n)$ be jets and let
$\omega\in S^{n-1}$ be a direction. We say that
$\omega$ is a \textit{forbidden direction} of $p_1,\dots p_L$ if the following
holds:
\begin{multline}\label{equiv_cond_one_for_forbidden}
  \text{There exist }c,\delta,r>0 \text{
such that} \\ \abs{p_1(\vec x)}+\abs{p_2(\vec x)}+\dots+\abs{p_L(\vec x)}>
c\cdot\abs{\vec
{x}}^m \text { for all }\vec x\in\Gamma(\omega,\delta,r).
\end{multline}Otherwise, we say that $\omega$ is an \textit{allowed direction}
of $p_1,\dots p_L$.
\item Let $I\lhd\mathcal{P}_0^{m}(\bR^n)$ be an ideal. A direction $\omega\in
S^{n-1}$
is said to be a \textit{forbidden direction} of $I$, if there exist $p_{1},p_2,\dots,p_L\in
I$ such that $\omega$ is a forbidden direction of $p_{1},p_2,\dots,p_L$.
Otherwise, we say that
$\omega$ is an \textit{allowed direction} of $I$.
\end{itemize}
We denote
the sets of forbidden and allowed directions of $I$ by $Forb(I)$ and $Allow(I)$
respectively. Note that the set $Allow(I)\subset S^{n-1}$ is always closed. We will often use the fact that $\dim\text{span}_{\mathbb{R}}Allow(I)$ is invariant with respect to $C^m$
coordinate changes around the origin (see \cite[Remark 2.16]{FS2}). 
\end{definition}

\begin{definition}[tangent, forbidden and allowed directions of a set; see {\cite[Definition 2.7]{FS2}}]\label{def_tangent_directions_of_a_set}Let
$E\subset\bR^n$ be a closed subset containing the origin
and let $\omega\in S^{n-1}$. We say that \emph{$E$ is tangent to the direction
$\omega$} if \begin{gather}E\cap\Gamma(\omega,\delta,r)\neq\emptyset\text{
for all }\delta>0 \text{ and all }r>0.\end{gather} We denote by $T(E)\subset
S^{n-1}$ the set of all directions to which $E$
is tangent. Finally, for a fixed $m\in\mathbb{N}$, we say that \emph{$\omega$
is a forbidden (resp. allowed) direction
of $E$}, if $\omega$ is a forbidden (allowed) direction of $I^m(E)$. We denote
the sets of forbidden and allowed directions of $E$ by $Forb(E)$ and $Allow(E)$
respectively (here $m$ should be implicitly
understood from the context).\end{definition}

\begin{lemma}[see
{\cite[Lemma 2.8]{FS2}}]\label{tangent_subset_bad_lemma_SECOND_PAPER}Let $\vec 0\in E\subset\bR^n$
be a closed subset and $m\in\mathbb{N}$. Then, $T(E)\subset Allow(E)$.\end{lemma}

\begin{definition}[implied jets; see {\cite[Definition 2.10]{FS2}}]\label{new_def_implied_jet_SECOND_PAPER}Let $I\lhd\mathcal{P}_0^{m}(\bR^n)$
be an ideal, denote $\Omega:=Allow(I)$ and let $p\in\mathcal{P}_0^{m}(\bR^n)$
be some polynomial. We
say that\textit{ $I$ implies $p$} (or that $p$ is implied by $I$) if there
exist a constant $A>0$ and $Q_1,Q_2,\dots, Q_L\in I$ such that the following
holds:

For any $\epsilon>0$ there exist $\delta,r>0$ such that for any $0<\rho\leq
r$ there exist functions $F,S_1,S_2,\dots,S_L\in C^m(\text{Ann}_4(\rho))$
satisfying \begin{gather}\label{new_def_of_implied_label_1_SECOND_PAPER}\abs{\partial^\alpha
F(\vec x)}\leq \epsilon \rho^{m-\abs{\alpha}} \text { for all }\vec x\in\text{Ann}_4(\rho)
\text{ and all }\abs{\alpha}\leq m; \end{gather}
\begin{gather}\label{new_def_of_implied_label_2_SECOND_PAPER}\abs{\partial^\alpha S_l(\vec
x)}\leq A \rho^{-\abs{\alpha}} \text { for all }\vec x\in\text{Ann}_4(\rho)
\text{, all }\abs{\alpha}\leq m \text{ and all }1\leq l\leq L; \end{gather}
\begin{multline}\label{new_def_of_implied_label_3_SECOND_PAPER}p(\vec x)=F(\vec x)+S_1(\vec x)Q_1(\vec x)+S_2(\vec
x)Q_2(\vec x)+\dots+S_L(\vec
x)Q_L(\vec x) \\ \text {for all }\vec x\in\text{Ann}_2(\rho)
\text{ such that }\text{dist}(\frac{\vec x}{\abs{\vec x}},\Omega)<\delta.
\end{multline}
\end{definition}

\begin{definition}[the closure of an ideal; see {\cite[Definition 2.12]{FS2}}]\label{the_closure_of_ideal_definition_SECOND_PAPER}Let
$I\lhd\mathcal{P}_0^m(\mathbb{R}^n)$ be an ideal. We define its \textit{implication-closure} (or simply \textit{closure}) $cl(I)$ by $$cl(I):=\{p\in\mathcal{P}_0^m(\mathbb{R}^n)|p\text{
is implied by }I\}.$$ We say that $I$ is \textit{closed} if $I=cl(I)$.\end{definition}

\begin{corollary}[sets with no allowed directions; see {\cite[Corollary 2.17]{FS2}}]\label{cor_no_allowed_directions_SECOND_PAPER}The
only closed ideal $I\lhd\mathcal{P}_0^m(\mathbb{R}^n)$ such that $Allow(I)=\emptyset$
is $I=\mathcal{P}_0^m(\mathbb{R}^n)$. Moreover, if $E\subset\mathbb{R}^n$
is a closed subset containing the origin such that $Allow(I^m(E))=\emptyset$,
then the origin is an isolated
point of $E$ and $I^m(E)=\mathcal{P}_0^m(\mathbb{R}^n)$.\end{corollary}

\begin{theorem}[see
{\cite[Theorem 2.18]{FS2}}]\label{main_theorem_on_necessary_condition_SECOND_PAPER}Fix $m,n\in\mathbb{N}$
and let $E\subset\mathbb{R}^n$ be a closed subset containing the origin.
Then, $I^m(E)\lhd\mathcal{P}_0^{m}(\bR^n)$ is a closed ideal.\end{theorem}

\begin{definition}[order of vanishing of functions; see {\cite[Definition 2.3]{FS2}}]\label{order-of-vanishing-def_SECOND_PAPER}Let
$n,m\in\mathbb{N}$, and let $f\in C^m(\bR^n)$. \textit{The order of vanishing
(at the origin)} of $f$ is said to be

\begin{itemize}
\item the minimal $1\leq m'\leq m$ such that $J^{m'}(f)\neq0$, if
such $m'$ exists and $f(\vec 0)=0$;
\item more than $m$, if $J^m(f)=0$;
\item 0, if $f(\vec 0)\neq 0$. 
\end{itemize}
\textit{The order of vanishing of a jet} $p\in \mathcal{P}^m(\bR^n)$
is the order of vanishing of the polynomial $p$. 
  
\end{definition}

\begin{definition}[lowest degree homogenous part; see {\cite[Definition 3.1]{FS2}}]Let $p\in\mathcal{P}_0^{m}(\bR^n)$
be a non-zero jet. We may always uniquely write $p=p_k+q$,
with $p_k$ a homogenous polynomial
of degree $k$, where $k$ is the order of vanishing
of $p$ (as $p$ is not the zero jet we have $1\leq k\leq m$ -- see Definition
\ref{order-of-vanishing-def_SECOND_PAPER}), and $q$ is a (possibly zero) polynomial of
order of vanishing more than 
$k$. We call $p_k$ \textit{the lowest degree homogenous part of} $p$.\end{definition}

\begin{lemma}[see {\cite[Lemma 3.2]{FS2}}]\label{equiv_conditions_for_forbidden_directions_lemma_SECOND_PAPER}
Let $p\in\mathcal{P}_0^m(\bR^n)$ be a non-zero jet and let $\omega\in
S^{n-1}$ be a direction. Let $p_k$ be the lowest degree
homogenous part of $p$. Then, the following are equivalent:
\begin{gather}\label{equiv_cond_two_for_forbidden_SECOND_PAPER}
  \text{ There exist a cone }\Gamma(\omega,\delta,r)\text{ and }c>0\text{
such that }\abs{p(\vec x)}> c\cdot\abs{\vec
{x}}^k\text { for all }\vec x\in\Gamma(\omega,\delta,r);
\end{gather}
\begin{gather}\label{equiv_cond_three_for_forbidden_SECOND_PAPER}
  p_k(\omega)\neq 0;
\end{gather}

\end{lemma}

Recall that  $\langle p_1,\dots p_t\rangle_m$ stands
for the ideal in $\mathcal{P}_0^m(\mathbb{R}^n)$ generated by the jets $p_1,\dots,p_k$.

\begin{corollary}[see {\cite[Corollary 3.3]{FS2}}]\label{old_new_cor_on_how_to_calc_allow_with_inclusion_SECOND_PAPER}Let
$I=\langle
p_1,p_2,\dots p_t\rangle_m\lhd\mathcal{P}^{m}(\bR^n)$ be an ideal. Then,
$$Allow(I)\subset\bigcap\limits_{i=1}^t\{\text{the zero set in } S^{n-1}\text{
of the lowest degree homogenous part of }p_i\}.$$Moreover, if $
p_1,p_2,\dots p_t$ are all homogenous, then $$Allow(I)=\bigcap\limits_{i=1}^t\{\text{the
zero set in } S^{n-1}\text{
of }p_i\}.$$\end{corollary}

\begin{definition}[Negligible or Whitney-negligible functions; see {\cite[Lemma 3.8]{FS2}}]Let
$U\subset\mathbb{R}^n$ be open and let $\Omega\subset S^{n-1}$. A function
$F\in C^m(U)$
is called \textit{(m--)negligible for $\Omega$}, or \textit{Whitney--(m--)negligible for $\Omega$}, if it satisfies the following two equivalent conditions (NEG) and (W-NEG):

(NEG) For all $\epsilon>0$ there
exist  $\delta>0$ and $r>0$ such that the following hold:
\begin{gather}\label{negligible_fun_condition_0_SECOND_PAPER}\Gamma(\Omega,\delta,r)\subset
U;\end{gather}
\begin{gather}\label{negligible_fun_condition_a_SECOND_PAPER}\abs{\partial^\alpha
F(\vec x)}\leq\epsilon\abs{\vec x}^{m-\abs{\alpha}}\text { for all
}\vec x\in\Gamma(\Omega,\delta,r)\text{ and all }\abs{\alpha}\leq
m;\end{gather}
\begin{multline}\label{negligible_fun_condition_b_SECOND_PAPER}\abs{\partial^\alpha F(\vec
x)-\sum\limits_{\abs{\gamma}\leq m-\abs{\alpha}}\frac{1}{\gamma!}\partial^{\alpha+\gamma}
F(\vec y)\cdot(\vec x-\vec y)^{\gamma}}\leq \epsilon\abs{\vec x-\vec
y}^{m-\abs{\alpha}} \\ \text { for all }\vec x,\vec y\in\Gamma(\Omega,\delta,r)
\text{ distinct, and all }\abs{\alpha}\leq m.\end{multline}

(W-NEG) For all $\epsilon>0$
there
exist  $\delta>0$, $r>0$ and $F_\epsilon\in C^m(\mathbb{R}^n\setminus\{\vec0\})$
such that the following hold:
\begin{gather}\label{whitney_negligible_fun_condition_0_SECOND_PAPER}\Gamma(\Omega,\delta,r)\subset
U;\end{gather}
\begin{gather}\label{whitney_negligible_fun_condition_1_SECOND_PAPER}\abs{\partial^\alpha
F_\epsilon(\vec x)}\leq\epsilon\abs{\vec x}^{m-\abs{\alpha}}\text { for all
}\vec x\in\mathbb{R}^n\setminus\{\vec0\}\text{ and all }\abs{\alpha}\leq
m;\end{gather}
\begin{gather}\label{whitney_negligible_fun_condition_2_SECOND_PAPER}F_\epsilon(\vec x)=F(\vec
x)\text { for all
}\vec x\in\Gamma(\Omega,\delta,r).\end{gather}  
\end{definition}

\begin{example}[see {\cite[Example 3.9]{FS2}}]\label{amazing_exaple_for_negligible_part1_SECOND_PAPER}Set
$n=3$, $m=2$
and $(x,y,z)$ a standard coordinate system on $\mathbb{R}^3$. Let $U=\mathbb{R}^3\setminus\{z=0\}$,
$\Omega=\{(0,0,\pm1)\}\subset
S^2$ and let $F(x,y,z)=\frac{y^3}{z}\in C^2(U)$.
Then, $F$ is negligible for $\Omega\cap D(\omega,10^{-3})$, for any $\omega\in\Omega$.\end{example}

\begin{definition}[strong directional implication; see {\cite[Definition 3.11]{FS2}}]\label{implied_polynomial_definition_SECOND_PAPER}Let
$I\lhd\mathcal{P}_0^{m}(\bR^n)$
be an ideal and $p\in\mathcal{P}_0^{m}(\bR^n)$ be some polynomial. We
say that\textit{ $I$ strongly implies $p$ in the direction $\omega\in S^{n-1}$
}if there exist $\delta_\omega>0$, $r_\omega>0$, polynomials $Q_1,\dots ,Q_L\in
I$,
functions $ S_1,\dots ,S_L\in C^m(\Gamma(\omega,\delta_\omega,r_\omega))$,
 positive
constants $C_1,\dots ,C_L>0$ and a function $F\in C^m(\Gamma(\omega,\delta_\omega,r_\omega))$
 such that the following hold:
\begin{gather}\label{first_cond_of_implied_poly_in_direction_SECOND_PAPER}F\text{
is negligible for }Allow(I)\cap D(\omega,\delta_\omega);\end{gather}
\begin{gather}\label{second_cond_of_implied_poly_in_direction_SECOND_PAPER}\abs{\partial
^\alpha S_{l}(\vec x)}\leq C_{l}\abs{\vec x}^{-\abs{\alpha}}\text{ for all
}\abs{\alpha}\leq
m, \text{ all }1\leq l \leq L\text{ and all }\vec x\in\Gamma(\omega,\delta_\omega,r_\omega);\end{gather}
\begin{gather}\label{third_cond_of_implied_poly_in_direction_SECOND_PAPER}p(\vec x)=\sum\limits_{l=1}^{L}S_l(\vec
x)\cdot Q_l(\vec x)+F(\vec x)\text { for all }\vec x\in\Gamma(\omega,\delta_\omega,r_\omega).\end{gather}\end{definition}

\begin{corollary}[see {\cite[Corollary 3.18]{FS2}}]\label{cor_strong_directional_implication_imply_implication_SECOND_PAPER}Let
$I\lhd\mathcal{P}_0^{m}(\bR^n)$
be an ideal and $p\in\mathcal{P}_0^{m}(\bR^n)$ be some polynomial. If $I$
strongly implies $p$ in the direction $\omega$ for any $\omega\in Allow(I)$
(as in Definition \ref{implied_polynomial_definition_SECOND_PAPER}),
then $I$ implies $p$ (as in Definition \ref{new_def_implied_jet_SECOND_PAPER}).\end{corollary}

\begin{example}[see {\cite[Example 3.19]{FS2}}]\label{amazing_exaple_for_negligible_part2_SECOND_PAPER}Set $n=3$, $m=2$
and $(x,y,z)$ a standard coordinate system on $\mathbb{R}^3$. Let $I\lhd\mathcal{P}_0^2(\mathbb{R}^3)$
be such that $x^2,y^2-xz\in I$. Then, $xy\in cl(I)$. In particular $\langle
x^2,y^2-xz\rangle_2$
is not closed.\end{example}

\section{Some easy examples and useful lemmas}\label{section_easy_examples}

In this section we collect some simple examples and lemmas that will be useful in later sections. Many additional examples will be presented in later sections as well. 

\begin{example}\label{a_union_of_coordinates}Fix $m,n\in\mathbb{N}$ and let $\emptyset\neq A\subset\{1,\dots,
n\}$. Then, $\langle \{x_i\}_{i\in A}
\rangle_m=I^m(\{\sum\limits_{i\in A}x_i^2=0\})$.

Indeed,
by re-indexing we may assume without loss of generality that  $A=\{1,2,\dots,k\}$
for some $1\leq k\leq n$. For any $1\leq i\leq k$ we have the coordinate
function $x_i\in C^m(\bR^n)$ that vanishes on the set $\{\sum\limits_{i=1}^{k}x_i^2=0\}$,
and so $I^m(\{\sum\limits_{i=1}^k x_i^2=0\})\supset\langle \{x_i\}_{i=1}^k
\rangle_m$. Assume $p\in I^m(\{\sum\limits_{i=1}^{k}x_i^2=0\})$. There
exists $f\in C^m(\bR^n)$ such that $f|_{\{\sum\limits_{i=1}^{k}x_i^2=0\}}=0$
and $J^m(f)=p$. We write $$p=p_0(x_{k+1},x_{k+2},\dots,x_n)+\sum\limits_{i=1}^k
x_i\cdot p_i(x_1,x_2,\dots,x_n),$$ for some polynomial $p_0$ in $n-k$ variables
of degree at most $m$, and some polynomials $p_1,p_2,\dots,p_k$ in $n$ variables.
We have for each $n$-$k$-tuple $(x_{k+1},x_{k+2},\dots,x_n)$: $$0=f|_{\{\sum\limits_{i=1}^{k}x_i^2=0\}}=p_0(x_{k+1},x_{k+2},\dots,x_n)+o((\sum\limits_{i=k+1}^n
x_i^2)^{m/2}),$$ which is only possible if $p_0\equiv0$, i.e. $p\in\langle
\{x_i\}_{i=1}^k
\rangle_m$.

\end{example}
\begin{lemma}\label{lemma-on-cooord-change}Fix $n\in\mathbb{N}$, and
let $\mathbb{N}\ni k\leq n$. Assume $p_1,\dots,p_k$ are polynomials in $n$
variables, all vanishing at the origin. Moreover, assume that the $k\times
n$ Jacobian matrix -- the matrix whose $(i,j)$ entire is $\frac{\partial
p_i}{\partial x_j}(\vec0)$ -- has rank $k$.
Then, for any $m\in\mathbb{N}$: $$\langle p_1,\dots, p_k  \rangle_m=I^m(\{\sum\limits_{i=1}^k
p_i^2=0\}).$$

\end{lemma}

\begin{proof}By renaming the coordinates and using basic linear algebra,
we may
assume the Jacobian of the map $(x_1,\dots,x_n)\to(p_1(x_1,\dots,x_n),\dots
,p_k(x_1,\dots,x_n),x_{k+1},\dots ,x_n)$ is invertible at the origin. Then,
by the inverse function theorem, this is an analytic diffeomorphism, and
in particular a $C^m$-diffeomorphism, of a neighborhood of the origin. In
the
new coordinate system $(p_1,\dots, p_k,x_{k+1},\dots, x_n)$ we are reduced
to Example \ref{a_union_of_coordinates}.\end{proof}

\begin{corollary}\label{cor-on-prime-with-a-linear-part}Fix $n\in\mathbb{N}$,
and let $p$ be a polynomial in $n$ variables. If the order of vanishing of
$p$ is 1 (i.e., its constant term is zero and its linear term is non-zero),
then $\langle p \rangle_m=I^m(\{p=0\})$ for any $m\in\mathbb{N}$.\end{corollary}
\begin{proof}This is immediate from Lemma \ref{lemma-on-cooord-change} and
from the fact that $\{p=0\}=\{p^{2}=0\}$.\end{proof}

\begin{example}[straight lines in the plane]\label{claim-on-lines}Let $m,k\in\mathbb{N}$. Then,  \begin{multline}\label{example_on_lines_random_label}\langle
\prod\limits_{i=1}^{k}(a_ix+b_iy)\rangle_{m} = I^m(\{\prod\limits_{i=1}^{k}(a_ix+b_iy)=0\})\text{ for any }\\(a_1,b_1),\dots,(a_k,b_k)\in \mathbb{R}^2\setminus\{(0,0)\}\text{ such that }a_i\cdot b_{i'}\neq a_{i'}\cdot b_i\text{ for any }1\leq i< i'\leq k.\end{multline}

We prove (\ref{example_on_lines_random_label}) by induction on $k$. The base case ($k=1$) is a special case of Corollary \ref{cor-on-prime-with-a-linear-part}. 

We now assume that (\ref{example_on_lines_random_label}) holds for up to $k$, and assume we are given $(a_1,b_1),\dots,(a_k,b_k),(a_{k+1},b_{k+1})\in
\bR^2\setminus\{(0,0)\}$ such that for any $1\leq i< i'\leq k$ we have $a_i\cdot b_{i'}\neq a_{i'}\cdot
b_i$. The function $
\prod\limits_{i=1}^{k+1}(a_ix+b_iy) \in C^m(\mathbb{R}^2)$ vanishes on $\{\prod\limits_{i=1}^{k+1}(a_ix+b_iy)=0\}$, so the inclusion $\subset$ in (\ref{example_on_lines_random_label}) holds. It is left to show the other inclusion $\supset$ also holds. By a linear coordinate change we may assume without loss of
generality that $(a_{k+1},b_{k+1})=(1,0)$, and then by assumption $b_i\neq
0$ for all $i=1,\dots,k$. Let $p\in I^m(\{\prod\limits_{i=1}^{k+1}(a_ix+b_iy)=0\})$.
There exists
$f\in C^m(\bR^2)$ such that $f|_{\{\prod\limits_{i=1}^{k+1}(a_ix+b_iy)=0\}}=0$
and $J^m(f)=p$. In particular $f|_{\{\prod\limits_{i=1}^{k}(a_ix+b_iy)=0\}}=0$,
and so $p\in I^m(\{\prod\limits_{i=1}^{k}(a_ix+b_iy)=0\})$. By induction
hypothesis we have $p\in \langle\prod\limits_{i=1}^{k}(a_ix+b_iy)\rangle_m$.
If $k>m$ then $p=0$ and so $p\in \langle x\cdot\prod\limits_{i=1}^{k}(a_ix+b_iy)\rangle_{m}$,
i.e., (\ref{example_on_lines_random_label}) holds for $k+1$. Otherwise, there exists a polynomial $q(x,y)$ of degree at most $m$-$k$,
such that $p=q(x,y)\cdot\prod\limits_{i=1}^{k}(a_ix+b_iy)$. We thus
have $$f(x,y)=q(x,y)\cdot\prod\limits_{i=1}^{k}(a_ix+b_iy)+o((x^2+y^2)^{m/2}).$$
As $f|_{\{a_{k+1}x+b_{k+1}y=0\}}=f|_{\{x=0\}}=0$, in particular for any $y$
we have
$$0=f(0,y)=q(0,y)\cdot(\prod\limits_{i=1}^{k}b_i)\cdot y^k+o(|y|^m).$$ Recalling
that $q$ is of degree at most $m-k$, and that $b_i\neq
0$ for all $i=1,\dots,k$, we conclude that $q(0,y)$ is identically zero.
This implies that $q$ is divisible by $x$ and so $p\in \langle x\cdot\prod\limits_{i=1}^{k}(a_ix+b_iy)\rangle_{m}$,
i.e., (\ref{example_on_lines_random_label}) holds for $k+1$.
\end{example}

\begin{example}[in the plane]\label{example-xx-xy}Fix $m\geq2$. Then, $\langle
x^2,xy\rangle_m=I^m(\{x^2=y^{2m-1}\}\cap\{x\geq0\}).$

Indeed, first note that the functions $x^2-y^{2m-1}, (x-|y|^{m-\frac{1}{2}})y\in
C^m(\bR^2)$ both vanish on the set $\{x^2=y^{2m-1}\}\cap\{x\geq0\}$, and
so $\langle x^2,xy\rangle_m\subset
I^m(\{x^2=y^{2m-1}\}\cap\{x\geq0\})$.

It is enough to show that if $f\in C^m(\bR^2)$ is such that $J^m(f)=ax+p(y)$
for some $a\in\bR$ and some polynomial $p$ of degree at most $m$, and moreover
$f$ vanishes on $\{x^2=y^{2m-1}\}\cap\{x\geq0\}$, then $a=0$
and $p$ is the zero polynomial. Let $f$ be as assumed. We may write $p(y)=\sum\limits_{i=1}^{m}p_iy^i$,
for some $p_1,\dots,p_m\in\mathbb{R}$ and $$f(x,y)=ax+p(y)+o((x^2+y^2)^{m/2}).$$ In particular, for any $t>0$ we have 
$$0=f(t^{2m-1},t^2)=at^{2m-1}+\sum\limits_{i=1}^{m}p_it^{2i}+o((t^4+t^{2(2m-1)})^{m/2})=at^{2m-1}+\sum\limits_{i=1}^{m}p_it^{2i}+o(t^{2m}).$$
As all the exponents $2,4,\dots,2m$ are even and $2m-1$ is odd, and moreover
all of these are not greater than $2m$, we conclude that $a=0$ and that $p_1=p_2=\cdots= p_m=0$.
\end{example} 

\begin{example}[a non-radical ideal in the plane]\label{example-xx}Fix $m\geq2$. Then, $$\langle x^2\rangle_m= I^m(\{x(x^2-y^{2m-1})=0\}\cap\{x\geq0\}).$$

Indeed, let $\theta(t)$ be a $C^\infty$ function of one variable supported
in $[-1,
1]$, such that $\theta|_{[-\frac{1}{2},\frac{1}{2}]}=1$, and define \begin{gather}\label{theta_nice_usful_function}\tilde \theta(x,y):=\begin{cases}
  \theta(\frac{x}{y}),  &\text{if } y\neq 0 \\
  0, &\text{if } y=0
\end{cases}.\end{gather}The function
$x(x-|y|^{m-1/2}\cdot \tilde \theta(x,y))\in
C^m(\bR^2)$ vanishes on $\{x(x^2-y^{2m-1})=0\}\cap\{x\geq0\}\cap B(10^{-10})$, and its $m$-jet at the origin is $x^2$,
so $$\langle x^2\rangle_m\subset
I^m(\{x(x^2-y^{2m-1})=0\}\cap\{x\geq0\}\cap B(10^{-10}))=I^m(\{x(x^2-y^{2m-1})=0\}\cap\{x\geq0\})\footnote{\label{footnote_on_function_vanishes_on_a_set}In order to show that the jet $x^2$ belongs to $I^m(E)$ for $E=\{x(x^2-y^{2m-1})=0\}\cap\{x\geq0\}$, we constructed a $C^m$ function with jet $x^2$ that vanishes on $E'$, which is the intersection of $E$ with a small ball about the origin. This showed that $x^2\in I^m(E')$ and since $I^m(E)=I^m(E')$ also $x^2\in I^m(E)$. Later in this paper, when we construct other examples of sets $E$ and calculate their corresponding $I^m(E)$, we will sometimes write that a $C^m$ function vanishes on $E$, when we actually mean that it vanishes on the intersection of $E$ with a small ball about the origin.}.$$ As $$\{x(x^2-y^{2m-1})=0\}\cap\{x\geq0\}\supset \{x^2=y^{2m-1}\}\cap\{x\geq0\},$$
we have, by Example \ref{example-xx-xy}, that $I^m(\{x(x^2-y^{2m-1})=0\}\cap\{x\geq0\})\subset\langle
x^2,xy\rangle_m$. We are thus left to show that if $f\in C^m(\bR^2)$ is such
that \begin{gather}\label{rand_label_for_non_radical_example3}f(x,y)=x\cdot p(y)+o((x^2+y^2)^{m/2})\end{gather} for some polynomial $p$ of degree
at most $m-1$ and $f|_{\{x(x^2-y^{2m-1})=0\}\cap\{x\geq0\}}=0$, then $p\equiv0$. 

Let $f$ be as assumed. In particular, for any $t>0$ we have $f(0,t)=f(t^{m-1/2},t)=0$. Thus, by Rolle's Theorem \begin{gather}\label{rand_label_for_non_radical_example}\text{ there exists }0<\delta(t)<t^{m-1/2}\text{ such that }f_{x}(\delta(t),t)=0.\end{gather}We also have by (\ref{rand_label_for_non_radical_example3}) that\begin{gather}\label{rand_label_for_non_radical_example2}f_{x}(x,y)=p(y)+o(x^2+y^2)^{\frac{m-1}{2}}.\end{gather}
Combining (\ref{rand_label_for_non_radical_example}) and (\ref{rand_label_for_non_radical_example2}), on the set $\{(\delta(t),t)\}_{t>0}$
we have $$0=f_{x}(\delta(t),t)=p(t)+o(\delta(t)^2+t^2)^{\frac{m-1}{2}}=p(t)+o(t^{m-1}),$$
which
implies $p\equiv 0$ as $p$ is of degree
at most $m-1$.\end{example}

\begin{lemma}\label{lemma_induction_adding_a_coordinate_jet}Let $E'\subset\mathbb{R}^{n-1}$ be a closed subset containing the origin,
and define $$E:=\{(x_1,\dots x_n)\in\mathbb{R}^n|(x_1,\dots x_{n-1})\in E',x_{n}=0\}.$$
If $I^m(E')=\langle p_1,\dots p_l\rangle_m\lhd \mathcal{P}_0^{m}(\mathbb{R}^{n-1})$ for some
$p_1,\dots, p_l\in\mathbb{R}[x_1,\dots,x_{n-1}]$, then $$I^m(E)=\langle
p_1,\dots p_l,x_n\rangle_m\lhd \mathcal{P}_0^m(\mathbb{R}^n).$$ \end{lemma}
\begin{proof}If $f\in C^m(\mathbb{R}^n)$ vanishes on $E$, then $f|_{\bR^{n-1}}\in
C^m(\mathbb{R}^{n-1})$ vanishes on $E'$, and so there there exists a projection
map $T:I^m(E)\to I^m(E')$ explicitly given by $$p(x_1,\dots,x_{n-1}, x_n)\mapsto
p(x_1,\dots,x_{n-1}, 0).$$ In particular  $T(I^m(E))\subset\langle p_1,\dots
p_l\rangle_m$. Also note that any $f\in C^{m}(\mathbb{R}^{n-1})$ that vanishes on
$E'$ can be regarded as a $C^m$ function on $\mathbb{R}^n$ (a function that does
not depend on $x_n$) and moreover this function vanishes on $E$. Hence, the
map $T$ is onto. We conclude that \begin{gather}\label{rand_label_some_induction_step_lemma1}T(I^m(E))=\langle p_1,\dots p_l\rangle_m.\end{gather} Note that
the coordinate function $x_n$ is a $C^m$ function on $\bR^n$ that vanishes
on $E$, and so  \begin{gather}\label{rand_label_some_induction_step_lemma2}x_n\in I^m(E).\end{gather} Now (\ref{rand_label_some_induction_step_lemma1}) and (\ref{rand_label_some_induction_step_lemma2}) imply that $I^m(E)=\langle
p_1,\dots p_l,x_n\rangle_m$. \end{proof}

\section{The easy cases: $n=1$ or $m=1$}\label{section_m_or_n_is_1}

\begin{proposition}\label{prop_n_is_one} Fix $m\in\mathbb{N}$
and let $I\lhd\mathcal{P}_0^m(\mathbb{R})$
be an ideal. Then, the following are equivalent:
\begin{itemize}
\item $I$ is closed.
\item There exists a closed set $\vec 0 \in E\subset\mathbb{R}$ such that
$I=I^m(E)$.
\item There exists a closed semi-algebraic set $\vec 0 \in E\subset\mathbb{R}$
such that $I=I^m(E)$.
\end{itemize}\end{proposition}

\begin{proof} 

We will show that the only closed
ideals in $\mathcal{P}_0^m(\mathbb{R})$ are $\{\vec 0\}=I^{m}(\mathbb{R})$ and $\mathcal{P}_0^m(\mathbb{R})=I^m(\{\vec 0\})$.
As each of these arises as $I^m(E)$ for some semi-algebraic set $E$, the
proposition holds by Theorem \ref{main_theorem_on_necessary_condition_SECOND_PAPER}. Indeed,
let $\{\mathcal{P}_0^m(\mathbb{R}^n),\{\vec 0\}\}\not\ni
\tilde I\lhd \mathcal{P}_0^m(\bR)$ be an ideal. As $\tilde I\neq\{\vec 0\}$,  $\tilde
I$ contains a non-zero jet.
The lowest degree homogenous part of this jet is a 
non-constant
homogenous polynomial in one variable, and so it does not vanish at $\pm1$
(recall (\ref{equiv_cond_three_for_forbidden_SECOND_PAPER}) and Definition \ref{definition_allowed_and_forbidden_directions_SECOND_PAPER}).
We conclude that $Allow(\tilde
I)=\emptyset$, and so $\tilde I$ is not closed by Corollary \ref{cor_no_allowed_directions_SECOND_PAPER}.\end{proof}

\begin{proposition}\label{prop_m_is_one} Fix $n\in\mathbb{N}$
and let $I\lhd\mathcal{P}_0^1(\mathbb{R}^n)$
be an ideal. Then, the following are equivalent:
\begin{itemize}
\item $I$ is closed.
\item There exists a closed set $\vec 0 \in E\subset\mathbb{R}^n$ such that
$I=I^1(E)$.
\item There exists a closed semi-algebraic set $\vec 0 \in E\subset\mathbb{R}^n$
such that $I=I^1(E)$.
\end{itemize}\end{proposition}

\begin{proof} 
Let $I\lhd \mathcal{P}_0^1(\mathbb R ^n)$. If $I=\{\vec 0\}$
then $I=I^1(\mathbb R^n)$. Otherwise, fixing coordinate functions $x_1,\dots,x_n$
we have that $I$ is a linear subspace of $\text{span}_{\mathbb R}\{x_1,\dots,x_n\}$.
Let $\{p_1,\dots,p_k\}$ be a basis of $I$. Then, the $k\times n$ Jacobian
matrix whose entries are $\frac{\partial
p_i}{\partial x_j}(\vec0)$ is of rank $k$, and by Lemma \ref{lemma-on-cooord-change} we have that
$I=I^1(\{\sum\limits_{i=1}^k
p_i^2=0\})$. We conclude that any ideal in $\mathcal{P}_0^1(\mathbb{R}^n)$
is of the form $I^1(E)$ for some semi-algebraic $E$. In particular, by Theorem
\ref{main_theorem_on_necessary_condition_SECOND_PAPER} any ideal in $\mathcal{P}_0^1(\mathbb{R}^n)$
is closed.\end{proof}

\section{The $C^2(\mathbb{R}^2)$ case}\label{section_2_2}
\begin{proposition}\label{prop_C2R2_case_new_version} Let $I\lhd\mathcal{P}_0^2(\mathbb{R}^{2})$
be an ideal. Then, the following are equivalent:
\begin{itemize}
\item $I$ is closed.
\item There exists a closed set $\vec 0 \in E\subset\mathbb{R}^2$ such that
$I=I^2(E)$.
\item There exists a closed semi-algebraic set $\vec 0 \in E\subset\mathbb{R}^2$
such that $I=I^2(E)$.
\end{itemize}\end{proposition}
\begin{proof}
By Theorem \ref{main_theorem_on_necessary_condition_SECOND_PAPER} it is enough to show that any closed ideal $I\lhd\mathcal{P}_0^2(\mathbb{R}^2)$ is, up to a $C^2$ semi-algebraic coordinate change around the origin, one of the following\footnote{In remark \ref{remark_on_distinctness_for_2_2_case} we will see that the ideals in this list are pairwise distinct even if $C^2$ coordinate changes around the origin are allowed. Similar remarks will follow Lemmas \ref{lemma_classify_closed_principal_in_C3R2_case_new_version}, \ref{lemma_classify_closed_non_principal_in_C3R2_case_new_version},
\ref{lemma_C2R3_case__with_jet_of_oov_1}, \ref{lemma_classify_closed_ideals_in_C2R3_1_dim_allow},\ref{lemma_classify_closed_ideals_in_C2R3_2_dim_allow}
and \ref{lemma_classify_closed_ideals_in_C2R3_3_dim_allow}.}:

\begin{gather}\label{model_for_2_2_1}\{\vec 0\}=I^2(\mathbb{R}^2)\end{gather}
\begin{gather}\label{model_for_2_2_2}\mathcal{P}_0^2(\mathbb{R}^2)=I^2(\{\vec0\})\end{gather}
\begin{gather}\label{model_for_2_2_3}\langle x\rangle_2=I^2(\{x=0\}) \text{  -- see Example \ref{a_union_of_coordinates}} \end{gather}
\begin{gather}\label{model_for_2_2_4}\langle x^2\rangle_2= I^2(\{x(x^2-y^{3})=0\}\cap\{x\geq0\}) \text{ -- see Example \ref{example-xx}}\end{gather}
\begin{gather}\label{model_for_2_2_5}\langle xy\rangle_2= I^2(\{xy=0\})
\text{ -- see Example \ref{claim-on-lines}}
\end{gather}
\begin{gather}\label{model_for_2_2_6}\langle x^{2},xy\rangle_2= I^2(\{x^2-y^{3}=0\}\cap\{x\geq0\})
\text{ -- see Example \ref{example-xx-xy}}
\end{gather}

\

Let us start: let $I\lhd \mathcal{P}_0^2(\mathbb{R}^2)$ be a closed ideal.

\

\textbf{A. Assume $I$ is principal.} In that case $I=\langle p \rangle_2$ for some jet $p$. If $p=0$ then we are in case (\ref{model_for_2_2_1}). If $p$ is of order of vanishing 1 then by applying a $C^2$
semi-algebraic coordinate change around the origin we may assume $p=x$ and we are in case (\ref{model_for_2_2_3}). Otherwise, $p$ is a non-zero quadratic form, so by applying a linear coordinate change and possibly replacing $p$ by $-p$ we may assume $p(x,y)\in\{x^2,x^2-y^2,x^2+y^2\}$. If $p(x,y)=x^2$ we are in case (\ref{model_for_2_2_4}). If $p(x,y)=x^2-y^2=(x+y)(x-y)$ then by applying a coordinate change $(u,v)=(x+y,x-y)$ we get $p(u,v)=uv$ and we are in case (\ref{model_for_2_2_5}). Finally, if $p(x,y)=x^2+y^2$ then $Allow(I)=\emptyset$, which is a contradiction to $I$ being closed by Corollary \ref{cor_no_allowed_directions_SECOND_PAPER}.

\

\textbf{B. Assume $I$ in not principal.} We analyze two sub-cases:

\

\textbf{B.1. Assume $I$ contains a jet of order of vanishing 1.} In that case by applying a linear coordinate change we may assume $p(x,y)=x+q(x,y)\in I$ for some quadratic form $q$. As $I$ is an ideal this implies that also $x^2,xy\in I$ so there exists some $a\in \mathbb{R}$ such that $x+ay^2\in I$. Applying the coordinate change $(X,Y)=(x+ay^2,y)$ we get $X,X^2,XY\in I$. If there exists $AY+BY^2\in I$ for some $A,B\in\mathbb{R}$ and $A\neq 0$ then $I=\mathcal{P}_0^2(\mathbb{R}^2)$ and we are in case (\ref{model_for_2_2_2}). If no such jet exists but there exists $BY^2\in I$ with $B\neq0$ then we have both $X,Y^2\in I$, so  $Allow(I)=\emptyset$,
which is a contradiction to $I$ being closed by
Corollary \ref{cor_no_allowed_directions_SECOND_PAPER}. If there does not exist $AY+BY^2\in I$ with $A^2+B^2\neq 0$ then $I=\langle X\rangle_2$, which is a contradiction to $I$ being not principal.

\

\textbf{B.2. Assume $I$ contains no jet of order of vanishing 1.} In that case $I$ is a linear subspace of $\text{span}_{\mathbb{R}}\{x^2,xy,y^2\}$.

\

\textbf{B.2.1. $\dim I\in\{0,1\}$. }In that case $I$ is principal, which is a contradiction.

\

\textbf{B.2.2. $\dim I = 3$. }In that case $x^2+y^2\in I$. Then $Allow(I)=\emptyset$, which is a contradiction to $I$ being closed by
Corollary \ref{cor_no_allowed_directions_SECOND_PAPER}.

\

\textbf{B.2.3. $\dim I = 2$. }In that case there exist 2 quadratic forms $p(x,y),q(x,y)$ such that $I=\langle p,q\rangle_2$. By applying a linear coordinate change and possibly replacing $p$ by $-p$ we may assume that $p(x,y)\in\{x^2,x^2-y^2,x^2+y^2\}$.
\begin{itemize}
\item If $p(x,y)=x^2+y^2$ then $Allow(I)=\emptyset$,
which is a contradiction to $I$ being closed by
Corollary \ref{cor_no_allowed_directions_SECOND_PAPER}.
\item If $p(x,y)=x^2$ then $I=\langle x^2,ax^2+bxy+cy^2\rangle_2$ for some
$a,b,c\in\mathbb{R}$
where $b^2+c^2\neq0$. In that case also $bxy+cy^2\in I$. If $b=0$ then $x^2+y^2\in
I$. This implies  $Allow(I)=\emptyset$,
which is a contradiction to $I$ being closed by
Corollary \ref{cor_no_allowed_directions_SECOND_PAPER}. If $c=0$ then $I=\langle
x^2,xy\rangle_2$ and we are in case (\ref{model_for_2_2_6}). Otherwise,
also $xy+\frac{c}{b}y^2\in I$, with $c\neq0$. If $\frac{b}{c}>0$ then we
have that $\frac{b}{c}x^2+xy+\frac{c}{b}y^2\in I$, which is a positive definite
quadratic form. If $\frac{b}{c}<0$ then $-\frac{b}{c}x^2-xy-\frac{c}{b}y^2\in I$, which is
a positive definite
quadratic form. In both cases $I$ contains a positive definite quadratic form and so  $Allow(I)=\emptyset$,
which is a contradiction to $I$ being closed by
Corollary \ref{cor_no_allowed_directions_SECOND_PAPER}. 
\item If $p(x,y)=xy$ then  $I=\langle xy, ax^2+bxy+cy^2\rangle_2$ for some
$a,b,c\in\bR$ where $a^2+c^2\neq0$. In that case also $I=\langle xy, ax^2+cy^2\rangle_2$,
for the same
$a,c\in\bR$. If either
$a=0$ or $c=0$ then either $I=\langle xy, y^2\rangle_2$ or $I=\langle xy,
x^2\rangle_2$, and (by possibly replacing $x\leftrightarrow y$) we are in case (\ref{model_for_2_2_6}). Otherwise, $ac\neq0$ and so $I=\langle xy, x^2+\frac{c}{a}y^2\rangle_2$.  If $\frac{c}{a}>0$ then $x^2+\frac{c}{a}y^2$ is a positive definite quadratic form. If $\frac{c}{a}<0$ then $I=\langle xy, (x+\sqrt{\frac{-c}{a}}y)(x-\sqrt{\frac{-c}{a}}y)\rangle_2$. In both cases $Allow(I)=\emptyset$,
which is a contradiction to $I$ being closed by
Corollary \ref{cor_no_allowed_directions_SECOND_PAPER}.\end{itemize}\end{proof}

\begin{remark}\label{remark_in_C2R2_enough_to_have_allowed_direction}In the proof of Proposition \ref{prop_C2R2_case_new_version} the only way we conclude that an ideal $I\neq\mathcal{P}_0^2(\mathbb{R}^2)$ is not closed is by showing that $Allow(I)=\emptyset$,
which implies it is not closed by
Corollary \ref{cor_no_allowed_directions_SECOND_PAPER}. So we also proved that an ideal $\mathcal{P}_0^2(\mathbb{R}^2)\neq I\lhd\mathcal{P}_0^2(\mathbb{R}^2)$ is closed if and only if $Allow(I)\neq\emptyset$. We will use this fact in the proof of Lemma \ref{lemma_C2R3_case__with_jet_of_oov_1}. A similar statement is not true in general, e.g., the principal ideal $I:=\langle x(x^2+y^2)\rangle_3\lhd\mathcal{P}_0^3(\mathbb{R}^2)$ is not closed by Lemma \ref{lemma_linear_times_definite_in_closed_imply_many_jets}, although $Allow(I)\neq\emptyset$. In $\mathcal{P}_0^2(\mathbb{R}^3)$ we have that $\langle x^2,y^2-xz\rangle_2$ is not closed by Example \ref{amazing_exaple_for_negligible_part2_SECOND_PAPER}, although $Allow(I)\neq\emptyset$. \end{remark}

\begin{remark}\label{remark_on_2_dim_subspaces_in_the_plane}In the last part of the proof of Proposition \ref{prop_C2R2_case_new_version} we also proved that any 2-dimensional subspace of $\text{span}_{\mathbb{R}}\{x^2,xy,y^2\}$ either contains a positive definite form, or has the form, up to a linear coordinate change, $\text{span}_{\mathbb{R}}\{xy,(x+\alpha y)(x-\alpha y)\}$ for some $\alpha\in\mathbb{R}$. Note that if $\alpha\neq0$ then by further scaling $y$ we get get that any 2-dimensional subspace of $\text{span}_{\mathbb{R}}\{x^2,xy,y^2\}$
either contains a positive define form, has the form, up to a linear
coordinate change, $\text{span}_{\mathbb{R}}\{xy,x^{2}\}$ or has the form, up to a linear
coordinate change, $\text{span}_{\mathbb{R}}\{xy,(x+ y)(x- y)\}$. We will use this fact in the proof of Lemma \ref{lemma_classify_closed_non_principal_in_C3R2_case_new_version}. \end{remark}

\begin{remark}\label{remark_on_distinctness_for_2_2_case}The ideals (\ref{model_for_2_2_1})-(\ref{model_for_2_2_6}) are pairwise distinct, i.e., there does not exist a $C^2$ coordinate change around the origin such that a given ideal can take two different forms from this list. To see this first note that (\ref{model_for_2_2_1}) and (\ref{model_for_2_2_2}) are obviously different from the rest. (\ref{model_for_2_2_6}) is distinct as it is the only remaining ideal of dimension 2. (\ref{model_for_2_2_5}) is a also
distinct as it is the only remaining ideal of dimension 1 with two linearly independent allowed directions. Finally, (\ref{model_for_2_2_3}) and (\ref{model_for_2_2_4}) are also different as the first contains a jet of order of vanishing 1, while the latter does not. \end{remark}

\section{Principal ideals in $\mathcal{P}_0^2(\mathbb{R}^n)$ and some other ideals with few generators}\label{section_special_examples_with_m_is_2}

We start with two lemmas that will be very useful for later calculations, and another example we will later need for the proof of Proposition \ref{prop_principal_ideals_in_P2Rn}.

\begin{lemma}\label{lemma_vanishes_on_corners_of_rectangles}Consider $\mathbb{R}^2$
with standard coordinates $x,y$, and let $f\in C^2(\mathbb{R}^2)$. Assume that for some $x_1<x_2$ and some $y_1<y_2$ we have $f(x_i,y_j)=0$ for all $1\leq i,j\leq 2$, i.e., $f$ vanishes on the vertices of a rectangle $A$ whose edges are parallel
to the axes. Then, there exist $x_1\leq x_0\leq x_2$ and $y_1\leq y_0\leq y_2$ such that $f_{xy}(x_0,y_0)=0$, i.e., the mixed partial derivative $\partial_{xy}$ of $f$ vanishes at some point inside the rectangle $A$. \end{lemma}
\begin{proof}This result follows easily from the standard identity:   
$$\int\limits_{x_1}^{x_2}\int\limits_{y_{1}}^{y_2} f_{xy}(x,y)dydx=f(x_2,y_2)-f(x_1,y_2)-f(x_{2},y_{1})+f(x_1,y_1).$$\end{proof}

\begin{lemma}\label{rep_theory_lemma}Let $V\subset\mathbb{R}^n$ be a subspace closed
under all permutations of coordinates, and containing a vector $\vec v\neq\vec0$
whose all coordinates are equal. Then, $V\in\big\{\text{span}_\mathbb{R}\{\vec
v\},\mathbb{R}^n\big\}$.\end{lemma}
\begin{proof} If $V=\text{span}_\mathbb{R}\{\vec
v\}$ we are done. Otherwise, $V$ contains a non-zero vector that has two coordinates that are different. Moreover, since $V$ is closed with respect to coordinate permutations, there exists $V\ni \vec w=(w_1,w_2,\dots,w_n)$ such that $w_1\neq w_2$. Again, as $V$ is closed with respect to coordinate permutations we have both $(w_1,w_2,w_3,w_4,\dots,w_n),(w_2,w_1,w_3,w_4,\dots,w_n)\in V$, and so also their (non-zero) difference $(w_1-w_2,-(w_1-w_2),0,0,\dots,0)\in V$. Scaling this vector we get that $(1,-1,0,0,\dots,0)\in V$ and since $V$ is closed with respect to coordinate
permutations also the $n-1$ linearly independent vectors $(1,-1,0,0,\dots,0),(0,1,-1,0,0,\dots,0),(0,0,1,-1,0,0,\dots,0),\dots,(0,0,\dots,0,1,-1)\in V$.  Recalling that $\vec v \in V$ we get that $\dim V=n$ and so $V=\mathbb{R}^n$ and the lemma holds\footnote{In representation theoretic language we proved that the ($n-1$ dimensional) standard representation of $S_n$ is irreducible.}. \end{proof} 

\begin{example}\label{non-real-radical-example}$\langle
x^2\rangle_2=I^2(\{x^2=y^4\})$.

Indeed, note that $\{x^2=y^4\}=\{y^{2}-x=0\}\cup\{y^{2}+x=0\}$, and so by
Corollary \ref{cor-on-prime-with-a-linear-part}: $$I^2(\{x^2=y^4\})\subset
I^2(\{y^2-x=0\})\cap I^2(\{y^2+x=0\})=\langle y^2-x\rangle_2\cap\langle y^2+x\rangle_2.$$
Let $p(x,y)\in I^2(\{x^2=y^4\})$. There exist two polynomials $$q_i(x,y)=a_ix+b_iy+c_{i}\text{
};\text{
} i=1,2,$$such that as 2-jets we have $$p(x,y)=q_1(x,y)\cdot(y^2-x)=q_2(x,y)\cdot(y^2+x).$$
Observing the coefficients of the $x$ and the $y^2$ monomials, we see that $c_1=c_2=0$ and so $I^2(\{x^2=y^4\})\subset\langle
xy,x^2\rangle_2$. 
Since the function $x^2-y^4$ vanishes on $\{x^2=y^4\}$ we have $I^2(\{x^2=y^4\})\supset\langle
x^2\rangle_2$. Thus, it is only left to show that $xy\notin I^2(\{x^2=y^4\})$.
We show this by proving that if $f\in C^2(\bR^2)$ vanishes on $\{x^2=y^4\}$,
then its second mixed partial derivative $f_{xy}$ vanishes at the origin.
In particular $J^2(f)$ cannot be $xy$. Indeed, for any $\epsilon>0$ we have
$f(\pm\epsilon^4,\pm\epsilon^2)=0$, so by Lemma \ref{lemma_vanishes_on_corners_of_rectangles}
there exist $x_\epsilon,y_\epsilon\in\bR$ such that $\abs{x_\epsilon}\leq\epsilon^4$,
$\abs{y_\epsilon}\leq\epsilon^2$ and $f_{xy}(x_\epsilon,y_\epsilon)=0$. As
$\epsilon$ is arbitrary small and $f_{xy}$ is continuous, we must have $f_{xy}(\vec
0)=0$. We conclude that $\langle
x^2\rangle_2=I^2(\{x^2=y^4\})$.

\end{example}

Before we proceed, let us fix some notation that will be useful in the proofs
of Proposition \ref{prop_principal_ideals_in_P2Rn} and Proposition \ref{prop_ideals_in_P2Rn_generated_by_2_foreign_jets}.

\begin{notation}\label{notation_restrictions_of_sets_and_functions}Let $E\subset\bR^n$
be a closed subset containing the origin. For any subset $J\subset\{1,\dots
n\}$ denote $$E^{J}:=E\cap\big(\bigcap\limits_{k\in\{1,\dots,n\}\setminus
J}\{x_k=0\}\big).$$We consider $E^{J}$ as a subset of $\bR^{|J|}\cong\bR^n\cap\big(\bigcap\limits_{k\in\{1,\dots,n\}\setminus
J}\{x_k=0\}\big)$. If $J=\{i_1,\dots i_{|J|}\}$ we also denote $E^{i_1i_2\dots
i_{|J|}}:=E^J$. Note that $E^{J}$ is a closed subset of $\bR^{|J|}$ containing
the origin. Let $f\in C^2(\bR^n)$ be such that $f|_{E}=0$. For any subset
$J\subset\{1,\dots n\}$ denote $$f^{J}:=f|_{\bigcap\limits_{k\in\{1,\dots,n\}\setminus
J}\{x_k=0\}}.$$ If $J=\{i_1,\dots i_{|J|}\}$ we also denote $f^{i_1i_2\dots
i_{|J|}}:=f^J$. Note that  $f^{J}\in C^2(\bR^{|J|})$ and $f^{J}|_{E^{J}}=0$.
Also, note that $$J^2(f^{J})=J^2(f)|_{\bigcap\limits_{k\in\{1,\dots,n\}\setminus
J}\{x_k=0\}},$$ i.e., there exists a natural map $T^{J}:I^2(E)\to I^2(E^J)$.\end{notation}

\begin{proposition}\label{prop_principal_ideals_in_P2Rn} Fix $n\in\mathbb{N}$ and let $I\lhd\mathcal{P}_0^2(\mathbb{R}^n)$ be a principal ideal. Then, the following are equivalent:
\begin{itemize}
\item $I$ is closed.
\item There exists a closed set $\vec 0 \in E\subset\mathbb{R}^n$ such that $I=I^2(E)$.
\item There exists a closed semi-algebraic set $\vec 0 \in E\subset\mathbb{R}^n$
such that $I=I^2(E)$.
\end{itemize}\end{proposition}

\begin{proof} 

Let $I=\langle p\rangle_2\lhd\mathcal{P}_0^2(\mathbb{R}^n)$ for some generator $p$. By Theorem \ref{main_theorem_on_necessary_condition_SECOND_PAPER} it is enough to prove that if $I$ is closed then there exists a closed semi-algebraic set $\vec 0 \in E\subset\mathbb{R}^n$
such that $I=I^2(E)$. If $p\equiv0$ then $I=I^2(\bR^n)$ and the proposition holds. If the order of vanishing of $p$ is 1, then by Corollary \ref{cor-on-prime-with-a-linear-part},
$I=I^2(\{p=0\})$ and again the proposition holds. We are left to check the case where $p$ is a quadratic form on $\bR^n$. If $p$ is a definite form (either positive
or definite), then $Allow(I)=\emptyset$,
and so $I$ is not closed by
Corollary \ref{cor_no_allowed_directions_SECOND_PAPER}.

\

For the remaining cases, without loss of generality (by possibly applying
a linear coordinate change and replacing the generator $p$ by the generator
$-p$) we may assume that $n\geq3$ (the lower dimensional cases were proven
separately) and that there exists a (pairwise disjoint) partition $$\{1,2\dots,n\}=J_+\cup
J_-\cup J_0$$ with $J_+\neq\emptyset$ and $J_-\cup J_0\neq\emptyset$, such
that $i<j<k$ for all $i\in J_+,j\in J_-,k\in J_0$ (if $J_-=\emptyset$ we only demand $i<k$ and similarly if $J_0=\emptyset$ we only demand $i<j$), and $$p(x_1,\dots x_n)=\sum\limits_{i\in
J_+}x_i^2-\sum\limits_{i\in
J_-}x_i^2.$$ Define $$q(x_{1},\dots x_n):=\sum\limits_{i\in
J_+}x_i^2-\sum\limits_{i\in
J_-}x_i^2-\big(\sum\limits_{i\in
J_0}x_i^2\big)^2.$$ We will show that $I=I^2(\{q=0\})$.

\

Before we start set $E:=\{q=0\}$ and recall Notation \ref{notation_restrictions_of_sets_and_functions}.
Fix some $f\in C^2(\bR^n)$ such that $f|_E=0$. First let us show that \begin{gather}\label{rand_label_for_princ_ideal_in_C2Rn}\frac{\partial f}{\partial x_i}(\vec
0)=\frac{\partial f}{\partial x_j}(\vec
0)=\frac{\partial^2f}{\partial x_i\partial
x_j}(\vec 0)=0\text{ for all }1\leq i\neq j\leq n.\end{gather}To show that (\ref{rand_label_for_princ_ideal_in_C2Rn}) holds, it is enough to show that for any fixed $1\leq i\neq j\leq n$ and any arbitrary small $\epsilon>0$ there exists a point $(x_1,x_2,\dots,x_n)\in E\cap B(\epsilon)$,
such that both $x_i\neq0$ and $x_j\neq0$. Then, as $E$ is symmetric
with respect to $x_i\mapsto -x_i$ and $x_j\mapsto -x_j$, we can apply Rolle's
Theorem to get $\frac{\partial f}{\partial x_i}(\vec
0)=\frac{\partial f}{\partial x_j}(\vec
0)=0$ and apply Lemma \ref{lemma_vanishes_on_corners_of_rectangles} to get
$\frac{\partial^2f}{\partial x_i\partial
x_j}(\vec 0)=0$. We will do this by analyzing 6 possible cases, corresponding
to the 6 possible unordered pairs of sets chosen from $J_+,J_-,J_0$:

\begin{enumerate}[label=(\roman*)]
\item $i\in J_+,j\in J_+$: in this case there exists $k\in J_-\cup J_0$.
\begin{enumerate}
\item If there exists $k\in J_0$ then for any $\epsilon>0$ we have $\big(\epsilon,\epsilon,(2\epsilon^2)^{1/4}\big)\in E^{ijk}$.
\item If there exists $k\in J_-$ then for any $\epsilon>0$ we have $\big(\epsilon,\epsilon,(2\epsilon^2)^{1/2}\big)\in
E^{ijk}$.

\end{enumerate}
\item $i\in J_-,j\in J_-$: in this case there exists $k\in J_+$. Then, for
any $\epsilon>0$ we have $(\sqrt{2}\epsilon,\epsilon,\epsilon)\in
E^{kij}$.

\item $i\in J_0,j\in J_0$: in this case there exists $k\in J_+$. Then, for
any $\epsilon>0$ we have $(2\epsilon^2,\epsilon,\epsilon)\in
E^{kij}$.

\item $i\in J_-,j\in J_0$: in this case there exists $k\in J_+$. Then, for
any $0<\epsilon<1$ we have $(\epsilon,\sqrt{\epsilon^2-\epsilon^4},\epsilon)\in
E^{kij}$.
\item $i\in J_+,j\in J_-$: in this case there exists $\{i,j\}\not\ni k\in\{1,\dots,n\}$.
\begin{enumerate}
\item If there exists $k\in J_0$ then for any $\epsilon>0$ we have $(\epsilon,\sqrt{\epsilon^2-\epsilon^4},\epsilon)\in
E^{ijk}$.
\item If there exists $k\in J_-$ then for any $\epsilon>0$ we have $(\sqrt{2}\epsilon,\epsilon,\epsilon)\in
E^{ijk}$.

\item If there exists $k\in J_+$ then for any $\epsilon>0$ we have $(\sqrt{3}\epsilon,\epsilon,2\epsilon)\in
E^{ikj}$.

\end{enumerate}
\item $i\in J_+,j\in J_0$: in this case $E^{ij}=\{x_i^2=x_j^4\}$, so $J^2(f^{ij})\in
I^2(E^{ij})=\langle x_i^2\rangle_2$ by Example \ref{non-real-radical-example}.
So in particular $\frac{\partial^2f}{\partial
x_i\partial x_j}(\vec 0)=\frac{\partial f}{\partial x_i}(\vec 0)=\frac{\partial
f}{\partial
x_j}(\vec 0)=0$. Note that moreover this shows that $\frac{\partial^2f}{\partial
x_j^2}(\vec 0)=0$. Since $J_+\neq\emptyset$ we in fact showed that $\frac{\partial^2f}{\partial
x_j^2}(\vec 0)=0$ for any $j\in J_0$.
\end{enumerate}   
We thus proved that 
(\ref{rand_label_for_princ_ideal_in_C2Rn})  holds and moreover that $\frac{\partial^2f}{\partial
x_j^2}(\vec 0)=0$ for any $j\in J_0$, and so \begin{gather}\label{rand3_label_for_princ_ideal_in_C2Rn}I^2(\{q=0\})\subset \text{span}_\bR\{x_1^2,\dots,x_{|J_+|+|J_-|}^2\}.\end{gather}Recall that for any subset $J\subset\{1,\dots n\}$ there exists a map $T^{J}:I^2(E)\to
I^2(E^J)$, and note that \begin{gather}\label{rand2_label_for_princ_ideal_in_C2Rn}J^2(q)=\sum\limits_{i\in
J_+}x_i^2-\sum\limits_{i\in
J_-}x_i^2\in I^2(\{q=0\})=I^{2}(E).\end{gather}Thus, $T^{J_{+}}(I^2(E))$ is a linear
subspace of $\text{span}_\bR\{x_1^2,\dots,x_{|J_+|}^2\}$ by (\ref{rand3_label_for_princ_ideal_in_C2Rn}). Moreover, it is
closed under all permutations
of coordinates (this follows immediately from the definition of $E$), and
by (\ref{rand2_label_for_princ_ideal_in_C2Rn}) it contains the non-zero vector $\sum\limits_{i\in
J_+}x_i^2$ whose all coordinates
are equal. Thus, by Lemma \ref{rep_theory_lemma}\begin{gather}\text{either }T^{J_{+}}(I^2(E))=\text{span}_\bR\{\sum\limits_{i\in
J_+}x_i^2\}\text{ or }T^{J_{+}}(I^2(E))=\text{span}_\bR\{x_1^2,\dots,x_{|J_+|}^2\}.\end{gather} If $|J_+|=1$ these are the same. However, if $|J_+|>1$ then $\text{span}_\bR\{x_1^2,\dots,x_{|J_+|}^2\}$ is not
invariant with respect to rotations in $\bR^{|J_+|}$, whereas $E^{J_+}$ is
invariant, which is impossible. We conclude that \begin{gather}\label{rand8_label_for_princ_ideal_in_C2Rn}T^{J_{+}}(I^2(E))=\text{span}_\bR\{\sum\limits_{i\in
J_+}x_i^2\}.\end{gather} If $J_-=\emptyset$ then (\ref{rand3_label_for_princ_ideal_in_C2Rn}) together with (\ref{rand8_label_for_princ_ideal_in_C2Rn}) shows that $\langle p \rangle_2=I^2(\{q=0\})$ and the proposition is proven. Assume $J_-\neq\emptyset$. The same argument that lead to (\ref{rand8_label_for_princ_ideal_in_C2Rn}) shows that\begin{gather}\label{rand9_label_for_princ_ideal_in_C2Rn}T^{J_{-}}(I^2(E))=\text{span}_\bR\{\sum\limits_{i\in
J_-}x_i^2\}.\end{gather}
Combining (\ref{rand2_label_for_princ_ideal_in_C2Rn}),(\ref{rand8_label_for_princ_ideal_in_C2Rn}) and (\ref{rand9_label_for_princ_ideal_in_C2Rn}) we conclude that either $I^2(E)=\text{span}_\bR\{\sum\limits_{i\in
J_+}x_i^2-\sum\limits_{i\in
J_-}x_i^2\}$, or $I^2(E)=\text{span}_\bR\{\sum\limits_{i\in
J_+}x_i^2,\sum\limits_{i\in
J_-}x_i^2\}$. In the first case $I^2(E)=\langle p\rangle_2$ and the proposition is proven. Assume that $I^2(E)=\text{span}_\bR\{\sum\limits_{i\in
J_+}x_i^2,\sum\limits_{i\in
J_-}x_i^2\}$. 
In that case $$\sum\limits_{i\in
J_+}x_i^2+\sum\limits_{i\in
J_-}x_i^2\in T^{J_+\cup J_-}(I^2(E))\subset I^2(E^{J_+\cup J_-}),$$ and so
$Allow(I^{2}(E^{J_+\cup J_-}))=\emptyset$, which implies (by Corollary \ref{cor_no_allowed_directions_SECOND_PAPER})
that the origin is an isolated point of $E^{J_+\cup J_-}\subset\bR^{|J_+|+|J_-|}$.
This is a contradiction of course as the set $$E^{J_+\cup J_-}=\{\sum\limits_{i\in
J_+}x_i^2-\sum\limits_{i\in
J_-}x_i^2=0\}\subset\bR^{|J_+|+ |J_-|}$$ contains a line passing through
the origin whenever both $J_+\neq\emptyset$ and $J_-\neq\emptyset$: if $i\in
J_+$ and $j\in J_-$ then the line $x_i=x_j$ and all other coordinates are
zero is contained in $E^{J_+\cup J_-}$.  \end{proof}

\begin{remark}\label{remark_on_prop_principal_ideals_in_P2Rn}The proof of Proposition \ref{prop_principal_ideals_in_P2Rn} also shows that if $I\lhd\mathcal{P}_0^2(\mathbb{R}^n)$ is a principal ideal generated by a homogenous polynomial of degree 2, then $I$ is closed if and only if $Allow(I)\neq\emptyset$. \end{remark}

Lemmas \ref{calc_lemma_1_on_non_isolated_point}-\ref{calc_lemma_9_on_non_isolated_point}
below are some technical computations that will be used in the proof of Proposition
\ref{prop_ideals_in_P2Rn_generated_by_2_foreign_jets}.

\begin{lemma}\label{calc_lemma_1_on_non_isolated_point}Let $E\subset\bR^3$
be given by
$$E:=\{(x,y,z)\in\bR^3|x^2-y^2-z^4=z^2-(x^2+y^2)^2=0\}.$$
Then, for any $\epsilon>0$ there exists $(x,y,z)\in E\cap B(\epsilon)$ such that $xyz\neq
0$. 
\end{lemma}
\begin{proof}We can rewrite $E=\{(x,y,z)\in\bR^3|x^2-y^2-(x^2+y^2)^4=0;z^2=(x^2+y^2)^{2}\}$.
The first is $g(x,y):=(x-y)(x+y)-(x^2+y^2)^4=0$,
so in the coordinate system $(u,v):=(x-y,x+y)$ we have $\frac{\partial^2g}{\partial
u \partial v}(\vec 0)\neq 0$. By the Morse Lemma (see \cite{P}) there exists a continuous (in fact $C^\infty$) coordinate change around the origin in the $u-v$ plane, which is the $x-y$
plane, $(u,v)\mapsto(\tilde u,\tilde v)$ such that $g(\tilde u,\tilde v)=\tilde
u\tilde v$. So we can rewrite $$E=\{(x(\tilde u,\tilde v),y(\tilde u,\tilde v),z)\in\bR^3|\tilde
u\tilde v=0,z^2=(x(\tilde u,\tilde v)^2+y(\tilde u,\tilde v)^2)^{2}\},$$
and so the origin is not an isolated point of $E$, e.g., for any $\epsilon>0$
$$(x(\tilde u,\tilde v),y(\tilde
u,\tilde v),z)=\big(x(0,\epsilon),y(0,\epsilon),x(0,\epsilon)^2+y(0,\epsilon)^2\big)\in E.$$As the origin is not an isolated point of $E$, it is enough to show that there are finitely many points $(x,y,z)\in E$ such that $xyz=0$. Indeed, if $y=0$ then $(x,y,z)\in
E$ implies  $x^2=z^4$ and $z^2=x^4=z^8$, so $x,z\in\{0,\pm1\}$. Similarly, if $x=0$ then $(x,y,z)\in
E$ implies  $-y^2-z^4=0$ and so $y=z=0$. Finally, if $z=0$ then $(x,y,z)\in
E$ implies
 $x^2+y^2=0$ and so $x=y=0$.\end{proof}

\begin{lemma}\label{calc_lemma_2_on_non_isolated_point}Let $E\subset\bR^3$
be given by
$$E:=\{(x,y,z)\in\bR^3|x^2-(y^2+z^2)^{2}=y^2-(x^2+z^2)^{2}=0\}.$$
Then, for any $\epsilon>0$ there exists $(x,y,z)\in E\cap B(\epsilon)$ such
that $xyz\neq
0$. 
\end{lemma}
\begin{proof}For any $0<\epsilon<1$, $(\epsilon,\epsilon,\sqrt{\epsilon-\epsilon^2})\in
E$.\end{proof}

\begin{lemma}\label{calc_lemma_3_on_non_isolated_point}Let $E\subset\bR^4$
be given by
$$E:=\{(x,y,z,w)\in\bR^4|x^2+y^2-(z^2+w^{2})^{2}=z^{2}-(x^{2}+y^2+w^{2})^{2}=0\}.$$
Then, for any $\epsilon>0$ there exists $(x,y,z,w)\in E\cap B(\epsilon)$ such
that $xyzw\neq
0$. 
\end{lemma}
\begin{proof}Set $t^2=x^2+y^2$ we may rewrite $E$ as $$E=\{(x,y,z,w)\in\bR^4|t^2=x^2+y^2;t^2-(z^2+w^2)^{2}=z^2-(t^{2}+w^{2})^{2}=0\},$$and
the Lemma now easily follows from Lemma \ref{calc_lemma_2_on_non_isolated_point}:
for any $(t,z,w)\in\bR^3$ such that $t^2-(z^2+w^2)^{2}=z^2-(t^{2}+w^{2})^{2}=0$
and $xyt\neq0$, we have $\big(\frac{t}{\sqrt{2}},\frac{t}{\sqrt{2}},z,w\big)\in
E$.
\end{proof}

\begin{lemma}\label{calc_lemma_4_on_non_isolated_point}Let $E\subset\bR^4$
be given by
$$E:=\{(x,y,z,w)\in\bR^3|x^2+y^2-(z^2+w^{2})^{2}=z^{2}-w^2-(x^{2}+y^2)^{2}=0\}.$$
Then, for any $\epsilon>0$ there exists $(x,y,z,w)\in E\cap B(\epsilon)$
such
that $xyzw\neq
0$. 
\end{lemma}
\begin{proof}We can rewrite $$E=\{(x,y,z,w)\in\bR^4|x^2+y^2=(z^2+w^2)^2;z^2-w^{2}-(z^2+w^2)^{4}=0\}.$$
The second condition may be written as $g(z,w):=(z-w)(z+w)-(z^2+w^2)^4=0$,
so in the coordinate system $(u,v):=(z-w,z+w)$ we have $\frac{\partial^2g}{\partial
u \partial v}(\vec 0)\neq 0$. By the Morse Lemma (see \cite{P}) there exists a continuous coordinate change around the origin in the $u-v$ plane, which is the $z-w$
plane, $(u,v)\mapsto(\tilde u,\tilde v)$ such that $g(\tilde u,\tilde v)=\tilde
u\tilde v$. So we can rewrite $$E=\{(x,y,\tilde z(\tilde u,\tilde v),\tilde w(\tilde u,\tilde v))\in\bR^4|\tilde
u\tilde v=0,x^2+y^2=(z(\tilde u,\tilde v)^2+w(\tilde u,\tilde v)^2)^{2}\},$$
and so the origin is not an isolated point of $E$, e.g., for any $\epsilon>0$
$$\gamma(\epsilon):=\big(\frac{z(0,\epsilon)^2+w(0,\epsilon)^2}{\sqrt{2}},\frac{z(0,\epsilon)^2+w(0,\epsilon)^2}{\sqrt{2}},z(0,\epsilon),w(0,\epsilon)\big)\in
E.$$Note that for any $\epsilon>0$ the point $\gamma(\epsilon)$ has coordinates $x$ and $y$ non-zero, so it is enough to show no point $\vec 0\neq(x,y,z,w)\in E\cap B(\frac{1}{2})$ satisfies $zw=0$. Indeed, if $z=0$ then
$(x,y,z,w)\in E$ implies  $-w^2-(x^{2}+y^2)^{2}=0$ and so $x=y=w=0$. Similarly, if $w=0$
then $(x,y,z,w)\in E$ implies  $z^2-z^8=0$ and so $z\in\{0,\pm1\}$. However, $z\notin\{\pm1\}$ as then $(x,y,z,w)\notin B(\frac{1}{2})$. So $z=0$ which we just saw implies that $x=y=w=0$.\end{proof}

\begin{lemma}\label{calc_lemma_5_on_non_isolated_point}Let $E\subset\bR^4$
be given by
$$E:=\{(x,y,z,w)\in\bR^4|x^2-y^2-(z^2+w^{2})^{2}=z^{2}-(x^{2}+y^2+w^{2})^{2}=0\}.$$
Then, for any $\epsilon>0$ there exists $(x,y,z,w)\in E\cap B(\epsilon)$
such
that $xyzw\neq
0$.
\end{lemma}
\begin{proof}Note that when $x=\sqrt{2}y$ we get that $(x,y,z,w)\in E$ is
equivalent to the conditions $$z^2=(3y^2+w^2)^2\text{ and }y^2-((3y^2+w^2)^2+w^2)^2=0.$$
The second condition is $$g(y,w):=y^2-((3y^2+w^2)^2+w^2)^2=0.$$
By the implicit function theorem the set $\gamma:=\{Y-((3Y+W)^2+W)^2=0\}$ is an analytic
curve around the origin, and $Y$ is an analytic function of $W$ in an neighborhood
of the origin. Moreover, $Y$ is always positive in such a punctured neighborhood, so
there exists arbitrary small $Y$ and $W$ both positive such that $(Y,W)\in \gamma$. This implies that there exists arbitrary small positive $y'$ and
$w'$ such that $g(y',w')=0$. Finally, set $x'=\sqrt{2}y'$ and $z'=3y'^{2}+w'^2$
to get $(x',y',z',w')\in E$ arbitrary close to the origin with $x'y'z'w'\neq
0$.\end{proof} 

\begin{lemma}\label{calc_lemma_6_on_non_isolated_point}Let $E\subset\bR^4$
be given by
$$E:=\{(x,y,z,w)\in\bR^4|x^2-(y^{2}+z^2+w^{2})^{2}=y^{2}-(x^{2}+z^2+w^{2})^{2}=0\}.$$
Then, for any $\epsilon>0$ there exists $(x,y,z,w)\in E\cap B(\epsilon)$
such
that $xyzw\neq
0$. 
\end{lemma}
\begin{proof}Set $t^2=z^2+w^2$ we may rewrite $E$ as $$E=\{(x,y,z,w)\in\bR^4|t^2=z^2+w^2;x^2-(y^{2}+t^{2})^{2}=y^{2}-(x^{2}+t^{2})^{2}=0\},$$and
the Lemma now easily follows from Lemma \ref{calc_lemma_2_on_non_isolated_point}:
for any $(x,y,t)\in\bR^3$ such that $x^2-(y^{2}+t^{2})^{2}=y^{2}-(x^{2}+t^{2})^{2}=0$
and $xyt\neq0$, we have $\big(x,y,\frac{t}{\sqrt{2}},\frac{t}{\sqrt{2}}\big)\in E$.\end{proof}

\begin{lemma}\label{calc_lemma_7_on_non_isolated_point}Let $E\subset\bR^4$
be given by
$$E:=\{(x,y,z,w)\in\bR^4|x^2+y^{2}-z^{2}-w^{4}=w^{2}-(x^{2}+y^2+z^{2})^{2}=0\}.$$
Then, for any $\epsilon>0$ there exists $(x,y,z,w)\in E\cap B(\epsilon)$
such
that $xyzw\neq
0$. 
\end{lemma}
\begin{proof}We may rewrite $$E:=\{(x,y,z,w)\in\bR^4|x^2+y^{2}-z^{2}-(x^{2}+y^2+z^{2})^{4}=0;w^{2}=(x^{2}+y^2+z^{2})^{2}\},$$so
in particular in the case of $y=z>0$ we get that if $w=(x^2+2y^2)$ and
$x^2=(x^2+2y^2)^4$ then $(x,y,z,w)\in E$. We note that last equation always
holds if $2y^2=\sqrt[4]{x^2}-x^2$, or $y=\sqrt{\frac{\sqrt[4]{x^2}-x^2}{2}}$.
We conclude that for any $0<\epsilon<1$ $$\big(\epsilon,\sqrt{\frac{\sqrt[4]{\epsilon^2}-\epsilon^2}{2}},\sqrt{\frac{\sqrt[4]{\epsilon^2}-\epsilon^2}{2}},\sqrt[2]{\epsilon}\big)\in
E,$$which proves the Lemma. \end{proof}

\begin{lemma}\label{calc_lemma_8_on_non_isolated_point}Let $E\subset\bR^4$
be given by
$$E:=\{(x,y,z,w)\in\bR^4|x^2-y^{2}-(z^2+w^{2})^{2}=z^{2}-w^{2}-(x^{2}+y^{2})^{2}=0\}.$$
Then, for any $\epsilon>0$ there exists $(x,y,z,w)\in E\cap B(\epsilon)$
such
that $xyzw\neq
0$.  
\end{lemma}
\begin{proof}Note that when $(x,y)=(z,w)$ we get that $(x,y,z,w)\in E$ is
equivalent to the condition $$g(x,y):=x^2-y^2-(x^2+y^2)^2=0.$$ So in the
coordinate system $(u,v):=(x-y,x+y)$ we have $\frac{\partial^2g}{\partial
u \partial v}(\vec 0)\neq 0$. By the Morse Lemma (see \cite{P}) there exists a continuous coordinate change around the origin in the $u-v$ plane, which is the $x-y$
plane, $(u,v)\mapsto(\tilde u,\tilde v)$ such that $g(\tilde u,\tilde v)=\tilde
u\tilde v$. Writing $x(\tilde u,\tilde v),y(\tilde u,\tilde v)$ we get that $g(x(\tilde u,0),y(\tilde u,0))=0$ for any $\tilde u$. Also note that $(x(\tilde u,0),y(\tilde u,0)\neq(0,0)$, so for any $\epsilon>0$ there exists  $\vec0\neq(x,y,z,w)=(x(\tilde u,0),y(\tilde u,0),x(\tilde u,0),y(\tilde u,0))\in E\cap B(\epsilon)$. It is now enough to show that no point $\vec 0\neq(x,y,z,w)\in
E\cap\{x=z\}\cap\{y=w\}\cap B(\frac{1}{2})$ satisfies $xy=0$. Indeed, if $y=w=0$ then $(x,y,z,w)\in
E$ implies $x^2-x^{4}=0$ and so $x\in\{0,\pm1\}$, which is impossible.  If $x=z=0$ then $(x,y,z,w)\in
E$ implies $-y^2-y^{4}=0$ and so $y=w=0$, which is impossible.\end{proof}

\begin{lemma}\label{calc_lemma_9_on_non_isolated_point}Let $E\subset\bR^4$
be given by
$$E:=\{(x,y,z,w)\in\bR^4|x^2-y^{2}-z^{2}-w^{4}=w^{2}-(x^{2}+y^2+z^{2})^{2}=0\}.$$
Then, for any $\epsilon>0$ there exists $(x,y,z,w)\in E\cap B(\epsilon)$
such
that $xyzw\neq
0$. 
\end{lemma}
\begin{proof}We may rewrite $$E:=\{(x,y,z,w)\in\bR^4|x^2-y^{2}-z^{2}-(x^{2}+y^2+z^{2})^{4}=0;w^{2}=(x^{2}+y^2+z^{2})^{2}\},$$so
in particular in the case of $y=z$ we get that if $w^2=(x^2+2y^2)^{2}$
and also
$x^2-2y^2=(x^2+2y^2)^4$ then $(x,y,z,w)\in E$. We can rewrite the last equation
as $$g(x,y)=x^2-2y^2-(x^2+2y^2)^4=0,$$ so in the
coordinate system $(u,v):=(x-\sqrt{2}y,x+\sqrt{2}y)$ we have $\frac{\partial^2g}{\partial
u \partial v}(\vec 0)\neq 0$. By the Morse Lemma (see \cite{P}) there exists a continuous coordinate change around the origin in the $u-v$ plane, which is the $x-y$
plane, $(u,v)\mapsto(\tilde u,\tilde v)$ such that $g(\tilde u,\tilde v)=\tilde
u\tilde v$. In particular
we have $$\tilde E:=\{(x(\tilde u,\tilde
v),y(\tilde u,\tilde v),y(\tilde u,\tilde v),x(\tilde u,\tilde
v)^2+2y(\tilde u,\tilde v)^2)\in\bR^4|\tilde
u\tilde v=0\}\subset E,$$
So the origin is not an isolated point of $E\cap\{y=z\}$. It
is now enough to show no point $\vec 0\neq(x,y,z,w)\in
E\cap\{y=z\}\cap B(\frac{1}{2})$ satisfies $xyw=0$. Indeed, if $y=z=0$ then $(x,y,z,w)\in
E$ implies $g(x,0)=x^2-x^8=0$, so $x=0$. We then also have $w^2=0$.  Similarly, if $x=0$ then $(x,y,z,w)\in
E$ implies $g(0,y)=-2y^2-(2y^2)^4=0$ and so $y=z=0$ which we already saw implies $z=w=0$. Finally, if $w=0$ then $(x,y,z,w)\in
E$ implies $(x^2+y^2+z^2)^2=0$ and so $x=y=z=0$.\end{proof}

We are now ready to state and prove Proposition \ref{prop_ideals_in_P2Rn_generated_by_2_foreign_jets}. 

\begin{proposition}\label{prop_ideals_in_P2Rn_generated_by_2_foreign_jets}Fix $n\in\mathbb{N}$
and let $I\lhd\mathcal{P}_0^2(\mathbb{R}^n)$ be an ideal such that there exist a partition $n=n_1+n_2$
and two homogenous polynomials of degree 2, $p_1\in\bR[x_1,\dots,x_{n_1}]$
and $p_2\in\bR[x_{n_1+1},\dots,x_{n}]$, such that $I=\langle p_1,p_2\rangle_2$. Then, the following are equivalent:
\begin{itemize}
\item $I$ is closed.
\item There exists a closed set $\vec 0 \in E\subset\mathbb{R}^n$ such that
$I=I^2(E)$.
\item There exists a closed semi-algebraic set $\vec 0 \in E\subset\mathbb{R}^n$
such that $I=I^2(E)$.
\item $Allow(I)\neq\emptyset$.
\end{itemize}\end{proposition}

\begin{proof}If $I$ is principal the proposition follows from Proposition \ref{prop_principal_ideals_in_P2Rn} and Remark \ref{remark_on_prop_principal_ideals_in_P2Rn}. So we may assume both $n_1,n_2>0$ and both $p_1$ and $p_2$ are non-zero quadratic forms. By Corollary \ref{cor_no_allowed_directions_SECOND_PAPER} and Theorem \ref{main_theorem_on_necessary_condition_SECOND_PAPER} it is enough to
prove that if $Allow(I)\neq\emptyset$ then there exists a closed semi-algebraic set
$\vec 0 \in E\subset\mathbb{R}^n$
such that $I=I^2(E)$. So we further assume that $Allow(I)\neq\emptyset$ and we will show that there exists a closed semi-algebraic set
$\vec 0 \in E\subset\mathbb{R}^n$
such that $I=I^2(E)$.

By applying a linear coordinate change on the first $n_1$ coordinates and
independently a linear coordinate change on the last $n_2$ coordinates, and
moreover possibly replacing $p_1$ by $-p_1$ and possible replacing $p_2$ by $-p_2$, and finally
reordering the coordinates, we may assume that there exists a disjoint partition
\begin{gather}\label{rand_label_for_foreign_prop1}\{1,2\dots,n\}=J^1_+\cup
J^1_-\cup J^2_+\cup
J^2_-\cup J_0\end{gather}
with
\begin{gather}\label{rand_label_for_foreign_prop2}J^{1}_+,J^{2}_+\neq\emptyset\end{gather}
such that $i<j<k<l<m$ for all $i\in J^1_+,j\in J^1_-,k\in J^2_+,l\in J^2_-,m\in
J_0$ (if $J_-^1=\emptyset$ we omit the index $j$ in the inequality and similarly if either $J_-^2$ or $J_0$ is empty), and \begin{gather}\label{rand_label_for_foreign_prop3}p_{\alpha}(x_1,\dots x_n)=\sum\limits_{i\in
J^{\alpha}_+}x_i^2-\sum\limits_{i\in
J^{\alpha}_-}x_i^2\text{ for all }\alpha\in\{1,2\}.\end{gather} First note that if $J^1_-\cup J^2_- \cup J_0=\emptyset$ then $I$ contains
a positive define quadratic form. In that case $Allow(I)=\emptyset$,
which is a contradiction. So we may assume from now on that
\begin{gather}\label{rand_label_for_foreign_prop4}J^1_-\cup J^2_-
\cup J_0\neq\emptyset.\end{gather}
Define

\begin{gather}\label{rand_label_for_foreign_prop5}q_\alpha(x_{1},\dots x_n):=\sum\limits_{i\in
J^\alpha_+}x_i^2-\sum\limits_{i\in
J^\alpha_-}x_i^2-\big(\sum\limits_{i\in\{1,\dots,n\}\setminus( J^\alpha_+\cup
J^\alpha_-)}x_i^2\big)^2\text{ for all }\alpha\in\{1,2\}.\end{gather}We will show that $I=I^2(\{q_1=q_2=0\})$. 

\

Before we start set $E:=\{q_1=q_2=0\}$. We obviously have

\begin{gather}\label{rand_label_for_foreign_prop17}\sum\limits_{i\in
J^\alpha_+}x_i^2-\sum\limits_{i\in
J^\alpha_-}x_i^2=J^2(q_\alpha)=p_\alpha\in
I^2(E)\text{ for all }\alpha\in\{1,2\}.\end{gather}
Now fix some $f\in C^2(\bR^n)$ such that $f|_E=0$. First let us show that
\begin{gather}\label{rand_label_for_foreign_prop6}\frac{\partial
f}{\partial x_i}(\vec
0)=\frac{\partial f}{\partial x_j}(\vec
0)=\frac{\partial^2f}{\partial x_i\partial
x_j}(\vec 0)=0\text{ for all }1\leq i\neq j\leq n.\end{gather}
To show that
(\ref{rand_label_for_foreign_prop6}) holds, it is enough to show that
\begin{multline}\label{rand_label_for_foreign_prop7}\text{for all } 1\leq
i\neq j\leq n \text{ and any }\epsilon>0\text{ there exits }(x_1,x_2,\dots,x_n)\in
E\cap B(\epsilon)\\\text{ such that }x_i\cdot x_j\neq0.\end{multline}Indeed, if (\ref{rand_label_for_foreign_prop7}) holds then as $E$ is symmetric
with respect to $x_i\mapsto -x_i$ and $x_j\mapsto -x_j$, we can apply Rolle's
Theorem to get $\frac{\partial f}{\partial x_i}(\vec
0)=\frac{\partial f}{\partial x_j}(\vec
0)=0$ and apply Lemma \ref{lemma_vanishes_on_corners_of_rectangles} to get
$\frac{\partial^2f}{\partial x_i\partial
x_j}(\vec 0)=0$, i.e., (\ref{rand_label_for_foreign_prop6})
holds. 

Any row (Case) in Table \ref{tab:table1} below represents a subset of indices
$J\subset\{1,\dots,n\}$. For instance, Case 1 represents a subset $J=\{i,j,k\}$
for some $i\in J_+^\alpha$, $j\in J_-^\alpha$ and $k\in J_+^\beta$ (note
that in Cases 1 and 2 $J$ has cardinality 3, whereas in the remaining cases
$J$ has cardinality 4). Recalling Notation \ref{notation_restrictions_of_sets_and_functions}, the corresponding lemma in the same row shows that
$E^J\subset\bR^{|J|}$ has the following property: for any $\epsilon>0$ there exists $\vec x\in E^J\cap B(\epsilon)$
such
that no coordinate of $\vec x$ is zero. So in order to show that
(\ref{rand_label_for_foreign_prop7}) holds, it is enough to show that any (unordered) pair of indices
 $i\neq j\in\{1,\dots n\}$ can be completed to an (unordered) triple $\{i,j,k\}$
or to an (unordered) quadruple $\{i,j,k,l\}$, such that this triple or quadruple
appears as a subset $J$ in one of the rows of Table \ref{tab:table1}. 

\begin{table}[h!]
  \begin{center}
    \caption{Subsets of indices and corresponding Lemmas.}
    \label{tab:table1}
    \begin{tabular}{c|c|c|c|c|c}

      \hline
      \textbf{Case \#} & \multicolumn{4}{c|}{\textbf{Indices types}} & \textbf{Lemma
\#}\\ 
      \hline
    1   & $J_+^\alpha$ & $J_-^\alpha$ & $J_+^\beta$ &  & \ref{calc_lemma_1_on_non_isolated_point}
\\ \hline
2 & $J_+^\alpha$ & $J_+^\beta$ & $J_0$ & & \ref{calc_lemma_2_on_non_isolated_point}
\\ \hline
3 & $J_+^\alpha$ & $J_+^\alpha$ & $J_+^\beta$ & $J_0$ & \ref{calc_lemma_3_on_non_isolated_point}
\\ \hline
4 & $J_+^\alpha$ & $J_+^\alpha$ & $J_+^\beta$ & $J_-^\beta$ & \ref{calc_lemma_4_on_non_isolated_point}
\\ \hline
5 & $J_+^\alpha$ & $J_-^\alpha$ & $J_+^\beta$ & $J_0$ & \ref{calc_lemma_5_on_non_isolated_point}
\\ \hline
6 & $J_+^\alpha$ & $J_+^\beta$ & $J_0$ & $J_0$ & \ref{calc_lemma_6_on_non_isolated_point}
\\ \hline
7 & $J_+^\alpha$ & $J_+^\alpha$ & $J_+^\beta$ & $J_-^\alpha$ & \ref{calc_lemma_7_on_non_isolated_point}
\\ \hline
8 & $J_+^\alpha$ & $J_-^\alpha$ & $J_+^\beta$ & $J_-^\beta$ & \ref{calc_lemma_8_on_non_isolated_point}
\\ \hline
9 & $J_+^\alpha$ & $J_-^\alpha$ & $J_-^\alpha$ & $J_+^\beta$ & \ref{calc_lemma_9_on_non_isolated_point}
\\ \hline

    \end{tabular}
  \end{center}
\end{table}

Indeed, fix some (unordered) pair of indices  $i\neq j\in\{1,\dots n\}$.
\begin{itemize}
\item If one of these indices lies in $J_0$ then without loss of generality $i\in J_0$ and we have $j\in J_0\cup J_+^\alpha\cup J_-^\alpha$ for some $\alpha\in\{1,2\}$.

\begin{itemize}
\item If $j\in J_0$ then as by (\ref{rand_label_for_foreign_prop2}) $J_+^1\neq\emptyset$ and $J_+^1\neq\emptyset$,
we can find $k\in J_+^1$ and $l\in J_+^2$ and we are in Case 6. 
\item If $j\in J_+^\alpha$ then denote $\{\beta\}=\{1,2\}\setminus\{\alpha\}$.
As $J_+^\beta\neq\emptyset$ we can find $k\in J_+^\beta$ and we are in Case
2.

\item If $j\in J_-^\alpha$ then denote $\{\beta\}=\{1,2\}\setminus\{\alpha\}$.
As both $J_+^1\neq\emptyset$ and $J_+^2\neq\emptyset$
we can find $k\in J_+^\alpha$ and $l\in J_+^\beta$ and we are in Case 5.
\end{itemize}

\item If one of these indices lies in $J_+^\alpha$ for some $\alpha\in\{1,2\}$
then without loss of generality $i\in J_+^\alpha$. We denote $\{\beta\}=\{1,2\}\setminus\{\alpha\}$ and we have $j\in J_0\cup
J_+^\alpha\cup J_-^\alpha\cup J_+^\beta\cup J_-^\beta$. The case $j\in J_0$
was already treated above, so it is left to check the cases in which $j\in 
J_+^\alpha\cup J_-^\alpha\cup J_+^\beta\cup J_-^\beta$.

\begin{itemize}

\item If $j\in J_+^\alpha$ then as $J_+^\beta\neq\emptyset$ we can find $k\in
J_+^\beta$. Since $J_-^\alpha\cup J_-^\beta\cup J_0\neq\emptyset$ by (\ref{rand_label_for_foreign_prop4}), we can
either find $l\in J_-^\alpha$ and we are in Case 7, find  $l\in J_-^\beta$
and we are in Case 4, or find $l\in J_0$ and we are in Case 3.
 
\item If $j\in J_-^\alpha$ then as $J_+^\beta\neq\emptyset$ we can find $k\in
J_+^\beta$ and we are in Case
1.

\item If $j\in J_+^\beta$ then since $J_-^\alpha\cup J_-^\beta\cup J_0\neq\emptyset$
we can
either find $k\in J_-^\alpha\cup J_-^\beta$ and we are in case 1 (possibly
after renaming $\alpha\leftrightarrow\beta$), or find $k\in J_0$ and we are
in Case 2.
\item If $j\in J_-^\beta$ then as  $J_+^\beta\neq\emptyset$ we can find $k\in
J_+^\beta$ and we are (up to renaming $\alpha\leftrightarrow\beta$) in Case
1.
\end{itemize}

\item If one of these indices lies in $J_-^\alpha$ for some $\alpha\in\{1,2\}$
the without loss of generality
 $i\in J_-^\alpha$. We denote $\{\beta\}=\{1,2\}\setminus\{\alpha\}$ and we have $j\in J_0\cup
J_+^\alpha\cup J_-^\alpha\cup J_+^\beta\cup J_-^\beta$. The case $j\in J_0\cup
J_+^\alpha\cup J_+^\beta$
was already treated above (after renaming $i\leftrightarrow j$), so it is left to check the cases in which $j\in 
 J_-^\alpha\cup J_-^\beta$.

\begin{itemize}

\item If $j\in J_-^\alpha$ then as both $J_+^\alpha\neq\emptyset$ and $J_+^\beta\neq\emptyset$
we can find $k\in J_+^\alpha$ and $l\in J_+^\beta$ and we are in Case 9.

\item If $j\in J_-^\beta$ then as both $J_+^\alpha\neq\emptyset$ and $J_+^\beta\neq\emptyset$
we can find $k\in J_+^\alpha$ and $l\in J_+^\beta$ and we are in Case 8.

\end{itemize}

\end{itemize}
 We thus showed that any (unordered) pair of indices
 $i\neq j\in\{1,\dots n\}$ can be completed to an (unordered) triple $\{i,j,k\}$
or to an (unordered) quadruple $\{i,j,k,l\}$, such that this triple or quadruple
appears as a subset $J$ in one of the rows of Table \ref{tab:table1}. So we proved that (\ref{rand_label_for_foreign_prop7}) holds, and so that (\ref{rand_label_for_foreign_prop6}) holds. 

We conclude that 

\begin{gather}\label{rand_label_for_foreign_prop16}I^2(E)\subset\langle x_1^2,x_2^2,\dots,x_n^2
\rangle_{2}\end{gather} and in particular
\begin{gather}\label{rand_label_for_foreign_prop8}T^{J^\alpha_+}(I^2(E))\subset
\langle\{x_i^2\}_{i\in J^\alpha_+}\rangle_2;T^{J^\alpha_-}(I^2(E))\subset
\langle\{x_i^2\}_{i\in J^\alpha_-}\rangle_2\text{ for all }\alpha\in\{1,2\}.\end{gather}
Note that the functions $q_1,q_2\in C^2(\bR^n)$ both vanish on $E$, and so $J^2(q_\alpha)=p_\alpha\in
I^2(E)$ for $\alpha\in\{1,2\}$, and in particular \begin{gather}\label{rand_label_for_foreign_prop9}\langle\sum\limits_{i\in
J^\alpha_+}x_i^2\rangle_2\subset T^{J^\alpha_+}(I^2(E));\langle\sum\limits_{i\in
J^\alpha_-}x_i^2\rangle_2\subset T^{J^\alpha_-}(I^2(E)\text{ for all }\alpha\in\{1,2\}.\end{gather}
By (\ref{rand_label_for_foreign_prop8}), $T^{J^\alpha_+}(I^2(E))$ is a linear
subspace of $\text{span}_\bR\{x_i^2\}_{i\in J^\alpha_+}$, it is closed under all permutations
of coordinates (this follows immediately from the definition of $E$), and
by (\ref{rand_label_for_foreign_prop9}), it contains a non-zero vector $\sum\limits_{i\in
J^\alpha_+}x_i^2$ whose all coordinates
are equal. So, by Lemma \ref{rep_theory_lemma}, either $T^{J^\alpha_+}(I^2(E))=
\langle\sum\limits_{i\in
J^\alpha_+}x_i^2\rangle_2$, or $T^{J_{+}}(I^2(E))=\text{span}_\bR\{x_i^2\}_{i\in
J^\alpha_+}$.
If $|J^\alpha_+|=1$ these are the same. However, if $|J^\alpha_+|>1$ the
latter is not
invariant with respect to rotations in $\bR^{|J^\alpha_+|}$, whereas $E^{J^\alpha_+}$
is
invariant, which is impossible. We conclude that \begin{gather}\label{rand_label_for_foreign_prop10}T^{J^\alpha_+}(I^2(E))=
\langle\sum\limits_{i\in
J^\alpha_+}x_i^2\rangle_2\text{ for all }\alpha\in\{1,2\}.\end{gather} Similar argument shows that \begin{gather}\label{rand_label_for_foreign_prop11}T^{J^\alpha_-}(I^2(E))=
\langle\sum\limits_{i\in
J^\alpha_-}x_i^2\rangle_2\text{
for all }\alpha\in\{1,2\}.\end{gather} Recall again that $J^2(q_\alpha)=p_\alpha\in
I^2(E)$, and we conclude that \begin{gather}\label{rand_label_for_foreign_prop12}T^{J^\alpha_+\cup
J^\alpha_-}(I^2(E))\in\big\{\langle\sum\limits_{i\in
J^\alpha_+}x_i^2-\sum\limits_{i\in
J^\alpha_-}x_i^2\rangle_2,\langle\sum\limits_{i\in
J^\alpha_+}x_i^2,\sum\limits_{i\in
J^\alpha_-}x_i^2\rangle_2\big\}\text{
for all }\alpha\in\{1,2\}.\end{gather}
If $J^\alpha_-=\emptyset$ these are the same.

Assume towards contradiction that $J^\alpha_-\neq\emptyset$ and that  $T^{J^\alpha_+\cup
J^\alpha_-}(I^2(E))=\langle\sum\limits_{i\in
J^\alpha_+}x_i^2,\sum\limits_{i\in
J^\alpha_-}x_i^2\rangle_2$. We denote $\{\beta\}:=\{1,2\}\setminus\{\alpha\}$.
In that case we can find $i\in J^\alpha_+$, $j\in J^\alpha_-$ and $k\in J^\beta_+$.
Then, $x_i^2+x_j^2+x_k^2\in I^2(E^{ijk})$ and so $Allow(I^2(E^{ijk}))=\emptyset$ and the origin must be an isolated
point of $E^{ijk}$ by Corollary \ref{cor_no_allowed_directions_SECOND_PAPER}. However, $$E^{ijk}=\{(x_i,x_j,x_k)\in\bR^3|x_i^2-x_j^2-x_k^4=x_k^2-(x_i^2+x_j^2)^{2}=0\},$$and
so by Lemma \ref{calc_lemma_1_on_non_isolated_point} the origin is not an
isolated point of $E^{ijk}$, which is a contradiction. We conclude that
\begin{gather}\label{rand_label_for_foreign_prop13}T^{J^\alpha_+\cup
J^\alpha_-}(I^2(E))=\langle\sum\limits_{i\in
J^\alpha_+}x_i^2-\sum\limits_{i\in
J^\alpha_-}x_i^2\rangle_2=\langle T^{J^\alpha_+\cup
J^\alpha_-}(p_\alpha)\rangle_2\text{
for all }\alpha\in\{1,2\}.\end{gather}
Let us continue with the fixed $f\in C^2(\bR^n)$ such that $f|_E=0$. We will next show that \begin{gather}\label{rand_label_for_foreign_prop14}\frac{\partial^2 f}{\partial
x_k^2}(\vec 0)=0\text{
for all }k\in J_0.\end{gather} Assume towards contradiction that $\frac{\partial^2 f}{\partial
x_k^2}(\vec 0)\neq0$ for some $k\in J_0$. By (\ref{rand_label_for_foreign_prop16}), and  as $p_1,p_2\in I^2(E)$, we can assume without
loss of generality that there exist $i\in J^1_+,j\in J^2_+$ and $C\neq0$ such that
$$f^{ijk}(x_i,x_j,x_k)=Cx_{k}^2+o(x_i^2+x_j^2+x_k^2).$$ In this case $$E^{ijk}=\{(x_i,x_j,x_k)\in\bR^3|x_i^2-(x_j^2+x_k^2)^{2}=x_{j}^{2}-(x_i^2+x_k^2)^{2}=0\},$$
and recall $f^{ijk}|_{E^{ijk}}=0$. Note  that for any $0<\epsilon<1$, we have $(\epsilon,\epsilon,\sqrt{\epsilon-\epsilon^2})\in
E^{ijk}$ and so $$0=f^{ijk}(\epsilon,\epsilon,\sqrt{\epsilon-\epsilon^2})=C(\epsilon-\epsilon^2)+o(\epsilon),$$which
is a contradiction.

We conclude that (\ref{rand_label_for_foreign_prop14}) holds, which together with (\ref{rand_label_for_foreign_prop16}) imply that \begin{gather}\label{rand_label_for_foreign_prop15}T^{J_0}(I^m(E))=\{\vec 0\}.\end{gather}
Combining (\ref{rand_label_for_foreign_prop17}),(\ref{rand_label_for_foreign_prop16}),(\ref{rand_label_for_foreign_prop13}) and (\ref{rand_label_for_foreign_prop15}) we conclude that $I^2(E)=\langle
p_1,p_2\rangle_2$.\end{proof}

\section{The $C^3(\mathbb{R}^2)$ case}\label{section_3_2}

We start with two simple but useful Lemmas.

\begin{lemma}\label{nice_easy_lemma} (i) The $m^{\text{th}}$
jet of any $C^m$ function in two variables that vanishes on at least $m+1$
distinct lines passing through the origin is zero. 
(ii) Let $p\in\mathcal{P}_0^m(\mathbb{R}^2)$ be a non-zero homogenous
polynomial in two variables of degree $k\leq m$. Then, $Allow(p)=\{\omega\in S^1|p(\omega)=0\}$ consists
of at most $2k$ points. Moreover, $Allow(p)$ consists of an even number of points.\end{lemma}
\begin{proof}(i) follows immediately from Example \ref{claim-on-lines}. Let $p$ be as in (ii) and assume towards contradiction that $Allow(p)$ consists of more than
$2k$ points. As $p$ is homogenous, if $\omega\in S^1$ is such that $p(\omega)=0$, then $p(\vec x)=0$ for any $\vec x$ that lies on the line containing
the origin and $\omega$. Thus, there exist at least $k+1$ distinct lines
passing through the origin such that $p$
vanishes on each of these lines. As $p\in C^k(\bR^2)$ we get a contradiction
by (i). The "moreover" part
follows from the fact as $p$ is homogenous, $p(\omega)=0\iff p(-\omega)=0.$\end{proof}

\begin{lemma}\label{lemma_linear_times_definite_in_closed_imply_many_jets}Let $I\lhd\mathcal{P}_0^3(\mathbb{R}^2)$ be a closed ideal and assume $l(x,y)\cdot q(x,y)\in I$ where $l(x,y)$ is a linear form and $q(x,y)$ is a definite quadratic form. Then, $l(x,y)\cdot x^2,l(x,y)\cdot y^2,l(x,y)\cdot xy\in I$.\end{lemma}

\begin{proof}Let $S_1(x,y):=\frac{x^2}{q(x,y)}$, one can easily check that in a punctured neighborhood of the origin  $\abs{\partial
^\alpha S_{1}(\vec x)}=O(\abs{\vec x}^{-\abs{\alpha}})$ for all $\abs{\alpha}\leq
3$, i.e., $S_1$ satisfies (\ref{second_cond_of_implied_poly_in_direction_SECOND_PAPER}) above (for any $\omega$ and any $\delta_\omega$, with some small fixed $r_\omega>0$). Thus, we have that the jet $\big(\frac{x^2}{q(x,y)}\big)\cdot \big(l(x,y)\cdot q(x,y)\big)=l(x,y)\cdot x^2$ is strongly implied by $I$ in any direction, and in particular in any allowed direction of $I$. As $I$ is closed we get $l(x,y)\cdot x^2\in I$ by Corollary \ref{cor_strong_directional_implication_imply_implication_SECOND_PAPER}. The same proof with straightforward adjustments shows that also $l(x,y)\cdot y^2,l(x,y)\cdot
xy\in I$.\ \end{proof}

We now classify all closed ideals in $\mathcal{P}_0^3(\mathbb{R}^2)$ up to $C^3$ semi-algebraic coordinate changes around the origin. We deal separately with principal ideals (in Lemma \ref{lemma_classify_closed_principal_in_C3R2_case_new_version}) and non-principal ideals (in Lemma
\ref{lemma_classify_closed_non_principal_in_C3R2_case_new_version}).

\begin{lemma}\label{lemma_classify_closed_principal_in_C3R2_case_new_version} Any principal closed ideal $I\lhd\mathcal{P}_0^3(\mathbb{R}^2)$ is, up to a $C^3$ semi-algebraic
coordinate change around the origin, one of the following\footnote{In Proposition \ref{prop_C3R2_case_new_version} we will see that all the ideals that appear in the statement of Lemma \ref{lemma_classify_closed_principal_in_C3R2_case_new_version} are closed. The same applies for the ideals in the statements of Lemma \ref{lemma_classify_closed_non_principal_in_C3R2_case_new_version} (also by Proposition \ref{prop_C3R2_case_new_version}) and Lemmas \ref{lemma_classify_closed_ideals_in_C2R3_1_dim_allow}, \ref{lemma_classify_closed_ideals_in_C2R3_2_dim_allow} and \ref{lemma_classify_closed_ideals_in_C2R3_3_dim_allow} (by Proposition \ref{prop_C2R3_case_new_version}).}:
\begin{gather}\label{model_for_3_2_1}\{\vec 0\}=I^3(\mathbb{R}^2)\end{gather}
\begin{gather}\label{model_for_3_2_3}\langle x\rangle_3=I^2(\{x=0\}) \text{
 -- see Example \ref{a_union_of_coordinates}}
\end{gather}
\begin{gather}\label{model_for_3_2_4}\langle x^2\rangle_3= I^3(\{x(x^2-y^{5})=0\}\cap\{x\geq0\})
\text{ -- see Example \ref{example-xx}}\end{gather}
\begin{gather}\label{model_for_3_2_5}\langle xy\rangle_3= I^3(\{xy=0\})
\text{ -- see Example \ref{claim-on-lines}}
\end{gather}
\begin{gather}\label{model_for_3_2_6}\langle xy(x+y)\rangle_3= I^3(\{xy(x+y)=0\}) \text{ -- see Example \ref{claim-on-lines}}
\end{gather}
\begin{gather}\label{model_for_3_2_7}\langle x^2-y^3\rangle_3
\end{gather}
\begin{gather}\label{model_for_3_2_9}\langle yx^{2}\rangle_3
\end{gather}
\begin{gather}\label{model_for_3_2_10}\langle x^{3}\rangle_3
\end{gather}

\end{lemma}

\begin{proof}

Let $I\lhd \mathcal{P}_0^3(\mathbb{R}^2)$ be a principal closed ideal. So $I=\langle p\rangle_3$ for some jet $p$. If $p=0$ then we are in case (\ref{model_for_3_2_1}). If
$p$ is of order of vanishing 1 then by applying a $C^3$ semi-algebraic
coordinate change around the origin we may assume $p=x$ and we are in case (\ref{model_for_3_2_3}).

\

\textbf{A. Assume
$p$ is of order of vanishing 2.} In that case by applying a linear coordinate change and possibly replacing $p$ by $-p$, we may assume that there exists a homogenous polynomial of degree 3 $q(x,y)$, where $q$ may be the zero polynomial, such that $p(x,y)=x^2+y^2+q(x,y)$, $p(x,y)=xy+q(x,y)$ or $p(x,y)=x^2+q(x,y)$.

\

\textbf{A.1. Assume $p(x,y)=x^2+y^2+q(x,y)$.} In that case $Allow(I)=\emptyset$, which is a contradiction to $I$ being closed by
Corollary \ref{cor_no_allowed_directions_SECOND_PAPER}.

\

\textbf{A.2. Assume $p(x,y)=xy+q(x,y)$.} In that case also $xy^2,x^2y\in I$, so without loss of generality $q(x,y)=ax^3+by^3$ for some $a,b\in\mathbb{R}$. We have $p(x,y)=x(y+ax^2)+by^3$, so in the coordinate system $(u,v)=(x,y+ax^2)$ whose inverse is $(x,y)=(u,v-au^2)$ we get that $p(u,v)=uv+b(v-au^{2})^3=uv+bv^3=v(u+bv^{2})$ (as a 3-jet). Finally, in the coordinate system $(\tilde u,\tilde v)=(u+bv^2,v)$
we get $p(\tilde u\tilde v)=\tilde u\tilde v$ and we are in case (\ref{model_for_3_2_5}).

\

\textbf{A.3. Assume $p(x,y)=x^2+q(x,y)$.} In that case also $x^3,x^2y\in
I$, so without loss of generality $p(x,y)=x^2+axy^2+by^3$ for some $a,b\in\mathbb{R}$. If $a=0$ and $b=0$ then $p=x^2$ and we are in case (\ref{model_for_3_2_4}). If $a=0$ and $b\neq 0$ then by scaling $y\mapsto-b^{1/3}y$ we get $p=x^2-y^3$ and we are in case (\ref{model_for_3_2_7}). Otherwise $a\neq0$, and scaling $y\mapsto\abs{a}^{1/2}y$ and possibly applying $x\mapsto -x$ we get $p=x^2-xy^2+cy^3$ for some $c\in\mathbb{R}$. Now apply the coordinate change $(X,Y)=(x-\frac{1}{2}y^{2},-y)$ we get $p=X^2-cY^3$ (as a 3-jet). If $c=0$ then we are in case (\ref{model_for_3_2_4}). Otherwise by scaling $Y\mapsto c^{1/3}Y$ we get $p=X^2-Y^3$
and we are in case (\ref{model_for_3_2_7}). 

\

\textbf{B. Assume
$p$ is of order of vanishing 3.} In that case  by Lemma \ref{nice_easy_lemma}(ii) we have either $Allow(p)=\emptyset$ or $\abs{Allow(p)}\in\{2,4,6\}$. If $Allow(I)=\emptyset$,
we get a contradiction to $I$ being closed by
Corollary \ref{cor_no_allowed_directions_SECOND_PAPER}.

\

\textbf{B.1. Assume $\abs{Allow(p)}=6$.} In that case $p$ vanishes on exactly 3 distinct lines passing through the origin, so by applying a linear coordinate change we may assume $p$ vanishes on $\{x=0\}\cup\{y=0\}\cup\{ax+by=0\}$ for some $a,b\in\mathbb{R}^\times$. Let $(\alpha,\beta)\in S^1$ be such that $\alpha\neq0$, $\beta\neq0$ and $\alpha a+\beta b\neq0$. Define $\tilde p(x,y):=p(x,y)-\frac{p(\alpha,\beta)}{\alpha \beta(\alpha a+\beta b)}xy(ax+by).$ Then, $\tilde p$ is a homogenous polynomial of degree 3 (or the zero polynomial), and it vanishes on 4 distinct lines passing through the origin: $\{x=0\}$, $\{y=0\}$, $\{ax+by=0\}$ and the line that passes through $(\alpha,\beta)$. By Lemma \ref{nice_easy_lemma}(i), $\tilde p(x,y)$ is the zero polynomial, and $p(x,y)=\frac{p(\alpha,\beta)}{\alpha
\beta(\alpha a+\beta b)}xy(ax+by)$ which implies $I=\langle xy(ax+by)\rangle_3$. Now, in the coordinate system $(X,Y)=(ax,by)$ we have $I=\langle XY(X+Y)\rangle_3$, i.e., we are in case (\ref{model_for_3_2_6}).

\

\textbf{B.2. Assume $\abs{Allow(p)}=4$.} In that case $p$ vanishes on exactly
2 distinct lines passing through the origin, so by applying a linear coordinate
change we may assume $p$ vanishes on $\{x=0\}\cup\{y=0\}$. So $p$ is divisible by both $x$ and $y$, i.e., $p(x,y)=xy(ax+by)$
for some $a,b\in\mathbb{R}$ such that $a^2+b^2\neq0$. If $ab\neq0$ then $p$ also vanishes on the line $ax+by=0$ which is a contradiction. So by possibly renaming $x\leftrightarrow y$ we get $I=\langle yx^{2}\rangle_3$ and we are in case (\ref{model_for_3_2_9}).  

\

\textbf{B.3. Assume $\abs{Allow(p)}=2$.} In that case $p$ vanishes on exactly
1 line passing through the origin, so by applying a linear coordinate
change we may assume $p$ vanishes on $\{x=0\}$. So $p$ is divisible
by $x$, i.e., $p(x,y)=x\cdot q(x,y)$
for some non-zero quadratic form $q(x,y)$. We thus have 3 options:

\begin{itemize}
\item If $q$ is an indefinite form then by applying a linear coordinate change it has the from $\tilde x\tilde y$, which implies $\abs{Allow(I)}\geq4$, and we get a contradiction.
\item If $q$ is a semi-definite (that is not definite) form then $q(x,y)=c(ax+by)^{2}$ for some $a,b,c\in\mathbb{R}$ such that $c\neq 0$ and $a^2+b^2\neq0$. If $b\neq 0$ then $p$ also vanishes on the line $\{y=-\frac{a}{b}x\}$, which is a contradiction. We thus have $I=\langle x^3\rangle_3$ and we are in case (\ref{model_for_3_2_10}). 
\item If $q$ is definite then we get a contradiction to $I$ being closed by Lemma \ref{lemma_linear_times_definite_in_closed_imply_many_jets}. 
\end{itemize}\end{proof}

\begin{remark}\label{remark_on_distinctness_for_3_2_case_principal}The ideals
(\ref{model_for_3_2_1})-(\ref{model_for_3_2_10}) are pairwise distinct, i.e.,
there does not exist a $C^3$ coordinate change around the origin such that
a given ideal can take two different forms from this list. (\ref{model_for_3_2_1}) is obviously distinct, and (\ref{model_for_3_2_3}) is distinct too as it is the only ideal in this list that contains a jet of order of vanishing 1. (\ref{model_for_3_2_4}) is distinct as it is the only ideal in this list that contains a square of a jet of order of vanishing 1. (\ref{model_for_3_2_5}) and (\ref{model_for_3_2_7}) are the remaining ideals in this list that contain a jet of order of vanishing 2, however the first has 2 linearly independent allowed directions, while the latter does not. Finally, (\ref{model_for_3_2_6}), (\ref{model_for_3_2_9}) and (\ref{model_for_3_2_10}) are the remaining ideals, and they satisfy $\abs{Allow(I)}=6,4,2$ (respectively), and so are also distinct.\end{remark}

\begin{lemma}\label{lemma_classify_closed_non_principal_in_C3R2_case_new_version}Any closed non-principal ideal $I\lhd\mathcal{P}_0^3(\mathbb{R}^2)$ is, up to a $C^3$
semi-algebraic
coordinate change around the origin, one of the following:

\begin{gather}\label{non_pricipal_model_for_3_2_1}\mathcal{P}_0^3(\mathbb{R}^2)=I^3(\{\vec0\})\end{gather}
\begin{gather}\label{non_pricipal_model_for_3_2_2}\langle x^{2},xy\rangle_3= I^3(\{x^2-y^{5}=0\}\cap\{x\geq0\})
\text{ -- see Example \ref{example-xx-xy}}\end{gather}
\begin{gather}\label{non_pricipal_model_for_3_2_3_new_as_missed}\langle
x^2,xy^2\rangle_3\end{gather}
\begin{gather}\label{non_pricipal_model_for_3_2_3}\langle
x^2-y^{3},xy^{2}\rangle_3\end{gather}
\begin{gather}\label{non_pricipal_model_for_3_2_4}\langle x^3,x^2y,xy^2\rangle_3\end{gather}
\begin{gather}\label{non_pricipal_model_for_3_2_5_a}\langle
x^3,x^{2}y\rangle_3\end{gather}
\begin{gather}\label{non_pricipal_model_for_3_2_5_b}\langle
x^2y,xy^{2}\rangle_3\end{gather}
\end{lemma}

\begin{proof}Let $I\lhd \mathcal{P}_0^3(\mathbb{R}^2)$ be a non-principal closed
ideal. If $I=\mathcal{P}_0^3(\mathbb{R}^2)$ then we are in case (\ref{non_pricipal_model_for_3_2_1}). So from now on we assume $I\neq\mathcal{P}_0^3(\mathbb{R}^2)$. We analyse by the minimal order of vanishing of jets in $I$.     

\

\textbf{A. Assume that the minimal order of vanishing of jets in $I$ is 1.} By applying a semi-algebraic coordinate change around the origin we may assume $x\in I$. As $I$ is not principal there exists $ay+by^2+cy^3\in I$ for some $a,b,c\in\mathbb{R}$ with $a^2+b^2+c^2\neq0$. In that case $Allow(I)=\emptyset$ though $I\neq\mathcal{P}_0^3(\mathbb{R}^2)$, which is a contradiction to $I$ being closed by Corollary \ref{cor_no_allowed_directions_SECOND_PAPER}. 

\

\textbf{B. Assume that the minimal order of vanishing of jets in $I$ is 2.}
In that case by applying a linear coordinate
change we may assume that there exists
a homogenous polynomial of degree 3 $q(x,y)$, where $q$ may be the zero polynomial,
such that $x^2+y^2+q(x,y)\in I$, $xy+q(x,y)\in I$ or $x^2+q(x,y)\in I$.

\

\textbf{B.1. Assume $x^2+y^2+q(x,y)\in I$.} In that case $Allow(I)=\emptyset$,
which is a contradiction to $I$ being closed by
Corollary \ref{cor_no_allowed_directions_SECOND_PAPER}. 

\

\textbf{B.2. Assume $xy+q(x,y)\in I$.} In that case also $xy^2,x^2y\in
I$, so without loss of generality $q(x,y)=ax^3+by^3$ for some $a,b\in\mathbb{R}$.
We have $x(y+ax^2)+by^3\in I$, so in the coordinate system $(u,v)=(x,y+ax^2)$
whose inverse is $(x,y)=(u,v-au^2)$ we get that $uv+b(v-au^{2})^3=uv+bv^3=v(u+bv^{2})\in I$. Now, in the coordinate system $(\tilde u,\tilde v)=(u+bv^2,v)$
we get $\tilde u\tilde v\in I$. We conclude that by applying a semi-algebraic coordinate change we are allowed to assume that $q(x,y)=0$ and $xy,xy^2,x^2y\in I$. As $I$ is not principal there exist $a,b,c,d\in\mathbb{R}$ such that $ax^2+by^2+cx^3+dy^3\in I$ and $a^2+b^2+c^2+d^2\neq 0$.

Assume there exist $a,b,c,d\in\mathbb{R}$ such that
$ax^2+by^2+cx^3+dy^3\in I$ and $a^2+b^2\neq 0$. Without loss of generality $a\neq0$ and $x^2+by^2+cx^3+dy^3\in I$ (with some other coefficients $b,c,d$). If $b>0$ then $x^2+by^2+cx^3+dy^3\in I$ implies $Allow(I)=\emptyset$ which is a contradiction to $I$ being closed by
Corollary \ref{cor_no_allowed_directions_SECOND_PAPER}. If $b<0$ then again $xy,(x+\sqrt{-b}y)\cdot(x-\sqrt{-b}y)+cx^3+dy^3\in I$  implies $Allow(I)=\emptyset$ which is a contradiction to $I$ being closed
by
Corollary \ref{cor_no_allowed_directions_SECOND_PAPER}. So we must have $b=0$ (for any such $x^2+by^2+cx^3+dy^3\in I$) and $x^{3},xy^{2},x^{2}+dy^{3}\in
I$ which implies that $Allow(I)\subset \{(x,y)=(0,\pm 1)\}$. Thus, $\frac{x}{y}$ satisfies (\ref{second_cond_of_implied_poly_in_direction_SECOND_PAPER}) in any allowed direction, so the jet $x^2=xy\cdot\frac{x}{y}\in I$ (by Corollary \ref{cor_strong_directional_implication_imply_implication_SECOND_PAPER}) and so also $dy^3\in I$. If $d\neq0$, then  $Allow(I)=\emptyset$ which is a contradiction to $I$ being closed
by
Corollary \ref{cor_no_allowed_directions_SECOND_PAPER}. So we must have $d=0$ (for any such $x^2+dy^3\in I$). Assume towards contradiction that $I$ contains a jet with a non-zero $y^2$ term. By assumption this jet is of order of vanishing 2, so there exist $a,b,c,d,e,g\in\mathbb{R}$ such that $y^2+axy+bx^2+cx^3+dx^2y+exy^2+gy^3\in I$. As $xy,x^2,x^{2}y,x^{3},xy^{2}\in I$ we also have $y^2+gy^3\in
I$ and so  $Allow(I)=\emptyset$ which is
a contradiction to $I$ being closed
by
Corollary \ref{cor_no_allowed_directions_SECOND_PAPER}. A similar argument shows that there are no jets in $I$ with a non-zero $y^3$ term. So $I=\langle x^2,xy\rangle_3$ and we are in case (\ref{non_pricipal_model_for_3_2_2}).

It is left to check the case $a=b=0$ (for any $ax^2+by^2+cx^3+dy^3\in I$), and so without loss of generality (by possibly renaming $x\leftrightarrow y$)  $x^3+dy^3\in I$ (with some new $d\in\mathbb{R}$). If $d\neq0$ then $xy,x^3+dy^3\in I$ implies $Allow(I)=\emptyset$ which is a contradiction to $I$ being closed
by
Corollary \ref{cor_no_allowed_directions_SECOND_PAPER}. Otherwise $d=0$ (for any such $x^3+dy^3\in I$) and $xy^2,x^3\in I$ which implies that $Allow(I)\subset \{(x,y)=(0,\pm 1)\}$. Thus,  $\frac{x}{y}$
satisfies (\ref{second_cond_of_implied_poly_in_direction_SECOND_PAPER}) in
any allowed direction, so the jet
$x^2=xy\cdot\frac{x}{y}\in I$ (by Corollary
\ref{cor_strong_directional_implication_imply_implication_SECOND_PAPER}), which is a contradiction to the assumption $a=b=0$ for any jet $ax^2+by^2+cx^3+dy^3\in I$.

\

\textbf{B.3. Assume $x^2+q(x,y)\in I$.} We then have $x^3,x^2y\in I$, so we can assume $q(x,y)=axy^2+by^3$ and so $x^2+axy^2+by^3\in I$. In the coordinate system $(X,Y)=(x+\frac{1}{2}ay^2,y)$ we have $x^2+axy^2+by^3=X^{2}+bY^3$ (as 3-jets) and so we may assume $x^2+by^3\in I$ and still $x^3,x^2y\in I$. Since by assumption every jet in $I$ has order of vanishing at least 2, modulo $x^3,x^2y,x^2+by^3$ (that we  saw all belong to $I$) any jet in $I$ has the form $\alpha xy+\beta y^2+\gamma xy^2+\delta y^3$ for some $\alpha,\beta,\gamma,\delta\in\mathbb{R}$. If there exists a jet of this form in $I$ with $\alpha\neq0$ then we are in case B.2 that we already analyzed. So we may assume that modulo $x^3,x^2y,x^2+by^3$ any jet in $I$ has the form $\beta y^2+\gamma xy^2+\delta
y^3$ for some $\beta,\gamma,\delta\in\mathbb{R}$. If there exists
a jet of this form in $I$ with $\beta^2+\delta^2\neq0$ then  $Allow(I)=\emptyset$
which is a contradiction to $I$ being closed
by
Corollary \ref{cor_no_allowed_directions_SECOND_PAPER}. So modulo $x^3,x^2y,x^2+by^3$
any
jet in $I$ has the form $\gamma xy^2$ for some $\gamma\in\mathbb{R}$. We conclude that $I\in\{\langle x^2+by^3\rangle_2,\langle x^2+by^3,xy^{2}\rangle_2\}$. The first is principal and so excluded, and so $I=\langle x^2+by^3,xy^{2}\rangle_2$. If $b=0$ we are in case (\ref{non_pricipal_model_for_3_2_3_new_as_missed}). Otherwise, by scaling $y\mapsto-b^{1/3}y$ we have $I=\langle
x^2-y^{3},xy^{2}\rangle_3$, i.e., we are in case (\ref{non_pricipal_model_for_3_2_3})

\

\textbf{C. Assume that the minimal order of vanishing of jets in $I$ is 3.}
In that case $I$ is a linear subspace of $\text{span}_\mathbb{R}\{x^3,x^2y,xy^2,y^3\}$. As $I$ is not principal we have $\dim I\in\{2,3,4\}$.

\

\textbf{C.1. Assume $\dim I=4$.} In that case $x^{3},y^{3}\in I$ and so  $Allow(I)=\emptyset$
which is a contradiction to $I$ being closed
by
Corollary \ref{cor_no_allowed_directions_SECOND_PAPER}.

\

\textbf{C.2. Assume $\dim I=3$.} Since $I$ is closed, $Allow(I)\neq\emptyset$. Apply a linear coordinate change to assume $(x,y)=(0,1)\in Allow(I)$. Then, any jet in $I$ vanishes on $(0,1)$ and so $I=\langle x^3,x^2y,xy^2\rangle_3$, and we are in case (\ref{non_pricipal_model_for_3_2_4}).

\

\textbf{C.3. Assume $\dim I=2$.} Since $I$ is closed, $Allow(I)\neq\emptyset$.
Apply a linear coordinate change to assume $(x,y)=(0,1)\in Allow(I)$. Then,
any jet in $I$ vanishes on $(0,1)$ and so $I$ is a 2-dimensional subspace of $\text{span}_{\mathbb{R}}\{x^3,x^2y,xy^2\}$. In Remark \ref{remark_on_2_dim_subspaces_in_the_plane} we saw that any 2-dimensional subspace of $\text{span}_{\mathbb{R}}\{x^2,xy,y^2\}$
either contains a positive define form, has the form, up to a linear
coordinate change, $\text{span}_{\mathbb{R}}\{xy,x^{2}\}$ or has the form, up to a linear
coordinate change, $\text{span}_{\mathbb{R}}\{xy,(x+ y)(x- y)\}$. 

Thus, by applying a linear
coordinate change we may assume that either (i)
$I=\text{span} _\bR\{(ax+by)xy,(ax+by)(x-y)(x+ y)\}$ for some $a,b\in\bR$ such that $a^2+b^2\neq0$; (ii)
$I=\text{span} _\bR\{(ax+by)xy,(ax+by)x^{2}\}$ for  some $a,b\in\bR$ such that $a^2+b^2\neq0$; or (iii) $I$ contains a jet of the form $x\cdot q(x,y)$ for some positive definite
quadratic form $q(x,y)$. In case (iii) we have $x^3,x^2,xy^2\in I$ by Lemma \ref{lemma_linear_times_definite_in_closed_imply_many_jets}, which is a contradiction to $\dim I=2$. In case (ii) if $b=0$ then $I=\langle x^3,x^2y\rangle_3$ and we are in case (\ref{non_pricipal_model_for_3_2_5_a}); if $a=0$ then $I=\langle x^2y,xy^2\rangle_3$ and
we are in case (\ref{non_pricipal_model_for_3_2_5_b}). If $ab\neq0$ then by scaling $x$ and $y$ we may assume $I=\langle (x+y)xy,(x+y)x^{2}\rangle_3$. Now, in the coordinate system $(X,Y)=(x,x+y)$ we have $I=\langle YX(Y-X),YX^{2}\rangle_3=\langle X^2Y,Y^2X\rangle_3$ and again we are in case (\ref{non_pricipal_model_for_3_2_5_b}). In case (i), if $ab=0$ then by possibly renaming $x\leftrightarrow y$ we may assume $b=0$ and $I=\langle x^2y,x^3-xy^2\rangle_3$. Let us show that this ideal is not closed: indeed, $Allow(I)=\{(0,\pm1)\}$.
In cones in the directions $(0,\pm1)$ we have that $S_1(x,y):=\frac{x}{y}$
satisfies the bounds in (\ref{second_cond_of_implied_poly_in_direction_SECOND_PAPER})
in any allowed direction and so $x^3=(\frac{x}{y})\cdot(x^{2}y)$ is implied by $I$ (by Corollary
\ref{cor_strong_directional_implication_imply_implication_SECOND_PAPER}). Therefore, $x^{3}\in cl(I)\setminus I$
and $I$ is not closed. Otherwise, $ab\neq 0$, and $Allow(I)=\{ax+by=0\}\cap S^1$. So, by a similar argument we now have $(ax+by)xy\cdot\frac{ax+by}{y}=(ax+by)^{2}x\in I$ and also $(ax+by)xy\cdot\frac{ax+by}{x}=(ax+by)^{2}y\in
I$. Since $\dim I=2$ we have $I=\langle (ax+by)^3,(ax+by)^2x\rangle_3$, so in the coordinate system $(u,v)=(ax+by,x)$ we have $I=\langle u^2v,u^3\rangle_3$, and we are again
in case (\ref{non_pricipal_model_for_3_2_5_a}). \end{proof}

\begin{remark}\label{remark_on_distinctness_for_3_2_case_non_principal}The ideals
(\ref{non_pricipal_model_for_3_2_1})-(\ref{non_pricipal_model_for_3_2_5_b}) are pairwise distinct, i.e.,
there does not exist a $C^3$ coordinate change around the origin such that
a given ideal can take two different forms from this list. (\ref{non_pricipal_model_for_3_2_1})
is obviously distinct, and (\ref{non_pricipal_model_for_3_2_4}) is distinct too as it
is the only ideal in this list that has dimension 3. (\ref{non_pricipal_model_for_3_2_2}) is the only ideal that is generated by jets of order of vanishing 2, so it is also distinct. (\ref{non_pricipal_model_for_3_2_3_new_as_missed}) and (\ref{non_pricipal_model_for_3_2_3}) are the only remaining ideals with a jet of order of vanishing 2, however the first contains a square of a jet of order of vanishing 1, while the latter does not, and so they are distinct. Finally, (\ref{non_pricipal_model_for_3_2_5_b}) and (\ref{non_pricipal_model_for_3_2_5_a}) are distinct as the first has two linearly independent allowed directions, while the latter does not.\end{remark}

\begin{proposition}\label{prop_C3R2_case_new_version} Let $I\lhd\mathcal{P}_0^3(\mathbb{R}^{2})$
be an ideal. Then, the following are equivalent:
\begin{itemize}
\item $I$ is closed.
\item There exists a closed set $\vec 0 \in E\subset\mathbb{R}^2$ such that
$I=I^3(E)$.
\item There exists a closed semi-algebraic set $\vec 0 \in E\subset\mathbb{R}^2$
such that $I=I^3(E)$.
\end{itemize}\end{proposition}
\begin{proof}
By Theorem \ref{main_theorem_on_necessary_condition_SECOND_PAPER}, Lemma \ref{lemma_classify_closed_principal_in_C3R2_case_new_version} and Lemma \ref{lemma_classify_closed_non_principal_in_C3R2_case_new_version} it is enough to show
that each of the ideals in (\ref{model_for_3_2_7}), (\ref{model_for_3_2_9}), (\ref{model_for_3_2_10}), (\ref{non_pricipal_model_for_3_2_3_new_as_missed}), (\ref{non_pricipal_model_for_3_2_3}), (\ref{non_pricipal_model_for_3_2_4}), (\ref{non_pricipal_model_for_3_2_5_a}) and (\ref{non_pricipal_model_for_3_2_5_b}) arises as $I^3(E)$ for some closed semi-algebraic set $E\subset\mathbb{R}^2$ containing the origin. This is established in Examples \ref{example_3_jets_x^2-y^3_principal}, \ref{first-example-in-P3},\ref{second-example-in-P3}, \ref{example_xx_xyy}, \ref{example_xx-yyy_xyy_in_P3R2},\ref{xxx-xxy-xyy-example-in-P3}, \ref{example_lxx_lxy_in_P3R2_a} and \ref{example_lxx_lxy_in_P3R2_b} below.
\end{proof}

\begin{example}\label{example_3_jets_x^2-y^3_principal}$\langle
x^2-y^3 \rangle_3=I^3(\{x^2=y^3\})$.

Indeed, first note that clearly $x^2-y^3\in I^3(\{x^2=y^3\})$ , and so also $x^3,x^2y\in I^3(\{x^2=y^3\})$. Let $f\in C^3(\bR^2)$ be such that $f|_{\{x^2=y^3\}}=0$, and let $J^3(f)=p(x,y)$. We can then write $f(x,y)=p(x,y)+o((x^2+y^2)^{3/2})$.
We need to show that $p(x,y)\in\langle x^2-y^{3}\rangle_3 $. So it is enough to show that if $$p(x,y)=ax+by+cx^2+dxy+ey^2+gxy^2$$ for
some $a,b,c,d,e,g\in\bR$, then $a=b=c=d=e=g=0$. First note that for any $t\in\bR$
we have $$0=f(t^3,t^2)=at^3+bt^2+ct^6+dt^5+et^4+gt^7+o(t^6),$$which implies
$a=b=c=d=e=0$. So $f(x,y)=gxy^2+o((x^2+y^2)^{3/2})$, and
we need to show that $g=0$. Note that for
any $\epsilon>0$ we have $f(\epsilon^{3/2},\epsilon)=f(-\epsilon^{3/2},\epsilon)=0$,
 and so by Rolle's Theorem there exists $-\epsilon^{3/2}<\delta(\epsilon)<\epsilon^{3/2}$
such that $f_{x}(\delta(\epsilon),\epsilon)=0$. On the other hand we
have
$f_{x}(x,y)=gy^{2}+o(x^2+y^2)$.
As $\abs{\delta(\epsilon)}<\epsilon^{3/2}$, on the set $\{(\delta(\epsilon),\epsilon)\}_{\epsilon>0}$
we have $$0=f_{x}(\delta(\epsilon),\epsilon)=g\epsilon^2+o(\delta(\epsilon)^2+\epsilon^2)=g\epsilon^2+o(\epsilon^2),$$which
is only possible if $g=0$. We conclude that indeed $\langle
x^2-y^3 \rangle_3=I^3(\{x^2=y^3\})$.\end{example}

\begin{example}\label{first-example-in-P3}$\langle
yx^2 \rangle_3=I^3(\{y(x^2-y^4)=0\})$.

Indeed, first note that $y(x^2-y^4)\in C^3(\bR^2)$ vanishes on $\{y(x^2-y^4)=0\}$,
and so we have $yx^2\in I^3(\{y(x^2-y^4)=0\})$. Recall that $I^2(\{y=0\})=\langle
y \rangle_2$ (Corollary \ref{cor-on-prime-with-a-linear-part}), and moreover
$I^2(\{x^2=y^4\})=\langle x^2 \rangle_2$ (Example \ref{non-real-radical-example}).
Thus, the 2-jet of any jet in $I^3(\{y(x^2-y^4)=0\})$ must be divisible by
both $y$ and by $x^2$, i.e., must be zero. We conclude that any non-zero
jet in $I^3(\{y(x^2-y^4)=0\})$ is of order of vanishing 3. It is thus enough to show that if $f\in C^3(\bR^2)$ is such that $f(x,y)=ax^3+bxy^2+cy^3+o((x^2+y^2)^{3/2})$ for some $a,b,c\in\mathbb{R}$ and $f|_{\{y(x^2-y^4)=0\}}=0$, then $a=b=c=0$. Let $f$ be as assumed. 
For any $x\in \bR$ we have $$0=f(x,0)=ax^3+o(x^3),$$ thus $a=0$ and so $f(x,y)=bxy^2+cy^3+((x^2+y^2)^{3/2})$.
Assume towards a contradiction that $c\neq0$. Then, $(x,y)=(0,1)\in Forb(I)$. However, clearly $(x,y)=(0,1)\in T(\{y(x^2-y^4)=0\})$, which is a contradiction by Lemma \ref{tangent_subset_bad_lemma_SECOND_PAPER}. Therefore, $f(x,y)=bxy^2+((x^2+y^2)^{3/2})$
and we need to show that $b=0$. Note that for any $\epsilon>0$ we have $f(\epsilon^2,\epsilon)=f(-\epsilon^2,\epsilon)=0$.
By Rolle's Theorem there exists $-\epsilon^2<\delta(\epsilon)<\epsilon^2$
such that $f_x(\delta(\epsilon),\epsilon)=0$. On the other hand we have $f_x(x,y)=by^2+o(x^2+y^2)$.
As $\abs{\delta(\epsilon)}<\epsilon^2$, on the set $\{(\delta(\epsilon),\epsilon)\}_{\epsilon>0}$
we have: $$0=f_x(\delta(\epsilon),\epsilon)=b\epsilon^2+o(\delta(\epsilon)^2+\epsilon^2)=b\epsilon^2+o(\epsilon^2),$$which
is only possible if $b=0$. We conclude that indeed $\langle
yx^2 \rangle_3=I^3(\{y(x^2-y^4)=0\})$. \end{example}

\begin{example}\label{second-example-in-P3}
$\langle
x^3 \rangle_3=I^3(\{x(x^2-y^4)=0\})$.

Indeed, first note that $x(x^2-y^4)\in C^3(\bR^2)$ vanishes on $\{x(x^2-y^4)=0\}$,
and so we have $x^3\in I^3(\{x(x^2-y^4)=0\})$. Let $f\in C^3(\bR^2)$ be such that
$f_{\{x(x^2-y^4)=0\}}=0$. In particular $f_{\{x^2-y^4=0\}}=0$ and so by Example
\ref{non-real-radical-example} the 2-jet of $f$ is divisible by $x^2$. Thus,
it is enough to show that if $$f(x,y)=ax^2y+bxy^2+cy^3+dx^2+o((x^2+y^2)^{3/2})$$
for some $a,b,c,d\in\bR$, then $a=b=c=d=0$. Let $f$ be as assumed. For any $y\in \bR$ we have $$0=f(0,y)=cy^3+o(y^3),$$
thus $c=0$ and $f(x,y)=ax^2y+bxy^2+dx^2+o((x^2+y^2)^{3/2})$. Note that for
any $\epsilon>0$ we have $f(\epsilon^2,\epsilon)=f(-\epsilon^2,\epsilon)=f(0,\epsilon)=0$.
By applying Rolle's Theorem twice we find that there exists $-\epsilon^2<\delta(\epsilon)<\epsilon^2$
such that $f_{xx}(\delta(\epsilon),\epsilon)=0$. On the other hand we found that
$f_{xx}(x,y)=2ay+2d+o((x^2+y^2)^{1/2})$.
As $\abs{\delta(\epsilon)}<\epsilon^2$, on the set $\{(\delta(\epsilon),\epsilon)\}_{\epsilon>0}$
we have: $$0=f_{xx}(\delta(\epsilon),\epsilon)=2a\epsilon+2d+o(\delta(\epsilon)^2+\epsilon^2)^{1/2}=2a\epsilon+2d+o(\epsilon),$$which
is only possible if $a=d=0$ and so $f(x,y)=bxy^2+o((x^2+y^2)^{3/2})$. Similarly,
as for
any $\epsilon>0$ we have $f(\epsilon^2,\epsilon)=f(-\epsilon^2,\epsilon)=0$,
 by Rolle's theorem there exists $-\epsilon^2<\tilde\delta(\epsilon)<\epsilon^2$
such that $f_{x}(\tilde\delta(\epsilon),\epsilon)=0$. On the other hand we
have
$f_{x}(x,y)=by^{2}+o(x^2+y^2)$.
As $\abs{\tilde\delta(\epsilon)}<\epsilon^2$, on the set $\{(\tilde\delta(\epsilon),\epsilon)\}_{\epsilon>0}$
we have: $$0=f_{x}(\tilde\delta(\epsilon),\epsilon)=b\epsilon^2+o(\tilde\delta(\epsilon)^2+\epsilon^2)=b\epsilon^2+o(\epsilon^2),$$which
is only possible if $b=0$. We conclude that indeed $\langle
x^3 \rangle_3=I^3(\{x(x^2-y^4)=0\})$. \end{example}

\begin{example}\label{example_xx_xyy}$\langle x^2,xy^2\rangle_3=I^3(\{(x=\abs{y}^{7/4})\wedge y>0\})$.

Indeed, denote $E:=\{(x=\abs{y}^{7/4})\wedge
y>0\}$.
Then, the functions
$$y^{2}(x-\abs{y}^{7/4}),x^2-\abs{y}^{7/2}\in
C^3(\bR^2)$$ both vanish on $E$, and so $x^2,xy^{2},x^{2}y,x^{3}\in\ I^3(E)$. So it is enough to show that if $f(x,y)\in C^3(\bR^2)$
has the form $$f(x,y)=ax+by+cxy+dy^2+ey^3+o((x^2+y^2)^{3/2})$$ for some
$a,b,c,d,e\in\bR$
and $f|_E=0$, then $a=b=c=d=e=0$. Let $f$ be as assumed. For any $t>0$ we have $$0=f(t^{7/4},t)=at^{7/4}+bt+ct^{11/4}+dt^{2}+et^3+o((t^{7/2}+t^2)^{3/2})$$ $$=at^{7/4}+bt+ct^{11/4}+dt^{2}+et^3+o(t^{3}),$$which
is only possible if $a=b=c=d=e=0$, so indeed $\langle x^2,xy^2\rangle_3=I^3(\{(x=\abs{y}^{7/4})\wedge
y>0\})$.
\end{example}

\begin{example}\label{example_xx-yyy_xyy_in_P3R2}$\langle
x^2-y^{3},xy^{2}\rangle_3=I^3(\{x^2=y^{3}\}\cap\{x\geq0\})$.

Indeed, denote $E:=\{x^2=y^{3}\}\cap\{x\geq0\}$ and recall that by Example
\ref{example-xx-xy} we have that $I^2(\{x^2=y^{3}\}\cap\{x\geq0\})=\langle
x^2,xy\rangle_2$. As any $C^3$-smooth function is also $C^2$-smooth, we have
that any jet in $I^3(E)$ is of the form $$ax^2+bxy+cx^3+dx^2y+exy^2+gy^3,$$
for some $a,b,c,d,e,g\in\bR$. Also note that the functions $$x^2-y^3,y^2(x-\abs{y}^{3/2})\in
C^3(\bR^2)$$ both vanish on $E$ and so $x^2-y^3,xy^2\in\ I^3(E)$ and also
$x^3,x^2y\in I^3(E)$. So it is enough to show that if $f(x,y)\in C^3(\bR^2)$
has the form $$f(x,y)=bxy+gy^{3}+o((x^2+y^2)^{3/2})$$ for some $b,g\in\bR$
and $f|_E=0$, then $b=g=0$. Let $f$ be as assumed. For any $t>0$ we have $$0=f(t^3,t^2)=bt^5+gt^6+o((t^6+t^4)^{3/2})=bt^5+gt^6+o(t^{6}),$$which
is only possible if $b=g=0$, and indeed $\langle
x^2-y^{3},xy^{2}\rangle_3=I^3(\{x^2=y^{3}\}\cap\{x\geq0\})$.\end{example}

\begin{example}\label{xxx-xxy-xyy-example-in-P3} Set $\epsilon=10^{-9}$.
Then, $\langle x^3,x^2y,xy^2 \rangle_3=I^3(\{x(x^2+y^2)-\abs{y}^{3+\epsilon}=0\})$.

Indeed, we first denote $E:=\{x(x^2+y^2)-\abs{y}^{3+\epsilon}=0\}$. Let us analyze the curve
$E$ near the origin.
Whenever $y\neq0$ we may write $x=\abs{y}^{1+\epsilon}\cdot z$, for some
function $z(x,y)$. Substituting for $x$ we get $$\abs{y}^{1+\epsilon}\cdot
z\cdot(\abs{y}^{2+2\epsilon}\cdot z^2+y^2)-\abs{y}^{3+\epsilon}=0,$$ which
implies by dividing by $\abs{y}^{3+\epsilon}$ $$z+\abs{y}^{2\epsilon}z^3-1=0,$$and
if we define $w:=\abs{y}^{2\epsilon}$ we get $$z+wz^3-1=0.$$ The implicit
function theorem and the fact that $(z,w)=(1,0)$ is a solution of this equation
now imply that in a neighborhood of $(z,w)=(1,0)$ we have $$z=1+w\cdot\varphi(w),$$where $\varphi(w)$ is a real analytic
function. Substituting back we get $$x=\abs{y}^{1+\epsilon}(1+\abs{y}^{2\epsilon}\varphi(\abs{y}^{2\epsilon})).$$
We conclude that \begin{multline}\label{rand_label_on_surce_example_in_C3R2}\text{for any }y\neq0\text{ (small enough) there exists }\phi(y)\in\mathbb{R},\\\text{such that }(\phi(y),y)\in
E\text{ and }\phi(y)=\Theta(\abs{y}^{1+\epsilon}).\end{multline}
As  $x(x^2+y^2)-\abs{y}^{3+\epsilon}\in C^3(\bR^2)$ vanishes on the set $E$,
we have  $x(x^2+y^2)\in I^3(E)$. By Lemma \ref{lemma_linear_times_definite_in_closed_imply_many_jets} and Theorem \ref{main_theorem_on_necessary_condition_SECOND_PAPER} we also have $x^3,x^2y,xy^2\in I^3(E)$. It
is thus enough to show that if $f\in C^3(\bR^3)$ is such that $f|_E=0$ and
$$f(x,y)=ay^3+bx^2+cxy+dy^2+ex+gy+o((x^2+y^2)^{3/2})$$ for some
$a,b,c,d,e,g\in\bR$, then $a=b=c=d=e=g=0$. Let $f$ be as assumed. By (\ref{rand_label_on_surce_example_in_C3R2}), for any small enough $y\neq0$
we have $$0=f(\phi(y),y)=ay^3+b\phi(y)^2+c\phi(y)y+dy^2+e\phi(y)+gy+o((\phi(y)^2+y^2)^{3/2})$$
$$=ay^3+b\Theta(\abs{y}^{1+\epsilon})^2+c\Theta(\abs{y}^{1+\epsilon})y+dy^2+e\Theta(\abs{y}^{1+\epsilon})+gy+o((\Theta(\abs{y}^{1+\epsilon})^2+y^2)^{3/2})$$
$$=ay^3+b\Theta(\abs{y}^{2+2\epsilon})+c\Theta(\abs{y}^{2+\epsilon})+dy^2+e\Theta(\abs{y}^{1+\epsilon})+gy+o(y^3),$$
which implies $a=b=c=d=e=g=0$, and so $\langle x^3,x^2y,xy^2 \rangle_3=I^3(\{x(x^2+y^2)-\abs{y}^{3+\epsilon}=0\})$.
\end{example}

\begin{example}\label{example_lxx_lxy_in_P3R2_a}$\langle
x^3,x^{2}y\rangle_3=I^3(\{x(x^2-y^{3})=0\}\cap\{x\geq0\})$.

Indeed, denote $E:=\{x(x^2-y^{3})=0\}\cap\{x\geq0\}$ and recall that by Example
\ref{example-xx} we have $I^2(E)=\langle x^2\rangle_2$. As any $C^3$-smooth
function is also $C^2$-smooth, we have
that any jet in $I^3(E)$ is of the form $$ax^2+bxy^{2}+cx^2y+dy^3+ex^{3},$$
for some $a,b,c,d,e\in\bR$. Let $\tilde \theta(x,y)$ be as in (\ref{theta_nice_usful_function}). Then, the functions
$$x(x^2-y^3),xy(x-|y|^{3/2}\cdot \tilde\theta(x,y))\in
C^3(\bR^2)$$ both vanish on $E$ (see footnote \ref{footnote_on_function_vanishes_on_a_set}), and so $x^3,x^{2}y\in\ I^3(E)$. So it is enough to show that if $f(x,y)\in C^3(\bR^2)$
has the form $$f(x,y)=ax^{2}+bxy^{2}+dy^{3}+o((x^2+y^2)^{3/2})$$ for some
$a,b,d\in\bR$
and $f|_E=0$, then $a=b=d=0$. Let $f$ be as assumed. For any $y>0$ we have $$0=f(0,y)=dy^3+o((0+y^2)^{3/2})=dy^3+o(y^{3}),$$which
is only possible if $d=0$, and therefore $f(x,y)=ax^2+bxy^2+o((x^2+y^2)^{3/2})$.
Now, for any $t>0$ we have $$0=f(t^3,t^2)=at^6+bt^7+o((t^6+t^4)^{3/2})=at^6+bt^7+o(t^{6}),$$which
is only possible if $a=0$, and therefore $f(x,y)=bxy^2+o((x^2+y^2)^{3/2})$.
As for any $\epsilon>0$ we have $f(0,\epsilon)=f(\epsilon^{3/2},\epsilon)=0$,
by Rolle's Theorem there exists $0<\delta(\epsilon)<\epsilon^{3/2}$
such that $f_{x}(\delta(\epsilon),\epsilon)=0$. On the other hand we
have $f_x(x,y)=by^{2}+o(x^2+y^2)$. So in particular for any $\epsilon>0$
we have $$0=f_x(\delta(\epsilon),\epsilon)=b\epsilon^2+o(\delta(\epsilon)^2+\epsilon^2)=b\epsilon^2+o(\epsilon^2),$$which
is only possible if $b=0$, and indeed $\langle
x^3,x^{2}y\rangle_3=I^3(\{x(x^2-y^{3})=0\}\cap\{x\geq0\})$.
\end{example}

\begin{example}\label{example_lxx_lxy_in_P3R2_b}$\langle
x^2y,xy^{2}\rangle_3=I^3(\{y(x^2-y^{3})=0\}\cap\{x\geq0\})$.

Indeed, denote  $E:=\{y(x^2-y^{3})=0\}\cap\{x\geq0\}$.
As $$E\supset \{x^2=y^{3}\}\cap\{x\geq0\},$$ we have \begin{gather}\label{rand_label_for_last_example_in_C3R2_numer_1}I^3(E)\subset
I^3(\{x^2=y^{3}\}\cap\{x\geq0\})=\langle
x^2-y^{3},xy^{2}\rangle_3,\end{gather}where the last equality follows from
Example
\ref{example_xx-yyy_xyy_in_P3R2}. An easy modification of the proof of Corollary \ref{cor-on-prime-with-a-linear-part} shows that \begin{gather}\label{rand_label_for_last_example_in_C3R2_numer_2}I^3(\{y=0\}\cap\{x\geq0\})=\langle y\rangle_3.\end{gather}  Since
$$E\supset \{y=0\}\cap\{x\geq0\},$$ we have by (\ref{rand_label_for_last_example_in_C3R2_numer_2})
that \begin{gather}\label{rand_label_for_last_example_in_C3R2_numer_3}I^3(E)\subset
I^3(\{y=0\}\cap\{x\geq0\})=\langle y\rangle_3.\end{gather} Now (\ref{rand_label_for_last_example_in_C3R2_numer_1})
and (\ref{rand_label_for_last_example_in_C3R2_numer_3}) imply that $$I^3(E)\subset\langle
x^2-y^{3},xy^{2}\rangle_3\cap\langle y\rangle_3=\langle
x^2y,xy^{2}\rangle_3.$$ 
The functions
$$y(x^2-y^3),y^{2}(x-|y|^{3/2})\in
C^3(\bR^2)$$ both vanish on $E$, and so indeed $I^3(E)=\langle
x^2y,xy^{2}\rangle_3$.\end{example}

\section{The $C^2(\mathbb{R}^3)$ case}\label{section_2_3}

\begin{lemma}\label{lemma_C2R3_case__with_jet_of_oov_1} Let $I\lhd\mathcal{P}_0^2(\mathbb{R}^{3})$
be an ideal, and assume $I$ contains a jet of order of vanishing 1. Then, the following are equivalent:
\begin{itemize}
\item $I$ is closed.
\item There exists a closed set $\vec 0 \in E\subset\mathbb{R}^3$ such that
$I=I^2(E)$.
\item There exists a closed semi-algebraic set $\vec 0 \in E\subset\mathbb{R}^3$
such that $I=I^2(E)$.
\end{itemize}\end{lemma}
\begin{proof}
Let $I\lhd\mathcal{P}_0^2(\mathbb{R}^{3})$ be an ideal that contains a jet of order
of vanishing 1. If $I$ is principal then the proposition is a special case of Proposition \ref{prop_principal_ideals_in_P2Rn} and we are done. So we assume $I$ is not principal. By Theorem \ref{main_theorem_on_necessary_condition_SECOND_PAPER} it is enough to show that if $I$ is closed then there exists a closed semi-algebraic set $\vec 0 \in E\subset\mathbb{R}^3$
such that $I=I^3(E)$.  So we assume $I$ is closed. If $I=\mathcal{P}_0^2(\mathbb{R}^3)=I^3(\{\vec 0\})$ we are done. We thus further assume from now on that $I\neq \mathcal{P}_0^2(\mathbb{R}^3)$ and so $Allow(I)\neq\emptyset$ by
Corollary \ref{cor_no_allowed_directions_SECOND_PAPER}.

By applying a $C^2$ semi-algebraic coordinate change around the origin we may assume $x\in I$, where $x,y,z$ is a (standard) coordinate system of $\mathbb{R}^3$. So we can write \begin{gather}\label{rand_label1_lemma_C2R3_case__with_jet_of_oov_1}I=\langle p_1,p_2,\dots p_l,x\rangle_2\text{ for some }p_1(y,z),p_2(y,z),\dots p_l(y,z).\end{gather} In particular $\omega=(\omega_x,\omega_y,\omega_z)\in Allow(I)$ only if $\omega_x=0$. As $Allow(I)\neq\emptyset$ we conclude that there exists $\omega\in Allow(I)\cap\{y-z\text{ plane}\}$, or explicitly
\begin{multline}\label{rand_label2_lemma_C2R3_case__with_jet_of_oov_1}\text{there exists }\omega\in\{y-z\text{ plane}\}\text{ such that } \text{there does not exist }c,\delta,r>0 \text{
such that} \\ \abs{p_1(x,y,z)}+\abs{p_2(x,y,z)}+\dots+\abs{p_l(x,y,z)}>
c\cdot(x^2+y^2+z^2)^{m/2}\\ \text { for all }(x,y,z)\in\Gamma(\omega,\delta,r).\end{multline}
Define $I'\lhd\mathcal{P}_0^2(\mathbb{R}^2)$ by $I':=\langle p_1,\dots p_l\rangle_2$, where we identify $\mathbb{R}^2$ with the $y-z$ plane. Note that by our assumption $I'\neq \mathcal{P}_0^2(\mathbb{R}^2)$, as $I\neq \mathcal{P}_0^2(\mathbb{R}^3)$. Recalling that $p_1,\dots,p_l$ do not depend on the $x$ coordinate, (\ref{rand_label2_lemma_C2R3_case__with_jet_of_oov_1}) implies that $Allow(I')\neq\emptyset$. So by Remark \ref{remark_in_C2R2_enough_to_have_allowed_direction} and Proposition \ref{prop_C2R2_case_new_version} there exists a semi-algebraic set $E'\subset\mathbb{R}^2$ such that \begin{gather}\label{rand_label3_lemma_C2R3_case__with_jet_of_oov_1}I'=I^2(E')=\langle p_1,\dots p_l\rangle_2\lhd\mathcal{P}_0^2(\mathbb{R}^2).\end{gather}Finally, defining the semi-algebraic set \begin{gather}\label{rand_label4_lemma_C2R3_case__with_jet_of_oov_1}E:=\{(x,y,z)\in\mathbb{R}^3|(y,z)\in E', x=0\},\end{gather}we get by (\ref{rand_label1_lemma_C2R3_case__with_jet_of_oov_1}), (\ref{rand_label3_lemma_C2R3_case__with_jet_of_oov_1}), (\ref{rand_label4_lemma_C2R3_case__with_jet_of_oov_1}) and Lemma \ref{lemma_induction_adding_a_coordinate_jet} that $I=I^2(E)$.\end{proof}

\begin{remark}The proofs of Proposition \ref{prop_C2R2_case_new_version} and of Lemma \ref{lemma_C2R3_case__with_jet_of_oov_1} easily imply that any closed ideal in $\mathcal{P}_0^2(\mathbb{R}^{3})$
that contains a jet of order of vanishing 1 is, up to a $C^2$ semi-algebraic coordinate change around the origin, one of the ideals (\ref{relative_of_model_for_2_2_1_prime})-(\ref{relative_of_model_for_2_2_6}) below. Moreover, these ideals are pairwise distinct, i.e.,
there does not exist a $C^2$ coordinate change around the origin such that
a given ideal can take two different forms from this list: (\ref{relative_of_model_for_2_2_1_prime}) and (\ref{relative_of_model_for_2_2_2})
are obviously different from the rest. (\ref{relative_of_model_for_2_2_6}) is distinct
as it is the only remaining ideal of dimension 6. (\ref{relative_of_model_for_2_2_3})
is a also
distinct as it is the only remaining ideal that contains 2 linearly independent jets of order of vanishing 1. Finally, (\ref{relative_of_model_for_2_2_4}) and (\ref{relative_of_model_for_2_2_5}) are distinct as the first does not have two linearly
independent allowed directions, while the latter does.

\begin{gather}\label{relative_of_model_for_2_2_1_prime}\langle x\rangle_2\end{gather}
\begin{gather}\label{relative_of_model_for_2_2_2}\mathcal{P}_0^2(\mathbb{R}^3)\end{gather}
\begin{gather}\label{relative_of_model_for_2_2_3}\langle x, y\rangle_2 \end{gather}
\begin{gather}\label{relative_of_model_for_2_2_4}\langle x, y^2\rangle_2\end{gather}
\begin{gather}\label{relative_of_model_for_2_2_5}\langle x, yz\rangle_2
\end{gather}
\begin{gather}\label{relative_of_model_for_2_2_6}\langle x, y^{2},yz\rangle_2
\end{gather}

 \end{remark}

\begin{lemma}\label{lemma_classify_closed_ideals_in_C2R3_1_dim_allow}Let $I\lhd\mathcal{P}_0^2(\mathbb{R}^{3})$
be a closed non-principal ideal. Assume $I$ does not contain any jet of order of vanishing 1 and that $\dim\text{span}_{\mathbb{R}}Allow(I)=1$. Then, $I$ is, up to a
linear
coordinate change, one of the following:

\begin{gather}\label{model_for_2_3_with_1_dim_allow_1}\langle x^2,y^2\rangle_2\end{gather}
\begin{gather}\label{model_for_2_3_with_1_dim_allow_2} \langle xy,y^{2} -x^{2}\rangle_2 \end{gather}
\begin{gather}\label{model_for_2_3_with_1_dim_allow_3}\langle
x^2,y^2,xy\rangle_2\end{gather}
\begin{gather}\label{model_for_2_3_with_1_dim_allow_3'_the_one_we_missed}\langle
x^2,y^2+xz,xy\rangle_2\end{gather}
\begin{gather}\label{model_for_2_3_with_1_dim_allow_4}\langle
x^2,y^2,xy,xz\rangle_2\end{gather}
\begin{gather}\label{model_for_2_3_with_1_dim_allow_5}\langle
x^2,y^2,xy,xz,yz\rangle_2\end{gather}

\end{lemma}

\begin{proof}As $\dim\text{span}_{\mathbb{R}}Allow(I)=1$, by applying a linear
coordinate change we may assume $Allow(I)=\{(0,0,\pm1)\}$ and
so $I$ is a linear subspace of $\text{span}_{\mathbb{R}}\{x^2,y^{2},xy,yz,xz\}$. As $I$ is not principal we have $\dim I\in\{2,3,4,5\}$. If $\dim I=5$ we are in case (\ref{model_for_2_3_with_1_dim_allow_5}).
We analyze the cases $\dim I\in\{2,3,4\}$ separately.  

\

\textbf{\uline{Assume $\dim I=2$.}} We start with the following observation: $I$ does not contain any jet of the form $(Ax+By)(\tilde Ax+\tilde By+\tilde Cz)$ with $A^2+B^2\neq0$ and $\tilde C\neq0$. Indeed, if $I$ contains such a jet, recall that $I$ is closed and $Allow(I)=(0,0,\pm1)$, and we get from Corollary \ref{cor_strong_directional_implication_imply_implication_SECOND_PAPER} that $$\tilde p\cdot(Ax+By)=\frac{\tilde p}{\tilde Ax+\tilde By+\tilde Cz}\cdot(Ax+By)(\tilde Ax+\tilde By+\tilde Cz)\in I$$ for any $\tilde p\in\{x,y\}$. This implies that $\dim I\geq3$, which is a contradiction, so we established the observation.

\

We also establish the following:\begin{multline}\label{new_small_observation_to_a_less_than_one_etc}\text{If }I=\langle x(x+y),y(y+ax)\rangle_2\text{ for some }a\in\mathbb{R} \text{ then } I\\ \begin{cases}
  \text{is up to a linear coordinate change of the form (\ref{model_for_2_3_with_1_dim_allow_2})}  & \text{if } a<1 \\
  \text{does not satisfy } \dim Allow(I)=1 & \text{if } a=1 \\
  \text{is up to a linear coordinate change of the form (\ref{model_for_2_3_with_1_dim_allow_1})} & \text{if } a>1 \end{cases}.\end{multline}We apply the coordinate change $(X,Y)=(x,x+y)$ and so $I=\langle XY,(Y-X)(Y-X+aX)\rangle_2=\langle XY, Y^2+(a-1)X^2\rangle_2$. If $a=1$ then by Corollary \ref{old_new_cor_on_how_to_calc_allow_with_inclusion_SECOND_PAPER} we have $(0,0,1),(1,0,0)\in Allow(I)$, and so (\ref{new_small_observation_to_a_less_than_one_etc}) holds. If $a>1$ by scaling $X$ we have $I=\langle XY,X^2+Y^2\rangle_2$. Now, in the coordinate system $(\tilde X,\tilde Y)=(X-Y,X+Y)$ we have $I=\langle \tilde X^2, \tilde Y^2 \rangle_2$, so again (\ref{new_small_observation_to_a_less_than_one_etc}) holds. Finally, if $a<1$ by scaling $X$ we have $I=\langle XY,Y^{2}-X^2\rangle_2$ and again (\ref{new_small_observation_to_a_less_than_one_etc})
holds. 

\

We now go back to our classification. We have $$I=\langle ax^2+by^2+m_1,a'x^2+b'y^2+m_2 \rangle_2,$$where $a,b,a',b'\in\mathbb{R}$ and $m_1,m_2\in\text{span}_{\mathbb{R}}\{xy,yz,xz\}$. Since $Allow(I)=\{(0,0,\pm1)\}$ without loss of generality $a\neq0$ and so by scaling $x$ we get $$I=\langle x^2+by^2+m_1,b'y^2+m_2 \rangle_2,$$with new $b,b'\in\mathbb{R}$ and $m_1,m_2\in\text{span}_{\mathbb{R}}\{xy,yz,xz\}$. We separately  analyze the cases (A) $b'\neq0$; (B) $b'=0$. We start with the latter.
 
\

\textbf{Case (B).} Assume $b'=0$. In this case we must have $b\neq0$ as $Allow(I)=\{(0,0,\pm1)\}$. We will later refer to the sub-case in which $b>0$ as (B$^+$) and the sub-case in which $b<0$ as (B$^-$). By scaling $y$ we have $$I=\langle x^2\pm y^2+axy+bxz+cyz,a'xy+b'xz+c'yz \rangle_2,$$ with some new $a,b,c,a',b',c'\in\mathbb{R}$ such that $a'^2+b'^2+c'^2\neq0$. If $a'=0$ then $z(b'x+c'y)\in I$ which is impossible thanks to the observation above. So we may assume $a'\neq0$, and then by multiplying the second generator by $a'^{-1}$ and replacing the first generator we have $$I=\langle x^2\pm y^2+z(bx+cy),xy+z(b'x+c'y)
\rangle_2,$$with some new $b,c,b',c'\in\mathbb{R}$. Note that if $b'=0$ then $y(x+c'z)\in I$, which implies (by the observation above) that $c'=0$. Vise-versa, if $c'=0$ then
$x(y+b'z)\in I$, which implies (by the observation above) that $b'=0$. So we either have $b'=c'=0$ or $b'c'\neq0$. 

\

\textbf{Case (B).1.} Assume $b'=c'=0$. In that case we have $I=\langle x^2\pm y^2+z(bx+cy),xy
\rangle_2$. If $b\neq 0$ then any
point
in $S^2$ that satisfies $(y=0\text{ } \&\text{ } x=-bz)$
lies in $Allow(I)$,
which is a contradiction to our assumption that $\dim\text{span}_{\mathbb{R}}Allow(I)=1$.
 If $c\neq 0$ then any
point
in $S^2$ that satisfies $(x=0\text{ } \&\text{ } y=\mp cz)$
lies in $Allow(I)$,
which is a contradiction to our assumption that $\dim\text{span}_{\mathbb{R}}Allow(I)=1$. Therefore, $b=c=0$ and $I=\langle x^2\pm
y^2,xy
\rangle_2$. The "minus" case is case (\ref{model_for_2_3_with_1_dim_allow_2}). In the "plus" case we have  $I=\langle x^2+
y^2,xy
\rangle_2=\langle x^2+y^2+2xy,x^2+y^2-2xy\rangle_2=\langle (x+y)^2,(x-y)^2\rangle_2$, so in the coordinate system $(X,Y)=(x+y,x-y)$ we have $I=\langle X^2,Y^2\rangle_2$ and we are in
case (\ref{model_for_2_3_with_1_dim_allow_1}). Case (B).1. is complete.

\

\textbf{Case (B).2.} Assume $b'c'\neq0$. In that case we scale $z$ and get $$I=\langle x^2\pm y^2+z(bx+cy),xy+z(x+c'y)
\rangle_2,$$ with some new $b,c,c'\in\mathbb{R}$ such that $c'\neq0$. 

\

\textbf{Case (B).2.1.} Assume $c\neq0$. In that case we write explicitly $$I=\langle x^2+sy^2+z(bx+cy),xy+z(x+c'y)\rangle_2,\text{ where }c\neq0,c'\neq0\text{ and }s\in\{\pm1\}.$$Note that by applying the linear coordinate change
$x\mapsto-x$ we get $\begin{cases}
  s\mapsto s \\ b\mapsto-b \\ c\mapsto c \\ c'\mapsto -c'
\end{cases}$. So without loss of generality by possibly applying this linear coordinate change we may assume $sc'>0$. We analyze 3 sub-cases: 

\

\textbf{Case (B).2.1.1.} Assume the following holds: $$(\star)\text{ }s=-1\text{ and }b=c.$$ Explicitly we now have, $I=\langle x^2-y^2+bz(x+y),xy+z(x+c'y)\rangle_2$. Note that $c'\neq1$, as $sc'>0$. So any
point
in $S^2$ that satisfies $(x=-y\text{ } \&\text{ } z=\frac{1}{1-c'}x)$
lies in $Allow(I)$,
which is a contradiction to our assumption that $\dim\text{span}_{\mathbb{R}}Allow(I)=1$. Case (B).2.1.1. is complete. 

\

\textbf{Case (B).2.1.2.} Assume the following holds: $$(\star\star)\text{ }s=+1\text{
and }c'>0\text{ and }bc'=c.$$ Explicitly we now have, $I=\langle x^2+y^2+bz(x+c'y),xy+z(x+c'y)\rangle_2$, with $b\neq0$ and $c'>0$. We can also write $I=\langle x^2+y^2+2xy+(b+2)z(x+c'y),xy+z(x+c'y)\rangle_2$ or alternatively $I=\langle(x+y)^2+(b+2)z(x+c'y),xy+z(x+c'y)\rangle_2$. 

\

\textbf{Case (B).2.1.2.1.} Assume $b=-2$. We then have $I=\langle(x+y)^2,xy+z(x+c'y)\rangle_2$ with $c'>0$. If $c'=1$ then in the coordinate system $(X,y,Z)=(x+y,y,y+z)$ we have $I=\langle X^2,y(X-y)+zX\rangle_2=\langle X^2,y^{2}-X(y+z)\rangle_2=\langle X^2,y^{2}-XZ\rangle_2$, which
is a contradiction to $I$ being closed by Example \ref{amazing_exaple_for_negligible_part2_SECOND_PAPER}. We conclude that $0<c'\neq 1$. In this case note that in the coordinate system $(X,Y,z)=(x+y,x+c'y,z)$ we have $xy=(1-c')^{-2}(X-Y)(Y-c'X)$ so $I=\langle X^2,\gamma(X-Y)(Y-c'X)+zY\rangle_2$ with $0<c'\neq1$ and $\gamma:=(1-c')^{-2}>0$. We can also write $I=\langle X^2,\gamma XY-\gamma Y^2-c'\gamma X^2+c'\gamma XY+zY\rangle_2=\langle X^2,Y^2-(1+c')XY-\gamma^{-1}zY\rangle_2=\langle X^2,Y\big(Y-(1+c')X-\gamma^{-1}z\big)\rangle_2$, which is impossible by the observation above. Case (B).2.1.2.1. is complete. 

\

\textbf{Case (B).2.1.2.2.} Assume $b\neq-2$. If $c'=1$ then $(x+y)\big(x+y+(b+2)z\big)\in I$, which is impossible by the observation above. Therefore, $c'\neq1$. By scaling $z$ we get $I=\langle(x+y)^2+z(x+c'y),xy+bz(x+c'y)\rangle_2$ with some $0<c'\neq1$ and some new $b\neq0$.  Again, in the coordinate
system $(X,Y,z)=(x+y,x+c'y,z)$ we have $xy=(1-c')^{-2}(X-Y)(Y-c'X)$ and so \begin{gather}\label{label_for_very_long_classification_red_star_of_david}I=\langle
X^2+zY,\gamma(X-Y)(Y-c'X)+bzY\rangle_2\text{ with }c'\notin\{0,1\}\text{ and }\gamma:=(1-c')^{-2}>0.\end{gather} In this case we have
(by assumption) $Allow(I)\subset\{(0,0,\pm1)\}$ and so (c.f., Example \ref{amazing_exaple_for_negligible_part1_SECOND_PAPER})
the function $\tilde F=-\frac{X^3}{z}$ in negligible for $Allow(I)$. So by
Definition \ref{implied_polynomial_definition_SECOND_PAPER}, setting $L=1$, $S_1=\frac{X}{z}$
and $Q_1=X^2+zY$ we have that $XY=S_1\cdot Q_1+\tilde F=(\frac{X}{z})\cdot(X^2+zY)-\frac{X^3}{z}$
is implied by $I$, so $XY\in I$. So we have $X^2+zY,\gamma(X-Y)(Y-c'X)+bzY,XY\in I$, with $\gamma\neq0$, which is a contradiction to $\dim I=2$ (observing the coefficient of $Y^2$ in the second jet we see that these 3 jets are linearly independent). We conclude that no ideal of the form (\ref{label_for_very_long_classification_red_star_of_david}) is both closed and satisfies $\dim\text{span}_{\mathbb{R}}Allow(I)=1$ and $\dim I=2$. Case (B).2.1.2.2. is complete and so also Case (B).2.1.2. is complete.

\

\textbf{Case (B).2.1.3.} Assume that neither $(\star)$ nor $(\star\star)$ holds. Recall we have $$I=\langle x^2+sy^2+z(bx+cy),xy+z(x+c'y)\rangle_2,\text{ where }c\neq0,c'\neq0\text{
and }s\in\{\pm1\},$$and furthermore $sc'>0$. Define a polynomial $$p_{s,b,c,c'}(\alpha):=\alpha(b+c\alpha)-(1+s\alpha^2)(1+c'\alpha).$$The leading coefficient of $p_{s,b,c,c'}$ is negative, it is a degree 3 polynomial, and moreover $p_{s,b,c,c'}(0)=-1$. Thus, there exists a negative root to $p_{s,b,c,c'}$. We first establish that the following holds:$$(\star\star\star)\text{ There exists }\alpha_0\neq-\frac{b}{c}\text{ such that }p_{s,b,c,c'}(\alpha_0)=0.$$Indeed, note that $p_{s,b,c,c'}(-\frac{b}{c})=-(1+s(\frac{-b}{c})^2)(1+c'(\frac{-b}{c}))$. 

If $(1+s(\frac{-b}{c})^2)=0$ then $s=-1$ and $\abs{b}=\abs{c}$. So either $b=c$ or $b=-c$. If $b=-c$ then $\frac{-b}{c}\geq0$, and so $(\star\star\star)$ holds as we saw that there exists a negative root to
$p_{s,b,c,c'}$. If $b=c$ then $(\star)$ holds, which is a contradiction.

If $(1+c'(\frac{-b}{c}))=0$ then $c=c'b$. If moreover $bc\leq0$ then $(\star\star\star)$
holds as we saw that there exists a negative root to
$p_{s,b,c,c'}$. Otherwise $b$ and $c$ have the same sign (and recall both $c\neq0$ and $c'\neq0$), we get that $c'>0$. Since $sc'>0$ we get $s=+1$, and so $(\star\star)$ holds, which is a contradiction.

If both $(1+s(\frac{-b}{c})^2)\neq0$ and $(1+c'(\frac{-b}{c}))\neq0$ then $(\star\star\star)$
clearly holds too. 

We conclude that $(\star\star\star)$ holds and fix $\alpha_0$ as in $(\star\star\star)$. Now, any
point
in $S^2$ that satisfies $(y=\alpha_0x\text{ } \&\text{ } z=-\frac{1+s\alpha_0^2}{b+c\alpha_0}x)$
lies in $Allow(I)$,
which is a contradiction to our assumption that $\dim\text{span}_{\mathbb{R}}Allow(I)=1$. In order to see this we need to show that the 2 generators of $I$ vanish
at such a point. Indeed, $$  x^2+sy^2+z(bx+cy)=x^2+s\alpha_0^2x^2-\frac{1+s\alpha_0^2}{b+c\alpha_0}x(bx+c\alpha_0x)=0;$$
$$xy+z(x+c'y)=\alpha_{0}x^2-\frac{1+s\alpha_0^2}{b+c\alpha_0}x(x+c'\alpha_0x)=\frac{x^2}{b+c\alpha_0}[\alpha_{0}(b+c\alpha_0)-(1+s\alpha_0^2)(1+c'\alpha_0)]=0.$$ Case (B).2.1.3. is complete, and so also Case (B).2.1. is complete.

\

\textbf{Case (B).2.2.} Assume $c=0$. Then, $I=\langle x^2\pm y^2+bxz,xy+xz+c'yz \rangle_2$ with some $c'\neq0$. 

\

\textbf{Case (B$^-$).2.2.} $I=\langle x^2- y^2+bxz,xy+xz+c'yz
\rangle_2$ with some $c'\neq0$. If $b\neq0$ we define the degree 3 polynomial $p_{b,c'}(\alpha):=b\alpha+(\alpha^2-1)+\alpha(\alpha^2-1)c'$. Fixing $\alpha_0$ such that $p_{b,c'}(\alpha_0)=0$ we get that any
point
in $S^2$ that satisfies $(y=\alpha_0x\text{ } \&\text{ } z=\frac{\alpha_0^2-1}{b}x)$
lies in $Allow(I)$,
which is a contradiction to our assumption that $\dim\text{span}_{\mathbb{R}}Allow(I)=1$. In order to see this we need to show that the 2 generators of $I$ vanish
at such a point. Indeed, $$   x^2- y^2+bxz= x^2- \alpha_0^2x^2+b\frac{\alpha_0^2-1}{b}x^{2}=0;$$
$$xy+xz+c'yz
=\alpha_0 x^2+\frac{\alpha_0^2-1}{b}x^2+c'\alpha_0 \frac{\alpha_0^2-1}{b}x^{2}
=\frac{x^2}{b}[b\alpha_0+(\alpha_0^2-1)+\alpha_0(\alpha_0^2-1)c']=0.$$ We conclude that $b=0$ and so $I=\langle x^2- y^2,xy+xz+c'yz
\rangle_2$ with some $c'\neq0$.  If $c'\neq-1$ then any
point
in $S^2$ that satisfies $(y=x\text{ } \&\text{ } z=-\frac{1}{1+c'}x)$
lies in $Allow(I)$,
which is a contradiction to our assumption that $\dim\text{span}_{\mathbb{R}}Allow(I)=1$. If $c'=-1$ then any
point
in $S^2$ that satisfies $(y=-x\text{ } \&\text{ } z=\frac{1}{2}x)$
lies in $Allow(I)$,
which is a contradiction to our assumption that $\dim\text{span}_{\mathbb{R}}Allow(I)=1$. Case (B$^-$).2.2 is complete.

\

\textbf{Case (B$^+$).2.2.} $I=\langle x^2+ y^2+bxz,xy+z(x+c'y)
\rangle_2$ with some $c'\neq0$. If $b\neq0$ then we define the degree 3 polynomial
$p_{b,c'}(\alpha):=b\alpha-(1+\alpha^2)(1+c'\alpha)$. Fixing $\alpha_0$
such that $p_{b,c'}(\alpha_0)=0$ we get that any
point
in $S^2$ that satisfies $(y=\alpha_0x\text{ } \&\text{ } z=-\frac{\alpha_0^2+1}{b}x)$
lies in $Allow(I)$,
which is a contradiction to our assumption that $\dim\text{span}_{\mathbb{R}}Allow(I)=1$. In order to see this we need to show that the 2 generators of $I$ vanish
at such a point. Indeed, $$    x^2+ y^2+bxz= x^2+ \alpha_0^2x^2-b\frac{\alpha_0^2+1}{b}x^{2}=0;$$
$$xy+z(x+c'y)=\alpha_0x^2-\frac{\alpha_0^2+1}{b}x(x+\alpha_0c'x)=\frac{x^2}{b}[b\alpha_0-(1+\alpha_0^2)(1+c'\alpha_0)]=0.$$
We conclude that $b=0$ and so $I=\langle x^2+ y^2,xy+z(x+c'y)
\rangle_2$ with some $c'\neq0$.
If $c'=1$ then $(x+y)(x+y+2z)\in I$ which is impossible by the observation above. We conclude that $c'\notin\{0,1\}$. Now, in the coordinate system $(X,Y,z)=(x+y,x+c'y,z)$ we get that $I=\langle
X^2+2zY,\gamma(X-Y)(Y-c'X)+zY\rangle_2$ with $c'\notin\{0,1\}$ and $\gamma:=(1-c')^{-2}>0$. Scaling $z$ we get $I=\langle
X^2+zY,\gamma(X-Y)(Y-c'X)+\frac{1}{2}zY\rangle_2$ with $c'\notin\{0,1\}$ and $\gamma:=(1-c')^{-2}>0$. This ideal is of the form (\ref{label_for_very_long_classification_red_star_of_david}), which we already showed  is never both closed and satisfies $\dim\text{span}_{\mathbb{R}}Allow(I)=1$ and $\dim I=2$.

Case (B$^+$).2.2 is complete, and so Case (B).2.2., Case (B).2. and Case (B) are complete too.

\

\textbf{Case (A).} Assume $b'\neq0$. In this case we can scale $y$ and get
$$I=\langle x(x+ay+bz)+cyz,y(y+a'x+b'z)+c'xz\rangle_2$$for some new $a,b,c,a',b',c'\in\mathbb{R}$.

\

\textbf{Case (A).1.} Assume $c=0$. By the observation $b=0$ so $I=\langle
x(x+ay),y(y+a'x+b'z)+c'xz\rangle_2$.

\

\textbf{Case (A).1.1.} Assume $c'=0$. By the observation $b'=0$ so $I=\langle
x(x+ay),y(y+a'x)\rangle_2$.

\

\textbf{Case (A).1.1.1.} Assume $a\neq0$. We scale $y$ to get  $I=\langle
x(x+y),y(y+Cx)\rangle_2$ with some $C\in\mathbb{R}$, which was already analyzed
in (\ref{new_small_observation_to_a_less_than_one_etc}).

\

\textbf{Case (A).1.1.2.} Assume $a=0$. Then, we have $I=\langle
x^{2},y(y+a'x)\rangle_2$. If $a'\neq0$ we scale $x$ to get $I=\langle
x^{2},y(y+x)\rangle_2$, and then rename $x\leftrightarrow y$ to get a special
case of case (\ref{new_small_observation_to_a_less_than_one_etc}). If $a'=0$
we are in case (\ref{model_for_2_3_with_1_dim_allow_1}). 

\

Case (A).1.1. is complete.

\

\textbf{Case (A).1.2.} Assume $c'\neq0$. We scale $z$ to get $I=\langle x(x+ay),y(y+a'x+b'z)+xz\rangle_2$,
with some new $b'\in\mathbb{R}$. 

\

\textbf{Case (A).1.2.1.} Assume $a=0$. Then, $I=\langle x^2,y^{2}+xz+y(a'x+b'z)\rangle_2$.
If $b'\neq0$ then any point
in $S^2$ that satisfies $(x=0\text{ } \&\text{ } y=-b'z)$ lies in $Allow(I)$,
which is a contradiction to our assumption that $\dim\text{span}_{\mathbb{R}}Allow(I)=1$.
Otherwise $I=\langle x^2,y^{2}+x(z+a'y)\rangle_2$, and in the coordinate
system $(X,Y,Z)=(x,y,-(z+a'y))$ we get $I=\langle X^2,Y^2-XZ\rangle_2$ which
is a contradiction to $I$ being closed by Example \ref{amazing_exaple_for_negligible_part2_SECOND_PAPER}.

\

\textbf{Case (A).1.2.2.} Assume $a\neq0$. Then, by scaling $y$ we get $I=\langle
x(x+y),y^{2}+txz+a'xy+b'yz\rangle_2$ with some new $t,a',b'\in\mathbb{R}$.
If $b'\neq t$ then the point
in $S^2$ that satisfies $(x=-y\text{ } \&\text{ } z=\frac{-1+a'}{b'-t}y)$
lies in $Allow(I)$,
which is a contradiction to our assumption that $\dim\text{span}_{\mathbb{R}}Allow(I)=1$.
Otherwise, $b'=t$ and $I=\langle x(x+y),y^{2}+b'z(x+y)+a'xy\rangle_2$. If
$b'\neq0$ then any point
in $S^2$ that satisfies $(x=0\text{ } \&\text{ } z=-\frac{1}{b'}y)$
lies in $Allow(I)$,
which is a contradiction to our assumption that $\dim\text{span}_{\mathbb{R}}Allow(I)=1$.
Otherwise, $b'=0$ and $I=\langle x(x+y),y(y+a'x)\rangle_2$, which was already
analyzed in (\ref{new_small_observation_to_a_less_than_one_etc}) above.

\

Case (A).1.2. is complete, and so Case (A).1. is complete as well.

\

\textbf{Case (A).2.} Assume $c\neq0$.

\

\textbf{Case (A).2.1.} Assume $c'=0$. Then rename $x\leftrightarrow y$ to
be in Case (A).1.2. that was analyzed.

\

\textbf{Case (A).2.2.} Assume $c'\neq0$. Scaling $z$ we get $$I=\langle x^{2}+yz+x(ay+bz),y^2+c'xz+y(a'x+b'z)\rangle_2,$$
with some new $b,b',c'\in\mathbb{R}$ and still $c'\neq0$. Further scaling
$x$ we get $$I=\langle k x^{2}+yz+x(ay+bz),y^2+xz+y(a'x+b'z)\rangle_2,$$
with some new $a,b,a'\in\mathbb{R}$ and some $k>0$. If $a'\neq0$ then in
the coordinate system $(X,Y,Z)=(y+a'x,y,z)$ the second generator becomes
$YX+Z(pX+qY)$ for some $p,q\in\mathbb{R}$ and so since $\dim\text{span}_{\mathbb{R}}Allow(I)=1$ we are in case (B) that
was already analyzed. Similarly if $a\neq0$ then in
the coordinate system $(X,Y,Z)=(x,kx+ay,z)$ the first generator becomes
$YX+Z(pX+qY)$ for some $p,q\in\mathbb{R}$ and so we are again, since $\dim\text{span}_{\mathbb{R}}Allow(I)=1$,
in case (B) that was already analyzed. We are left to check the case in which
$a=a'=0$: $$I=\langle k x^{2}+yz+bxz,y^2+xz+b'yz\rangle_2.$$ 

\

\textbf{Case (A).2.2.1.} Assume $bb'=1$. Then, $I=\langle
k x^{2}+z(y+bx),y^2+\frac{z}{b}(y+bx)\rangle_2$. In the coordinate system
$(x,Y,z)=(x,y+bx,z)$ we get $$I=\langle
k x^{2}+zY,(Y-bx)^{2}+\frac{z}{b}Y\rangle_2=\langle
k x^{2}+zY,Y^2+b^2x^2-2bxY+\frac{z}{b}Y\rangle_2$$ $$=\langle
k x^{2}+zY,Y\big(Y+(\frac{1}{b}-\frac{b^2}{k})z-2bx\big) \rangle_2.$$If $\frac{1}{b}-\frac{b^2}{k}\neq0$
then we get a contradiction by the observation above. Otherwise we have that
$I=\langle kx^2+zY,Y^2-2bxY\rangle_2$ with $b\neq0$, so any point
in $S^2$ that satisfies $(Y=2bx\text{ } \&\text{ } z=-\frac{k}{2b}x)$
lies in $Allow(I)$,
which is a contradiction to our assumption that $\dim\text{span}_{\mathbb{R}}Allow(I)=1$.

\

Case (A).2.2.1 is complete. 

\

\textbf{Case (A).2.2.2.} Assume $bb'\neq1$. Then we have $I=\langle
k x^{2}+yz+bxz,y^2+xz+b'yz\rangle_2$ with $k>0$. Define the polynomial $p_{k,b,b'}(\alpha):=(\alpha+b)\frac{\alpha^2}{k}-(b'\alpha+1)$.
This is a degree 3 polynomial so it has a real root. Moreover, $p_{k,b,b'}(-b)=bb'-1\neq0$,
so we can fix $\alpha_0\neq-b$ such that $p_{k,b,b'}(\alpha_0)=0$. Now, any
point
in $S^2$ that satisfies $(y=\alpha_0x\text{ } \&\text{ } z=-\frac{k}{\alpha_0+b}x)$
lies in $Allow(I)$,
which is a contradiction to our assumption that $\dim\text{span}_{\mathbb{R}}Allow(I)=1$.
In order to see this we need to show that the 2 generators of $I$ vanish
at such a point. Indeed, $$k x^{2}+yz+bxz=kx^2-\alpha_0\frac{k}{\alpha_0+b}x^2-b\frac{k}{\alpha_0+b}x^2=0;$$
$$y^2+xz+b'yz=\alpha_0^2x^2-\frac{k}{\alpha_0+b}x^2-b'\frac{k\alpha_0}{\alpha_0+b}x^2=\frac{kx^2}{\alpha_0+b}[(\alpha_{0}+b)\frac{\alpha_{0}^2}{k}-(b'\alpha_{0}+1)]=0.$$
 
\

Case (A).2.2.2. is complete, and so also Case (A).2.2., and so (A).2., and
(A) is complete too.

\

\textbf{\uline{Assume $\dim I=3$.}} Recall that $I\subset \text{span}_{\mathbb{R}}\{x^2,y^{2},xy,yz,xz\}$,
$\dim\text{span}_{\mathbb{R}}Allow(I)=1$ and $Allow(I)=\{(0,0,\pm1)\}$. We then
must have that $I$ contains a jet with a non-zero $x^2$ term (otherwise $(1,0,0)\in
Allow(I)$), and a jet with a non-zero $y^2$ term (otherwise $(0,1,0)\in
Allow(I)$). We start with the following observation: $I$ does not contain any jet of the form $z(\alpha x+\beta y)$ with $\alpha^2+\beta^2\neq0$. Indeed, assume $z(\alpha x+\beta y)\in I$ with $\alpha^2+\beta^2\neq0$. By possibly
renaming $x\leftrightarrow y$ we may assume $z(x+\beta y)\in I$. As $Allow(I)=\{(0,0,\pm1)\}$
we have that $S_1(x,y,z):=\frac{x}{z}$
satisfies the bounds in (\ref{second_cond_of_implied_poly_in_direction_SECOND_PAPER})
in a conic neighborhood of $Allow(I)$ and so $x(x+\beta y)= \frac{x}{z}\cdot
(x+\beta y)z\in
I$ by Corollary \ref{cor_strong_directional_implication_imply_implication_SECOND_PAPER} and similarly $y(x+\beta y)= \frac{y}{z}\cdot
(x+\beta y)z\in
I$. Altogether $x(x+\beta y),y(x+\beta y),z(x+\beta y)\in I$ and as $\dim
I=3$ we get $I=\langle x(x+\beta y),y(x+\beta y),z(x+\beta y) \rangle_2$.
But now we get that any
point
in $S^2$ that satisfies $(x=-\beta y\text{ } \&\text{ } z=0)$
lies in $Allow(I)$,
which is a contradiction to our assumption that $\dim\text{span}_{\mathbb{R}}Allow(I)=1$.

\

\textbf{Case (A). }Assume there exists a set of generators such that $I=\langle
ax^2+by^2+q_0,q_1,q_2\rangle_2$ with $q_0,q_1,q_2\in\text{span}\{xy,xz,yz\}$.
In that case $ab\neq0$. For $0\leq i\leq 2$ we can write $q_i=z(\alpha_i x+\beta_iy)+\gamma_ixy$
for some $\alpha_i,\beta_i,\gamma_i\in\mathbb{R}$. Since $\dim I=3$, by possibly
replacing $q_1,q_2$ by other two jets we may assume $\gamma_2=0$. So we have
$z(\alpha_2x+\beta_2y)\in I$ with $\alpha_2^2+\beta_2^2\neq0$. This is a contradiction by the observation.

\

\textbf{Case (B).} Assume there exists a set of generators such that $I=\langle
x^2+q_0,y^2+q_1,q_2\rangle_2$ with $q_0,q_1,q_2\in\text{span}\{xy,xz,yz\}$.
For $0\leq i\leq 2$ we can write $q_i=z(\alpha_i x+\beta_iy)+\gamma_ixy$
for some $\alpha_i,\beta_i,\gamma_i\in\mathbb{R}$. By the observation we must have $\gamma_2\neq0$ and so we may write $I=\langle
x^2+A_{1}xz+B_{1}yz,y^2+A_{2}xz+B_{2}yz,xy+A_{3}xz+B_{3}yz\rangle_2$ for some $A_1,A_2,A_3,B_1,B_2,B_3\in\mathbb{R}$. For $0\leq i\leq 2$ we define linear forms $l_i(x,y):=A_i x+B_i y$ and we have $I=\langle
x^2+zl_1,y^2+zl_2,xy+zl_3\rangle_2$. If $B_3=0$ but $A_3\neq0$ then $xy+A_3xz=x(y+A_3z)\in I$, and this implies (by the usual argument) that $x^2,xy\in I$, and so also $xz\in I$ and moreover $y^2+zl_2\in I$. This is a contradiction to $\dim I=3$. We conclude that if $B_3=0$
then $A_3=0$. Similarly, if $A_3=0$ then $B_3=0$ as well.

\

\textbf{Case (B).1.} Assume $l_3=0$. Then, $I=\langle
x^2+zl_1,y^2+zl_2,xy\rangle_2$. As $Allow(I)=\{(0,0,\pm1)\}$, by Example \ref{amazing_exaple_for_negligible_part1_SECOND_PAPER}
the function $\tilde F=-\frac{y^3}{z}$ in negligible for $Allow(I)$ in any allowed direction. So by
Definition \ref{implied_polynomial_definition_SECOND_PAPER}, setting $L=1$, $S_1=\frac{y}{z}$
and $Q_1=y^2+zl_2$ we have that $yl_2=S_1\cdot Q_1+\tilde F=(\frac{y}{z})\cdot(y^2+zl_2)-\frac{y^3}{z}$
is implied by $I$, and since $I$ is closed, $yl_2\in I$. So we have $yl_2,xy,x^2+zl_1,y^2+zl_2\in I$. As $\dim I=3$ these 4 jest
are linearly dependent, so there exist $k_1,k_2,k_3,k_4\in\mathbb{R}$ not
all zero such that $$k_1yl_2+k_2xy+k_3(x^2+zl_1)+k_4(y^2+zl_2)=0.$$Looking
at the coefficients of $x^2$ we get $k_3=0$ and so $$k_1yl_2+k_2xy+k_4(y^2+zl_2)=0.$$ Now, looking at the coefficients of $xy$ we must have $k_1A_2=-k_2$. Moreover we must have $k_{4}l_2=0$ by looking at the coefficient of $z$. If $l_2\neq0$ we have $k_4=0$ and so $B_2=0$ and $A_2\neq0$, by looking at the coefficients of $y^2$. So we proved that in any case $l_2=A_2x$ for some $A_2\in\mathbb{R}$. By repeating the argument after replacing $x\leftrightarrow y$ we also have $l_1=B_1y$ for some $B_1\in\mathbb{R}$. Altogether we have $I=\langle
x^2+B_{1}zy,y^2+A_{2}zx,xy\rangle_2$. Recall that $Allow(I)=\{(0,0,\pm1)\}$ and repeat the argument above while now noticing that we have that $-\frac{xy^2}{z}$ in negligible for $Allow(I)$ we get that $A_2x^2\in I$. If $A_2\neq0$ we will have $x^2,xy,y^2+A_2zx,B_1zy\in I$. As $\dim I=3$ we must then have $B_1=0$. Similarly, if $B_1\neq0$ we must have $A_2=0$. So without loss of generality (by possibly replacing $x\leftrightarrow y$) we may assume $B_1=0$ and so $I=\langle
x^2,y^2+A_{2}zx,xy\rangle_2$. If $A_2=0$ then $I=\langle x^2,y^2,xy\rangle_2$, i.e., we are
in case
(\ref{model_for_2_3_with_1_dim_allow_3}).  Otherwise by scaling $z$ we get $I=\langle
x^2,y^2+xz,xy\rangle_2$, i.e., we are in case (\ref{model_for_2_3_with_1_dim_allow_3'_the_one_we_missed}).

\

\textbf{Case (B).2.} Assume $l_3\neq0$. Then, $A_3B_3\neq0$ and by scaling
$z$ we can assume
that $l_3=x+B_3y$ for some $B_{3}\neq0$, i.e., $A_3=1$. Since $xy+zl_3 \in
I$ and $Allow(I)=\{(0,0,\pm1)\}$ we have (by a straightforward
adjustment of Example \ref{amazing_exaple_for_negligible_part2_SECOND_PAPER},
and now use the fact that both $\frac{x^2y}{z}$ and $\frac{xy^2}{z}$ are
negligible for $Allow(I)$) that also $xl_3,yl_3\in I$. So in particular,
as $\dim I=3$ the four jets $xl_3,yl_3,xy+zl_3,x^2+zl_1\in I$ are linearly
dependent. One easily sees that this is only possible if $l_1=\alpha l_3$
for some non-zero $\alpha\in\mathbb{R}$. Similarly we get that $l_2=\beta
l_3$ for some non-zero $\beta\in\mathbb{R}$. Recall
that $I=\langle
x^2+zl_1,y^2+zl_2,xy+zl_3\rangle_2$ we can now write $$I=\langle 
x^2+\alpha z(x+B_{3}y),y^2+\beta z(x+B_{3}y),xy+z(x+B_{3}y) \rangle_2\text{
for some non-zero }B_3,\alpha,\beta\in\mathbb{R}.$$ In particular $ 
x^2+\alpha z(x+B_{3}y)-\alpha [xy+z(x+B_{3}y)]=x^{2}-\alpha xy\in I$. We
already showed $xl_3=x(x+B_{3}y)=x^2+B_{3}xy\in I$ so also $(B_{3}+\alpha)xy\in
I$. If $B_{3}\neq-\alpha$ we will have immediately $xy,x^2,y^2,z(x+B_{3}y)\in
I$, which is impossible as $\dim I=3$. We conclude that $\alpha=-B_3$. Similarly,
we have $ 
y^2+\beta z(x+B_{3}y)-\beta [xy+z(x+B_{3}y)]=y^{2}-\beta xy\in I$. We
already showed $yl_3=y(x+B_{3}y)=B_{3}y^2+xy\in I$ so also $(1+\beta B_{3})xy\in
I$. If $B_{3}^{-1}\neq-\beta$ we will have immediately $xy,x^2,y^2,z(x+B_{3}y)\in
I$, which is impossible as $\dim I=3$. We conclude that $\beta=-B_3^{-1}$.
So we showed $$I=\langle 
x^2-B_3 z(x+B_{3}y),y^2-B_3^{-1} z(x+B_{3}y),xy+z(x+B_{3}y) \rangle_2\text{
for some non-zero }B_3\in\mathbb{R}.$$ We
already showed $xl_3=x(x+B_{3}y)\in I$ and $yl_3=y(x+B_{3}y)\in I$, and also
$xy+z(x+B_{3}y)\in I$. These 3 jets are linearly independent, and $\dim I=3$,
so $I=\langle x(x+B_{3}y),y(x+B_{3}y),xy+z(x+B_{3}y) \rangle_2$, with some
$B_3\neq0$. In the coordinate system $(X,Y,Z)=(x+B_{3}y,y,z)$ we have $I=\langle
(X-B_{3}Y)X,XY,(X-B_{3}Y)Y+XZ \rangle_2=\langle  X^{2},XY,Y^2-B_3^{-1}XZ\rangle_2$,
which we already  analyzed in Case (B).1. above.

\

\textbf{\uline{Assume $\dim I=4$.}} Recall that $I\subset \text{span}_{\mathbb{R}}\{x^2,y^{2},xy,yz,xz\}$, $\dim\text{span}_{\mathbb{R}}Allow(I)=1$ and $Allow(I)=\{(0,0,\pm1)\}$. We then must have that $I$ contains a jet with a non-zero $x^2$ term (otherwise $(1,0,0)\in Allow(I)$), and a jet with a non-zero $y^2$ term (otherwise $(0,1,0)\in
Allow(I)$). 

\

\textbf{Case (A). }Assume there exists a set of generators such that $I=\langle ax^2+by^2+q_0,q_1,q_2,q_3\rangle_2$ with $q_0,q_1,q_2,q_3\in\text{span}\{xy,xz,yz\}$. In that case $ab\neq0$. Moreover, as $\dim I=4$ we have $I=\langle
ax^2+by^2,xy,xz,yz\rangle_2$. As $Allow(I)=\{(0,0,\pm1)\}$ we have that $S_1(x,y,z):=\frac{x}{z}$
satisfies the bounds in (\ref{second_cond_of_implied_poly_in_direction_SECOND_PAPER}) in a conic neighborhood of $Allow(I)$ and so $x^2= \frac{x}{z}\cdot xz\in I$ by Corollary \ref{cor_strong_directional_implication_imply_implication_SECOND_PAPER}. This implies that $x^2,y^2,xy,xz,yz\in I$, which is a contradiction to $\dim I=4$.

\

\textbf{Case (B).} Assume there exists a set of generators such that $I=\langle
x^2+q_0,y^2+q_1,q_2,q_3\rangle_2$ with $q_0,q_1,q_2,q_3\in\text{span}\{xy,xz,yz\}$. For $0\leq i\leq 3$ we can write $q_i=z(\alpha_i x+\beta_iy)+\gamma_ixy$ for some $\alpha_i,\beta_i,\gamma_i\in\mathbb{R}$. Since $\dim I=4$, by possibly replacing $q_2,q_3$ by other two jets we may assume $\gamma_3=0$. So we have $z(\alpha_3x+\beta_3y)\in I$ with $\alpha_3^2+\beta_3^2\neq0$. By possibly renaming $x\leftrightarrow y$ and replacing $q_3$ by a scalar multiplication of it, we may assume $\alpha_3=1$ and so $z(x+\beta_3y)\in I$. As $I$ is closed and $Allow(I)=\{(0,0,\pm1)\}$, we have that also $x(x+\beta_3y),y(x+\beta_3y)\in I$. Altogether we have $x^2+q_{0},y^2+q_1,x(x+\beta_3y),y(x+\beta_3y),z(x+\beta_3y)\in I$ with $q_0,q_1\in\text{span}\{xy,xz,yz\}$. In the coordinate system $(X,y,z)=(x+\beta_3y,y,z)$ we get $X^2,Xy,Xz\in I$. As also in this new coordinate system $Allow(I)=\{(X,y,z)=(0,0,\pm1)\}$, we must also have a jet with a non-zero $y^2$ term in this new system, so $y^2+kyz\in I$ for some $k\in\mathbb{R}$. If $k\neq0$ then as $I$ is closed and $Allow(I)=\{(0,0,\pm1)\}$ we get $y^2=\frac{y}{y+kz}y(y+kz)\in I$ and also $yz\in I$. Altogether $X^2,Xy,Xz,y^2,yz\in I$ which is a contradiction to $\dim I=4$. Thus we must have $k=0$, so $y^2,X^2,Xy,Xz\in I$ and as $\dim I=4$ we get $I=\langle X^2,y^2,Xy,Xz\rangle_2$, i.e., we are in case (\ref{model_for_2_3_with_1_dim_allow_4}).\end{proof}

\begin{remark}\label{remark_on_distinctness_for_closed_ideals_in_C2R3_1_dim_allow}The ideals (\ref{model_for_2_3_with_1_dim_allow_1})-(\ref{model_for_2_3_with_1_dim_allow_5}) are pairwise distinct, i.e.,
there does not exist a $C^2$ coordinate change around the origin such that
a given ideal can take two different forms from this list. (\ref{model_for_2_3_with_1_dim_allow_4}) and (\ref{model_for_2_3_with_1_dim_allow_5}) are obviously distinct as they are the only ideals in this list with dimensions 4 and 5 respectively. (\ref{model_for_2_3_with_1_dim_allow_1}) and (\ref{model_for_2_3_with_1_dim_allow_2}) are the only ideals in this list of dimension 2, however it is easy to see that (\ref{model_for_2_3_with_1_dim_allow_2}) does not contain two linearly independent jets that are squares of linear forms, while (\ref{model_for_2_3_with_1_dim_allow_1}) clearly does. Thus, (\ref{model_for_2_3_with_1_dim_allow_1})
and (\ref{model_for_2_3_with_1_dim_allow_2}) are distinct. Finally, (\ref{model_for_2_3_with_1_dim_allow_3})
and (\ref{model_for_2_3_with_1_dim_allow_3'_the_one_we_missed}) are the only ideals in this
list of dimension 3, however the first contains two linearly independent jets that are squares of linear forms, while it is easy to see that the latter does not.\end{remark}

\begin{lemma}\label{lemma_classify_closed_ideals_in_C2R3_2_dim_allow}Let
$I\lhd\mathcal{P}_0^2(\mathbb{R}^{3})$
be a closed non-principal ideal. Assume $I$ does not contain any jet of
order of vanishing 1 and that $\dim\text{span}_{\mathbb{R}}Allow(I)=2$. Then,
$I$ is, up to a
linear
coordinate change, one of the following:

\begin{gather}\label{model_for_2_3_with_2_dim_allow_1}\langle
x^2,xy,yz,xz\rangle_2\end{gather}
\begin{gather}\label{model_for_2_3_with_2_dim_allow_2}\langle x^2, xy,yz\rangle_2\end{gather}
\begin{gather}\label{model_for_2_3_with_2_dim_allow_3}\langle x^2, xy,xz\rangle_2\end{gather}
\begin{gather}\label{model_for_2_3_with_2_dim_allow_4}\langle x^2, xy\rangle_2\end{gather}
\begin{gather}\label{model_for_2_3_with_2_dim_allow_5}\langle x^2, yz\rangle_2\end{gather}
\begin{gather}\label{model_for_2_3_with_2_dim_allow_6}\langle xz,yz-x^2\rangle_2\end{gather}
\begin{gather}\label{model_for_2_3_with_2_dim_allow_7}\langle x^2, xy+xz+yz\rangle_2\end{gather}

\end{lemma}

\begin{proof}As $\dim\text{span}_{\mathbb{R}}Allow(I)=2$, by applying a linear
coordinate change we may assume $(0,1,0),(0,0,1)\in Allow(I)$ and
so $I$ is a linear subspace of $\text{span}_{\mathbb{R}}\{x^2,xy,yz,xz\}$. Moreover, since $(1,0,0)\in Forb(I)$ there exists a jet in $I$ of the form $x^2+q(x,y,z)$ for some $q(x,y,z)\in\text{span}_{\mathbb{R}}\{xy,yz,xz\}$, and so there exists a generating set of the form $$I=\langle x^2+q_{0}(x,y,z),q_{1}(x,y,z),\dots,q_{\dim(I)-1}(x,y,z)\rangle_2,$$ for some $q_{0}(x,y,z),q_{1}(x,y,z),\dots,q_{\dim(I)-1}(x,y,z)\in\text{span}_{\mathbb{R}}\{xy,yz,xz\}$. As
$I$ is not principal we have $\dim I\in\{2,3,4\}$. If $\dim I=4$ then $I=\langle x^2,xy,yz,xz\rangle_2$ and we are in case (\ref{model_for_2_3_with_2_dim_allow_1}).

\

\textbf{Assume $\dim I=3$.} In this case we have $I=\langle x^2+q_{0}(x,y,z),q_{1}(x,y,z),q_{2}(x,y,z)\rangle_2$. 

If $q_1,q_2\in\text{span}_{\mathbb{R}}\{yz,xz\}$, then $I=\langle x^2+\alpha xy, xz,yz\rangle_2$ for some $\alpha\in\mathbb{R}$. If $\alpha=0$ we get that $I=\langle x^2, xz,yz\rangle_2$ and we are in case (\ref{model_for_2_3_with_2_dim_allow_2}) by switching $y\leftrightarrow z$. If $\alpha\neq0$ we get that $I=\langle x(x+\alpha y),xz,z(x+\alpha y)\rangle_2$, so in the coordinate system $(X,Y,Z)=(x,x+\alpha y,z)$ we get $I=\langle XY,XZ,YZ \rangle_2$. Since any linear coordinate changes preserve $\dim\text{span}_\mathbb{R} Allow (I)$ this is a contradiction. 

If $q_1,q_2\in\text{span}_{\mathbb{R}}\{xy,yz\}$, then by renaming $y\leftrightarrow z$ we are back in the previous case.

If $q_1,q_2\in\text{span}_{\mathbb{R}}\{xy,xz\}$, then $I=\langle x^2+\alpha yz, xy,xz\rangle_2$ for some $\alpha\in\mathbb{R}$. If $\alpha=0$ we get
that $I=\langle x^2, xy,xz\rangle_2$ and we are in case (\ref{model_for_2_3_with_2_dim_allow_3}).
If $\alpha\neq0$ by scaling $y$ we get $I=\langle x^2+yz,xy,xz\rangle_2$. Let us show that this ideal is not closed: indeed, $Allow(I)=\{(0,\pm1,0),(0,0,\pm1)\}$. In cones in the directions $(0,\pm1,0)$ we have that $S_1(x,y,z):=\frac{x}{y}$ satisfies the bounds in (\ref{second_cond_of_implied_poly_in_direction_SECOND_PAPER}) and so $x^2=(\frac{x}{y})\cdot(xy)$ is strongly implied in these directions. Similarly, in cones in the directions $(0,0,\pm1)$ we have that $S_1(x,y,z):=\frac{x}{z}$
satisfies the bounds in (\ref{second_cond_of_implied_poly_in_direction_SECOND_PAPER})
and so $x^2=(\frac{x}{z})\cdot(xz)$ is strongly implied in these directions as well. We conclude that $x^{2}$ is implied by $I$ by Corollary \ref{cor_strong_directional_implication_imply_implication_SECOND_PAPER}, and so $x^{2}\in cl(I)\setminus I$ and $I$ is not closed.

We are left to check the case $I=\langle x^2+q_{0}(x,y,z),axy+byz+cxz,a'xy+b'yz+c'xz\rangle_2$ for some $a,a',b,b',c,c'\in\mathbb{R}$ such that $a^2+a'^2\neq0$, $b^2+b'^2\neq0$ and $c^2+c'^2\neq0$. Without loss of generality $b=1$, and then $I=\langle x^2+axz+a'xy,yz+bxz+b' xy,cxz+c'xy\rangle_2$
for some new $a,a',b,b',c,c'\in\mathbb{R}$ and we may assume $b^2+c^2\neq0$ (otherwise we are in the previous case), $b'^2+c'^2\neq0$ (otherwise
we are in the previous case), and $c^2+c'^2\neq0$ (otherwise $\dim I<3$). Now without loss of generality $c\neq0$ (otherwise rename $y\leftrightarrow z$) and $I=\langle
x^2+axy,yz+b xy,xz+cxy\rangle_2$
for some new $a,b,c\in\mathbb{R}$, and we may assume $b^2+c^2\neq$ (otherwise
we are in the previous case) or $I=\langle
x(x+ay),y(z+bx),x(z+cy)\rangle_2$. 

If $b=0$ then $I=\langle
x(x+ay),yz,x(z+cy)\rangle_2$. In this case $c\neq0$ so we can scale $y$ to get that $I=\langle
x(x+ay),yz,x(z+y)\rangle_2$ with a new $a\in\mathbb{R}$. In that case $Allow(I)=\{(0,0,\pm1),(0,\pm1,0)\}$. We then have that  $S_1(x,y,z):=\frac{x}{y+z}$
satisfies the bounds in (\ref{second_cond_of_implied_poly_in_direction_SECOND_PAPER})
in any allowed direction, and so $x^2=(\frac{x}{y+z})\cdot x(y+z)$ is implied by $I$. We conclude that $a=0$, otherwise $x^2,xy,yz,xz\in I$, which is a contradiction to $\dim I=3$. We thus have $I=\langle
x^{2},yz,x(z+y)\rangle_2$ and we still have $Allow(I)=\{(0,0,\pm1),(0,\pm1,0)\}$. Now we note that around the allowed directions with $y=\pm1$ we have that  $S_1(x,y,z):=\frac{y}{y+z}$
satisfies the bounds in (\ref{second_cond_of_implied_poly_in_direction_SECOND_PAPER}), and so $xy=(\frac{y}{y+z})\cdot x(y+z)$ is strongly implied
by $I$ in these directions. Similarly, around the allowed directions with $z=\pm1$ we have that
that  $S_1(x,y,z):=\frac{x}{z}$
satisfies the bounds in (\ref{second_cond_of_implied_poly_in_direction_SECOND_PAPER}),
and so $xy=(\frac{x}{z})\cdot yz$ is strongly implied
by $I$ in these directions. Since by assumption $I$ is closed, we conclude by Corollary \ref{cor_strong_directional_implication_imply_implication_SECOND_PAPER} that $x^2,yz,xy,xz\in I$, which is a
contradiction to $\dim I=3$.  

Now assume $b\neq0$. In that case we have $I=\langle
x(x+ay),y(z+bx),x(z+cy)\rangle_2$ and by scaling $x$ we get $I=\langle
x(x+ay),y(z+x),x(z+cy)\rangle_2$ with some new $a,c\in\mathbb{R}$. We note that $(0,0,1),(0,1,0)\in Allow(I)$, so by our assumption that  $\dim\text{span}_{\mathbb{R}}Allow(I)=2$ we must have that any point $(x_0,y_0,z_0)\in Allow(I)$ satisfies $x_0=0$. Looking at the second generator of $I$ we get that $Allow(I)=\{(0,0,\pm1),(0,\pm1,0)\}$. If $c\neq0$ then $\frac{x}{z+cy},\frac{y}{z+cy},\frac{z}{z+cy}$ all satisfy the bounds in (\ref{second_cond_of_implied_poly_in_direction_SECOND_PAPER}) in any allowed direction, and so $x^2,xy,xz \in I$. But also $yz\in I$, which is a contradiction to $\dim I=3$. We conclude that $c=0$ and $I=\langle
x(x+ay),y(z+x),xz\rangle_2$, and still $Allow(I)=\{(0,0,\pm1),(0,\pm1,0)\}$. If $a\neq0$ then $\frac{x}{z},\frac{y}{z},\frac{z}{z}$ all satisfy
the bounds in (\ref{second_cond_of_implied_poly_in_direction_SECOND_PAPER}) in the allowed directions $\{0,0,\pm1\}$, while $\frac{x}{x+ay},\frac{y}{x+ay},\frac{z}{x+ay}$ all satisfy
the bounds in (\ref{second_cond_of_implied_poly_in_direction_SECOND_PAPER}) in the allowed
directions $\{0,\pm1,0\}$. So $x^2,xy,xz\in I$, but also $yz\in I$, which is a contradiction
to $\dim I=3$. So $a=0$ and $I=\langle x^{2}
,y(z+x),xz\rangle_2$. In the coordinate system $(x,Y,Z)=(x,x+z,y)$ we get $I=\langle x^2,YZ,x(Y-X) \rangle_2=\langle x^2,xY,YZ, \rangle_2$ and we are in case (\ref{model_for_2_3_with_2_dim_allow_2}). 

\

\textbf{Assume $\dim I=2$.} In this case we have $I=\langle x^2+axy+bxz+cyz,qxy+rxz+syz\rangle_2$ with some $a,b,c,q,r,s\in\mathbb{R}$. We analyze separately the case of $s=0$ and the case of $s\neq0$.

\

\textbf{Case A.} Assume $s=0$, and so $I=\langle x^2+axy+bxz+cyz,qxy+rxz\rangle_2$. Note that we always have $(0,1,0),(0,0,1)\in Allow(I)$. Since $\dim I=2$ we have $q^2+r^2\neq0$, so by possibly renaming $y\leftrightarrow z$ we may assume $q\neq0$. So we get $I=\langle x^2+bxz+cyz,xy+rxz\rangle_2$ with some new $b,c,r\in\mathbb{R}$. 

\

\textbf{Case A.1.} If $r\neq0$ then by scaling $z$ we get $I=\langle x^2+bxz+cyz,xy+xz\rangle_2$ with some (new) $b,c\in\mathbb{R}$.
If $b^2+4c>0$ then we define the polynomial
$p_{b,c}(\alpha):=\alpha^{2}-b\alpha-c$ and fix $\alpha_0\neq0$
such that $p_{b,c}(\alpha_0)=0$. Then, any
point
in $S^2$ that satisfies $(x=\alpha_0y\text{ } \&\text{ } z=-y)$
lies in $Allow(I)$,
which is a contradiction to our assumption that $\dim\text{span}_{\mathbb{R}}Allow(I)=2$ (since the $x$ coordinate of such a point in not zero).
In order to see this we need to show that the 2 generators of $I$ vanish
at such a point. Indeed, $$    xy+xz= x(y-y)=0;$$
$$ x^2+bxz+cyz=\alpha_0^{2}y^{2}-b\alpha_0y^{2}-cy^{2}=y^{2}[\alpha_0^{2}-b\alpha_0-c]=0.$$ If $b^2+4c\leq0$ then $c\leq0$. If $c=0$ then $b=0$ and $I=\langle x^2,xy+xz\rangle_2$. So in the coordinate system $(X,Y,Z)=(x,y+z,z)$ we have $I=\langle X^2,XY\rangle_2$ and we are in case (\ref{model_for_2_3_with_2_dim_allow_4}).
Otherwise $c<0$ and still $I=\langle x^2+bxz+cyz,xy+xz\rangle_2$. So scaling both $y$ and $z$ simultaneously by $(-c)^{-1/2}$ we get $I=\langle x^2+bxz-yz,x(y+z)\rangle_2$ for some new $b$. We have $\{(0,0,\pm1),(0,\pm1,0)\}\subset Allow(I)$, and these are the only allowed directions of $I$ with $x$ coordinate zero. So by our assumption that $\dim\text{span}_{\mathbb{R}}Allow(I)=2$ there does not exist a point  $(x_0,y_0,z_0)\in Allow(I)$ such that $x_0\neq0$. We conclude that $Allow(I)=\{(0,0,\pm1),(0,\pm1,0)\}$. So   $S_1(x,y,z):=\frac{x}{y+z}$
satisfies the bounds in (\ref{second_cond_of_implied_poly_in_direction_SECOND_PAPER})
in any allowed direction, and so, by Corollary \ref{cor_strong_directional_implication_imply_implication_SECOND_PAPER}, $x^2=(\frac{x}{y+z})\cdot x(y+z)$
is implied by $I$. So we have $x^{2},x(y+z),yz-bxz\in I$, which is a contradiction to $\dim I=2$. 

\

\textbf{Case A.2.} If  $r=0$ then $I=\langle x^2+bxz+cyz,xy\rangle_2$
with some $b,c\in\mathbb{R}$.
If $b\neq0$ then $(0,0,1),(0,1,0)\in Allow(I)$ and in addition any
point
in $S^2$ that satisfies $(y=0\text{ } \&\text{ } z=\frac{-x}{b})$
lies in $Allow(I)$,
which is a contradiction to our assumption that $\dim\text{span}_{\mathbb{R}}Allow(I)=2$. We conclude that  $I=\langle x^2+cyz,xy\rangle_2$
with some $c\in\mathbb{R}$. If $c=0$ then $I=\langle x^2, xy\rangle_2$ and we are in case (\ref{model_for_2_3_with_2_dim_allow_4}). Otherwise, by scaling $z$ we get $I=\langle x^2+yz,xy\rangle_2$. Now, in the coordinate system $(X,Y,Z)=(x,-z,y)$ we have $I=\langle XZ,YZ-X^{2}\rangle_2$, and we are in case (\ref{model_for_2_3_with_2_dim_allow_6}).

\

\textbf{Case B.} Assume $s\neq0$, and so $I=\langle x^2+axy+bxz+cyz,qxy+rxz+syz\rangle_2$ with $s\neq0$. By scaling $y$ we have $I=\langle x^2+axy+bxz,yz+qxy+rxz\rangle_2$
with some new $a,b,q,r\in\mathbb{R}$. Scaling $x,y$ and $z$, and possibly renaming $y\leftrightarrow z$  we may assume $(q,r)\in\{(0,0),(0,1),(1,1)\}$.

\

\textbf{Case B.1.} Assume $(q,r)=(0,0)$, and so $I=\langle x^2+axy+bxz,yz\rangle_2$
with $a,b\in\mathbb{R}$. Scaling $y$ and $z$, and possibly renaming
$y\leftrightarrow z$  we may assume $(a,b)\in\{(0,0),(1,0),(1,1)\}$. In the first case we get $I=\langle x^2,yz\rangle_2$ and we are in case (\ref{model_for_2_3_with_2_dim_allow_5}). In the
second case we get $I=\langle x^2+xy,yz\rangle_2$, and so $(\frac{1}{\sqrt2},-\frac{1}{\sqrt2},0),(0,1,0),(0,0,1)\in Allow(I)$, which is a contradiction to our assumption that $\dim\text{span}_{\mathbb{R}}Allow(I)=2$.
 In the third case we get $I=\langle x^2+xy+xz,yz\rangle_2$ and so $(\frac{1}{\sqrt2},-\frac{1}{\sqrt2},0),(0,1,0),(0,0,1)\in
Allow(I)$, which is again a contradiction to our assumption that $\dim\text{span}_{\mathbb{R}}Allow(I)=2$. 

\

\textbf{Case B.2.} Assume $(q,r)=(0,1)$, and so $I=\langle x^2+axy+bxz,yz+xz\rangle_2$
with $a,b\in\mathbb{R}$.  Scaling $z$ we may assume either $b=0$ of $b=1$. 

If $b=0$ then $I=\langle x^2+axy,yz+xz\rangle_2$. If $a\neq0$ then  $(\frac{-a}{\sqrt{a^2+1}},\frac{1}{\sqrt{a^2+1}},0),(0,1,0),(0,0,1)\in
Allow(I)$, which is again a contradiction to our assumption that $\dim\text{span}_{\mathbb{R}}Allow(I)=2$. So we must have $a=0$ and so $I=\langle x^2,yz+xz\rangle_2$. In the coordinate system $(x,Y,z)=(x,x+y,z)$ we have $I=\langle x^2,Yz\rangle_2$ and we are in case (\ref{model_for_2_3_with_2_dim_allow_5}).

If $b=1$ then $I=\langle x^2+axy+xz,yz+xz\rangle_2$
with $a\in\mathbb{R}$. If $a\neq0$ then  $(\frac{-a}{\sqrt{a^2}+1},\frac{1}{\sqrt{a^2+1}},0),(0,1,0),(0,0,1)\in
Allow(I)$, which is again a contradiction to our assumption that $\dim\text{span}_{\mathbb{R}}Allow(I)=2$.
 So we must have $a=0$ and so $I=\langle x^2+xz,yz+xz\rangle_2$. However, in this case we get that $(\frac{1}{\sqrt3},-\frac{1}{\sqrt3},-\frac{1}{\sqrt3}),(0,1,0),(0,0,1)\in
Allow(I)$, which is once again a contradiction to our assumption that $\dim\text{span}_{\mathbb{R}}Allow(I)=2$.

\

\textbf{Case B.3.} Assume $(q,r)=(1,1)$, and so $I=\langle x^2+axy+bxz,xy+xz+yz\rangle_2$
with $a,b\in\mathbb{R}$. If $a=b=0$ then $I=\langle x^2,xy+xz+yz\rangle_2$ and we are in case (\ref{model_for_2_3_with_2_dim_allow_7}). So for the remaining cases we may assume $a^2+b^2\neq0$. Note that we always have $(0,1,0),(0,0,1)\in Allow(I)$. So by our assumption that $\dim\text{span}_{\mathbb{R}}Allow(I)=2$ any direction $(x_0,y_0,z_0)\in Allow(I)$ satisfies $x_0=0$. Define the quadratic form $Q(y,z):=yz-(ay+bz)(y+z)$. Note that since $a^2+b^2\neq0$ this quadratic form is not the zero quadratic form -- if $a\neq0$ then it has a non-zero $y^2$ term, and if $b\neq0$ then it has a non-zero
$z^2$ term. Then, a direction $(x_0,y_0,z_0)\in Allow(I)$ if and only if one of the following holds: (i) $x_0=0$ and $y_0z_{0}=0$; (ii) $x_0=-ay_0-bz_{0}$ and $Q(y_0,z_0)=0$. 

Assume $ab=0$. In that case (by possibly renaming $y\leftrightarrow z$) we
may assume $a=0$ and $b\neq0$, $I=\langle x^2+bxz,xy+xz+yz\rangle_2$ and $Q(y,z)=yz-bz(y+z)=z[(1-b)y-bz]$. If $b\neq1$ then there exists $z_0\neq0$ such that $(y_0,z_0)=(\frac{b}{1-b}z_0,z_0)\in S^1$ is such that $Q(y_0,z_0)=0$.
We set $x_0=-bz_0\neq0$
and then $(\frac{x_0}{\sqrt{x_0^2+y_0^2+z_0^2}},\frac{y_0}{\sqrt{x_0^2+y_0^2+z_0^2}},\frac{z_0}{\sqrt{x_0^2+y_0^2+z_0^2}})\in Allow(I)$ with $\frac{x_0}{\sqrt{x_0^2+y_0^2+z_0^2}}\neq0$, which is a contradiction. So $b=1$ and we have $I=\langle x^2+xz,xy+xz+yz\rangle_2=\langle
x^{2} +xz,-x^{2} +xy+yz\rangle_2$. In the coordinate system $(X,Y,Z)=(x,y,x+z)$
we get $I=\langle XZ,YZ-X^{2} \rangle_2$, i.e., we are in case (\ref{model_for_2_3_with_2_dim_allow_6}). So for the remaining cases we may assume $ab\neq0$.

Recall that $Allow(I)=\{(0,0,\pm1),(0,\pm1,0)\}$.
Since $ab\neq 0$ we have $-\frac{x^3}{ay+bz}$ is negligible for $Allow(I)$ and $S_1(x,y,z):=\frac{x}{ay+bz}$
satisfies the bounds in (\ref{second_cond_of_implied_poly_in_direction_SECOND_PAPER})
in any allowed direction, and so $x^2=(\frac{x}{ay+bz})\cdot x(x+ay+bz)-\frac{x^3}{ay+bz}$
is implied by $I$ by Corollary \ref{cor_strong_directional_implication_imply_implication_SECOND_PAPER}. As $I$ is closed $x^2,x(ay+bz),xy+xz+yz\in I$,
which is a contradiction to $\dim I=2$ as these 3 jets are clearly linearly
independent.
\end{proof}

\begin{remark}\label{remark_on_distinctness_for_closed_ideals_in_C2R3_2_dim_allow}The
ideals (\ref{model_for_2_3_with_2_dim_allow_1})-(\ref{model_for_2_3_with_2_dim_allow_7})
are pairwise distinct, i.e.,
there does not exist a $C^2$ coordinate change around the origin such that
a given ideal can take two different forms from this list. (\ref{model_for_2_3_with_2_dim_allow_1}) is the only ideal of dimension 4 and so clearly distinct. (\ref{model_for_2_3_with_2_dim_allow_2}) and (\ref{model_for_2_3_with_2_dim_allow_3}) are the only ideals of dimension 3, however the set of allowed directions of (\ref{model_for_2_3_with_2_dim_allow_3}) contains a copy of $S^2$, while the set of allowed directions of (\ref{model_for_2_3_with_2_dim_allow_2}) does not. Thus, (\ref{model_for_2_3_with_2_dim_allow_2})
and (\ref{model_for_2_3_with_2_dim_allow_3}) are distinct. (\ref{model_for_2_3_with_2_dim_allow_4}), (\ref{model_for_2_3_with_2_dim_allow_5}), (\ref{model_for_2_3_with_2_dim_allow_6}) and (\ref{model_for_2_3_with_2_dim_allow_7}) are all of dimension 2, however (\ref{model_for_2_3_with_2_dim_allow_4}) is the only one whose set of allowed directions contains a copy of $S^2$, and so it is distinct. Also, (\ref{model_for_2_3_with_2_dim_allow_6}) does not contain any jet that is a square of a linear form (while (\ref{model_for_2_3_with_2_dim_allow_5}) and (\ref{model_for_2_3_with_2_dim_allow_7}) clearly do), and so (\ref{model_for_2_3_with_2_dim_allow_6}) is distinct too. Finally, let us show that (\ref{model_for_2_3_with_2_dim_allow_5})
and (\ref{model_for_2_3_with_2_dim_allow_7}) are distinct too: note that as both ideals contain only homogenous jets of order 2, if there exists a $C^2$ coordinate change around the origin that maps (\ref{model_for_2_3_with_2_dim_allow_5})
to (\ref{model_for_2_3_with_2_dim_allow_7}), the linearization of this coordinate change also  maps (\ref{model_for_2_3_with_2_dim_allow_5})
to (\ref{model_for_2_3_with_2_dim_allow_7}). So if  such a coordinate change exists, there exists a linear coordinate change  that maps (\ref{model_for_2_3_with_2_dim_allow_5})
to (\ref{model_for_2_3_with_2_dim_allow_7}). Moreover, this linear map  maps the sets of allowed directions to the sets of allowed directions, which are both $\{(0,0,\pm1),(0,\pm1,0)\}$ in this case. So this is a linear map that maps $y$ to some scalar multiple of either $y$ or $z$, and maps $z$ to some scalar multiple of either $y$ or $z$. In particular, as this map is invertible, it sends the jet $yz$ to some scalar multiple of itself, so it does not map (\ref{model_for_2_3_with_2_dim_allow_5})
to (\ref{model_for_2_3_with_2_dim_allow_7}). We thus proved that there does not exist a
$C^2$ coordinate change around the origin that maps (\ref{model_for_2_3_with_2_dim_allow_5})
to (\ref{model_for_2_3_with_2_dim_allow_7}), which is a contradiction to the existence of such a coordinate change.\end{remark}

\begin{lemma}\label{lemma_classify_closed_ideals_in_C2R3_3_dim_allow}Let
$I\lhd\mathcal{P}_0^2(\mathbb{R}^{3})$
be a closed non-principal ideal. Assume $I$ does not contain any jet of
order of vanishing 1 and that $\dim\text{span}_{\mathbb{R}}Allow(I)=3$. Then,
$I$ is, up to a
linear
coordinate change, one of the following:

\begin{gather}\label{model_for_2_3_with_3_dim_allow_1}\langle xy,yz,xz\rangle_2\end{gather}
\begin{gather}\label{model_for_2_3_with_3_dim_allow_2}\langle xy,yz\rangle_2\end{gather}
\begin{gather}\label{model_for_2_3_with_3_dim_allow_3}\langle xy,z(x+y)\rangle_2\end{gather}
\begin{gather}\label{model_for_2_3_with_3_dim_allow_3'_another_missed_example}\langle xy+yz,xz+yz\rangle_2\end{gather}

\end{lemma}

\begin{proof}As $\dim\text{span}_{\mathbb{R}}Allow(I)=3$, by applying a linear coordinate change we may assume $(1,0,0),(0,1,0),(0,0,1)\in Allow(I)$ and so $I$ is a linear subspace of $\text{span}_{\mathbb{R}}\{xy,yz,xz\}$. As $I$ is not principal we have $\dim I\in\{2,3\}$. If $\dim I=3$ then we are in case (\ref{model_for_2_3_with_3_dim_allow_1}). We are left to check the cases in which $\dim I=2$. 

If no jet in $I$ is of the form $axy+bxz+cyz$ with $b\neq 0$, then $I=\langle xy,yz\rangle_2$ and we are in case (\ref{model_for_2_3_with_3_dim_allow_2}). Similar argument applies for no jet with $a\neq0$ or no jet with $c\neq0$. 

So it remains to check the case $I=\langle axy+bxz+cyz,a'xy+b'xz+c'yz\rangle_2$ for some $a,a',b,b',c,c'\in\mathbb{R}$ with $a^2+a'^2\neq0$, $b^2+b'^2\neq0$ and $c^2+c'^2\neq0$. By scaling $x$ we get $I=\langle xy+b''xz+c''yz,b'''xz+c'''yz\rangle_2$
for some $b'',b''',c'',c''''\in\mathbb{R}$ with $b''^2+b'''^2\neq0$, $b'''^2+c'''^2\neq0$
and $c''^2+c'''^2\neq0$. Now without loss of generality $b'''\neq0$ (otherwise rename $x\leftrightarrow y$) and scaling $z$ we get $I=\langle xy+Cyz,xz+C'yz\rangle_2$ for some $C,C'\in \mathbb{R}$ such that $C^2+C'^2\neq0$. Now without loss of generality $C'\neq0$ (otherwise
rename $y\leftrightarrow z$) and scaling $y$ we get $I=\langle xy+\alpha yz,xz+yz\rangle_2$, for some $\alpha\in\mathbb{R}$. If $\alpha=0$ then $I=\langle xy,z(x+y)\rangle_2$ and we are in case (\ref{model_for_2_3_with_3_dim_allow_3}). Otherwise, scaling $z$ we get $I=\langle xy+yz,xz+yz\rangle_2$, i.e., we are in case (\ref{model_for_2_3_with_3_dim_allow_3'_another_missed_example}).\end{proof}

\begin{remark}\label{remark_on_distinctness_for_closed_ideals_in_C2R3_3_dim_allow}The
ideals (\ref{model_for_2_3_with_3_dim_allow_1})-(\ref{model_for_2_3_with_3_dim_allow_3'_another_missed_example})
are pairwise distinct, i.e.,
there does not exist a $C^2$ coordinate change around the origin such that
a given ideal can take two different forms from this list. (\ref{model_for_2_3_with_3_dim_allow_1})
is the only ideal of dimension 3 and so clearly distinct. The set of allowed directions of (\ref{model_for_2_3_with_3_dim_allow_2}) contains a copy of $S^2$, while those of (\ref{model_for_2_3_with_3_dim_allow_3}) and (\ref{model_for_2_3_with_3_dim_allow_3'_another_missed_example}) do not, and so (\ref{model_for_2_3_with_3_dim_allow_2}) is distinct too. Finally, the set of allowed directions of (\ref{model_for_2_3_with_3_dim_allow_3}) is $\{(0,0,\pm1),(0,\pm1,0),(\pm1,0,0)\}$ and so has cardinality 6, while the set of allowed direction of (\ref{model_for_2_3_with_3_dim_allow_3'_another_missed_example}) is $\{(0,0,\pm1),(0,\pm1,0),(\pm1,0,0),\pm(-\frac{1}{\sqrt{3}},\frac{1}{\sqrt{3}},\frac{1}{\sqrt{3}})\}$ and so has cardinality
8, thus (\ref{model_for_2_3_with_3_dim_allow_3}) and (\ref{model_for_2_3_with_3_dim_allow_3'_another_missed_example}) are distinct too.\end{remark}

\begin{proposition}\label{prop_C2R3_case_new_version} Let $I\lhd\mathcal{P}_0^2(\mathbb{R}^{3})$
be an ideal. Then, the following are equivalent:
\begin{itemize}
\item $I$ is closed.
\item There exists a closed set $\vec 0 \in E\subset\mathbb{R}^3$ such that
$I=I^2(E)$.
\item There exists a closed semi-algebraic set $\vec 0 \in E\subset\mathbb{R}^3$
such that $I=I^2(E)$.
\end{itemize}\end{proposition}
\begin{proof}
If $I$ is principal then the Proposition follows from Proposition \ref{prop_principal_ideals_in_P2Rn}
. Otherwise, by Theorem \ref{main_theorem_on_necessary_condition_SECOND_PAPER}, Lemmas \ref{lemma_C2R3_case__with_jet_of_oov_1}, \ref{lemma_classify_closed_ideals_in_C2R3_1_dim_allow}, \ref{lemma_classify_closed_ideals_in_C2R3_2_dim_allow} and \ref{lemma_classify_closed_ideals_in_C2R3_3_dim_allow}
it is enough to show
that each of the ideals in (\ref{model_for_2_3_with_1_dim_allow_1}), (\ref{model_for_2_3_with_1_dim_allow_2}), (\ref{model_for_2_3_with_1_dim_allow_3}), (\ref{model_for_2_3_with_1_dim_allow_3'_the_one_we_missed}), (\ref{model_for_2_3_with_1_dim_allow_4}), (\ref{model_for_2_3_with_1_dim_allow_5}), (\ref{model_for_2_3_with_2_dim_allow_1}), (\ref{model_for_2_3_with_2_dim_allow_2}), (\ref{model_for_2_3_with_2_dim_allow_3}), (\ref{model_for_2_3_with_2_dim_allow_4}), (\ref{model_for_2_3_with_2_dim_allow_5}), (\ref{model_for_2_3_with_2_dim_allow_6}), (\ref{model_for_2_3_with_2_dim_allow_7}), (\ref{model_for_2_3_with_3_dim_allow_1}), (\ref{model_for_2_3_with_3_dim_allow_2}), (\ref{model_for_2_3_with_3_dim_allow_3}) and (\ref{model_for_2_3_with_3_dim_allow_3'_another_missed_example}) arises as $I^2(E)$ for some closed
semi-algebraic set $E\subset\mathbb{R}^3$ containing the origin. This is
established in Examples \ref{example_for_2_3_xx_yy}, \ref{example_for_2_3_fancy_with_parameter}, \ref{example_for_2_3_xy_xx_yy}, \ref{example_for_2_3_xy_xx_yy+xz}, \ref{example_for_2_3_xz_xy_xx_yy}, \ref{example_for_2_3_yz_xz_xy_xx_yy}, \ref{example_for_2_3_yz_xz_xy_xx}, \ref{example_for_2_3_yz_xz_xx}, \ref{example_for_2_3_xz_xy_xx}, \ref{example_for_2_3_xy_xx}, \ref{example_for_2_3_yz_xx}, \ref{example_for_2_3_xz_yz-xx}, \ref{example_for_2_3_xx_xy+xz+yz}, \ref{example_for_2_3_xy_yz_xz}, \ref{example_for_2_3_xy_yz}, \ref{example_for_2_3_xy_z(x+y)} and \ref{example_for_2_3_xy+yz_xz+yz} below.\end{proof}

\begin{example}\label{example_for_2_3_xx_yy}There exists a semi-algebraic
set $\vec 0\in E\subset\mathbb{R}^3$ such that $\langle x^2,y^{2}\rangle_2=I^2(E)$. 

Indeed, this follows immediately from the fact that $ Allow(\langle x^2,y^{2}\rangle_2)\neq\emptyset$
and Proposition \ref{prop_ideals_in_P2Rn_generated_by_2_foreign_jets} (the set $E=\{x^2-(y^2+z^2)^2=y^2-(x^2+z^2)^2=0\}$ works; see the proof of Proposition \ref{prop_ideals_in_P2Rn_generated_by_2_foreign_jets}).\end{example}

\begin{example}\label{example_for_2_3_fancy_with_parameter}There exists a semi-algebraic
set $\vec 0\in E\subset\mathbb{R}^3$ such that $\langle xy,y^2-x^2\rangle_2=I^2(E)$. 

Indeed, by applying the linear coordinate change $(X,Y)=(x,y-x)$ the above ideal takes the form $\langle X(Y+X),X^2+Y^2+2XY-X^{2} \rangle_2=\langle X(X+Y), Y(Y+2X)\rangle_2$. Thus, it is enough to show that there exists a
semi-algebraic
set $\vec 0\in E\subset\mathbb{R}^3$ such that $\langle x(x+y),y(y+2x)\rangle_2=I^2(E)$. We will show that $$E=\{(x,y,z)\in\mathbb{R}^2|(x-z^{10})(x+y-z^{10})=(y-z^{10})(x+\frac{1}{2}y-2z^{10})=0\}$$ satisfies $\langle x(x+y),y(y+2x)\rangle_2=I^2(E)$. We start by understanding how $E$ looks like, by writing $E$ as a union of four sets: $E=E_1\cup E_2\cup E_3\cup E_4$ where $$E_1:=\{x=z^{10},y=z^{10}\};$$ $$E_2:=\{x=z^{10},x+\frac{1}{2}y=2z^{10}\}=\{x=z^{10},y=2z^{10}\};$$ $$E_3:=\{x+y=z^{10},y=z^{10}\}=\{x=0,y=z^{10}\};$$ $$E_4:=\{x+y=z^{10},x+\frac{1}{2}y=2z^{10}\}=\{x=3z^{10},y=-2z^{10}\}.$$So fixing 4 distinct points in the $x-y$ plane $$(x_1^0,y_1^0):=(1,1);\text{ }(x_2^0,y_2^0):=(1,2);\text{
}(x_3^0,y_3^0):=(0,1);\text{
}(x_4^0,y_4^0):=(3,-2),$$ we can write \begin{gather}\label{new_presentation_of_E_very_cool}E=\{(x_\nu^0 z^{10},y_\nu^0 z^{10},z)|\nu\in\{1,2,3,4\},z\in\mathbb{R}\}.\end{gather} For $F\in
C^2(\mathbb{R}^2)$ we set the notation $\abs{\nabla^2F(x,y)}:=\max\{\abs{\frac{\partial^2F}{\partial x^2}(x,y)},\abs{\frac{\partial^2F}{\partial
y^2}(x,y)},\abs{\frac{\partial^2F}{\partial
x\partial y}(x,y)}\}$ (the maximum norm of the Hessian matrix). We now prove the following: \begin{multline}\label{lemma_inside_the_proof_label_bounding_nabla_square}\text{There exists a positive constant }\tilde c>0\text{ such that} \\ \text{if }F\in C^2(\mathbb{R}^2)\text{ is such that }F(x_\nu^0,y_\nu^0)=x_\nu^0y_\nu^0\text{ for all }\nu\in\{1,2,3,4\}\\ \text{then }\max_\Omega\abs{\nabla^2F}>\tilde c,\text{ where }\Omega\text{  is the convex hull of }\{(x_\nu^0,y_\nu^0)\}_{\nu=1}^4.\end{multline}Indeed, let $F$ be as assumed, denote $M:=\max_\Omega\abs{\nabla^2F}$, and let $L(x,y)=\alpha+\beta x+\gamma y$ be the first order Taylor polynomial of $F$ at $(x_1^0,y_1^0)$. By Taylor's Theorem $$\max_\Omega\abs{F(x,y)-L(x,y)}\leq c_0 M,$$for some constant $c_0$. In particular, for $\nu\in\{1,2,3,4\}$ we have $$x_\nu^0 y_\nu^0=F(x_\nu^0,y_\nu^0)=L(x_\nu^0,y_\nu^0)+\epsilon_\nu,\text{ with }\abs{\epsilon_\nu}\leq c_0M,$$or explicitly, \begin{gather}\label{evaluation_at_a_point_in_a_secret_lemma_1}\alpha+\beta+\gamma=1+\epsilon_1\end{gather}
\begin{gather}\label{evaluation_at_a_point_in_a_secret_lemma_2}\alpha+\beta+2\gamma=2+\epsilon_2\end{gather}
\begin{gather}\label{evaluation_at_a_point_in_a_secret_lemma_3}\alpha+\gamma=\epsilon_3\end{gather}
\begin{gather}\label{evaluation_at_a_point_in_a_secret_lemma_4}\alpha+3\beta-2\gamma=-6+\epsilon_4.\end{gather}
Now, (\ref{evaluation_at_a_point_in_a_secret_lemma_1}) and (\ref{evaluation_at_a_point_in_a_secret_lemma_3}) imply that $\beta=1+[\epsilon_1-\epsilon_3]$; (\ref{evaluation_at_a_point_in_a_secret_lemma_1}) and (\ref{evaluation_at_a_point_in_a_secret_lemma_2})
imply that $\gamma=1-[\epsilon_1-\epsilon_2]$; (\ref{evaluation_at_a_point_in_a_secret_lemma_3})
implies that $\alpha=-\gamma+\epsilon_3=-1+[\epsilon_1-\epsilon_2+\epsilon_3]$. Substituting into the LHS of (\ref{evaluation_at_a_point_in_a_secret_lemma_4}) we get $-1+[\epsilon_1-\epsilon_2+\epsilon_3]+3(1+[\epsilon_1-\epsilon_3])-2(1-[\epsilon_1-\epsilon_2])=-1+3-2+[6\epsilon_1-3\epsilon_2-2\epsilon_3]=-6+\epsilon_4$, or \begin{gather}\label{epsilon_comparison_game_to_get_lower_bound}6\epsilon_1-3\epsilon_2-2\epsilon_3-\epsilon_4=-6.\end{gather} Recalling that $\abs{\epsilon_\nu}\leq c_0M$, we see that (\ref{epsilon_comparison_game_to_get_lower_bound}) implies that $6=\abs{6\epsilon_1-3\epsilon_2-2\epsilon_3-\epsilon_4}\leq12\cdot\max\{\abs{\epsilon_\nu}\}_{\nu=1,2,3,4}\leq12 c_0 M$, and so $M\geq \frac{1}{2c_0}$, i.e., (\ref{lemma_inside_the_proof_label_bounding_nabla_square}) holds. Note that (\ref{lemma_inside_the_proof_label_bounding_nabla_square}) implies the following:
\begin{multline}\label{corollary_from_secret_lemma_bounding_nabla_square}\text{There
exists a positive constant }\check c>0\text{
such that} \\ \text{if we fix }z\in\mathbb{R}^\times\text{ and if }\check F\in C^2(\mathbb{R}^2)\text{ is such that }\check F(x_\nu^0z^{10},y_\nu^0z^{10})=z^{20}x_\nu^0y_\nu^0\text{
for all }\nu\in\{1,2,3,4\}\\ \text{then }\max_\Omega\abs{\nabla^2F}>\check
c,\text{ where }\Omega\text{  is the convex hull of }\{(x_\nu^0z^{10},y_\nu^0z^{10})\}_{\nu=1}^4.\end{multline}
To see this, apply (\ref{lemma_inside_the_proof_label_bounding_nabla_square}) to $F(x,y):=z^{-20}\check F(xz^{10},yz^{10})$ and so (\ref{corollary_from_secret_lemma_bounding_nabla_square}) holds. 

We are now ready to calculate $I^2(E)$. First note that we have
$$(x-z^{10})(x+y-z^{10})=(y-z^{10})(x+\frac{1}{2}y-2z^{10})\in
C^2(\mathbb{R}^3)$$ and both vanish on $E$. So $x(x+y),y(y+2x)\in
I^2(E)$. So it is enough to show that if $f\in C^2(\mathbb{R}^3)$ is such that
$$f(x,y,z)=\tilde A x+\tilde B y+\tilde C z+\tilde Dxy+\tilde E z^2+\tilde R xz+\tilde S yz +\tilde f(x,y,z),$$ for some $\tilde A,\tilde B,\tilde C,\tilde D,\tilde
E,\tilde R,\tilde S\in\mathbb{R}$ and some $\tilde f(x,y,z)=o(x^2+y^2+z^2)$, and moreover $f|_{E}=0$, then $\tilde A=\tilde B=\tilde C=\tilde D=\tilde E=\tilde R=\tilde S=0$. Let $f$ be as assumed. Recalling (\ref{new_presentation_of_E_very_cool}) and $(x_1^0,y_1^0):=(1,1)$ we get that for any $z\in\mathbb{R}$ $$0=f(z^{10},z^{10},z)=\tilde A z^{10}+\tilde B z^{10}+\tilde C z+\tilde Dz^{20}+\tilde E z^2+\tilde
R z^{10}z+\tilde S z^{10}z +o(z^{2}),$$which implies $\tilde C=\tilde E=0$, and therefore $f(x,y,z)=\tilde A x+\tilde B y+\tilde Dxy+\tilde
R xz+\tilde S yz +\tilde f(x,y,z)$. Assume $\tilde D\neq 0$. For any fixed $z\in\mathbb{R}^\times$ we then define $$\check F(x,y):=\tilde D^{-1}\big(-[\tilde A+\tilde Rz]x-[\tilde B+\tilde Sz]y-\tilde f(x,y,z)\big)\in C^2(\mathbb{R}^2).$$ Since $f|_{E}=0$, we have that $\check F(x_\nu^0z^{10},y_\nu^0z^{10})=z^{20}x_\nu^0y_\nu^0$ for all $\nu\in\{1,2,3,4\}$. By (\ref{corollary_from_secret_lemma_bounding_nabla_square}) there exists a point $(x(z),y(z))$ in the convex hull of $\{(x_\nu^0z^{10},y_\nu^0z^{10})\}_{\nu=1}^4$ satisfying $\abs{\nabla^2\check F(x(z),y(z))}>\tilde c$. Note that $(x(z),y(z),z)\to\vec0$ as $z\to0$. On the other hand $\abs{\nabla^2\check F(x,y)}=\tilde D^{-1}\abs{\nabla_{(x,y)}^2\tilde f(x,y,z)}=o(1)$, which is a contradiction (here $\nabla_{(x,y)}^2$ is the partial Hessian with respect
to the first two coordinates only). We conclude that $\tilde D=0$ and therefore $f(x,y,z)=\tilde A x+\tilde B y+\tilde
R xz+\tilde S yz +\tilde f(x,y,z)$. For any $z\in\mathbb{R}$ and $\nu\neq\nu'\in\{1,2,3,4\}$ by (\ref{new_presentation_of_E_very_cool})
we have $f(x_\nu^0
z^{10},y_\nu^0 z^{10},z)=f(x_{\nu'}^0
z^{10},y_{\nu'}^0 z^{10},z)=0$. For $t>0$
 define $$g(t)=f(tx_\nu^0
z^{10}+(1-t)x_{\nu'}^0
z^{10},ty_\nu^0 z^{10}+(1-t)y_{\nu'}^0 z^{10},z).$$
By Rolle's Theorem there exists $t'\in[0,1]$ such that $$0=\frac{d}{dt}g(t')=[(x_\nu^0
z^{10}-x_{\nu'}^0
z^{10})f_{x}+(y_\nu^0
z^{10}-y_{\nu'}^0
z^{10})f_{y}](t'x_\nu^0
z^{10}+(1-t')x_{\nu'}^0
z^{10},t'y_\nu^0 z^{10}+(1-t')y_{\nu'}^0 z^{10},z).$$
Setting  $(\tilde x(z),\tilde y(z),z):=(t'x_\nu^0
z^{10}+(1-t')x_{\nu'}^0
z^{10},t'y_\nu^0 z^{10}+(1-t')y_{\nu'}^0 z^{10},z)$
we found that\begin{multline}\label{another_label_for_fancy_rolle_analysis_1}\text{for any }z>0\text{ and any }\nu\neq\nu'\in\{1,2,3,4\}\\\text{ there exist }\tilde x(z),\tilde y(z)\text{ such that }[(x_\nu^0
-x_{\nu'}^0
)f_{x}+(y_\nu^0
-y_{\nu'}^0
)f_{y}](\tilde
x(z),\tilde y(z),z)=0\\\text{ and moreover both }\tilde
x(z)=o(z)\text{ and }\tilde
y(z)=o(z).\end{multline}Taking $z\to0$ we in particular get  $$[(x_\nu^0
-x_{\nu'}^0
)f_{x}+(y_\nu^0
-y_{\nu'}^0
)f_{y}](\vec0)=0\text{ for any }\nu,\nu'\in\{1,2,3,4\}.$$Substituting $\nu=1$ and $\nu'=2$ we get $(1-1)f_x(\vec 0)+(1-2)f_y(\vec 0)=0$, which implies $f_y(\vec0)=0$, i.e., $\tilde B=0$.
Similarly, substituting
$\nu=1$ and $\nu'=3$ we get $(1-0)f_x(\vec 0)+(1-1)f_y(\vec 0)=0$, which
implies $f_x(\vec0)=0$, i.e., $\tilde A=0$. Therefore, $f(x,y,z)=\tilde R xz+\tilde S yz +\tilde f(x,y,z)$. In particular, $f_x(x,y,z)=\tilde Rz+o((x^2+y^2+z^2)^{1/2})$ and $f_y(x,y,z)=\tilde Sz+o((x^2+y^2+z^2)^{1/2})$. Substituting these expressions in (\ref{another_label_for_fancy_rolle_analysis_1}) we get that for any $z>0$ $$[(x_\nu^0
-x_{\nu'}^0
)\tilde Rz+(y_\nu^0
-y_{\nu'}^0
)\tilde Sz]+o(z)=0,$$ which is only possible if $(x_\nu^0
-x_{\nu'}^0
)\tilde R+(y_\nu^0
-y_{\nu'}^0
)\tilde S=0$. Substituting $\nu=1$
and $\nu'=2$ we get $(1-1)\tilde R+(1-2)\tilde S=0$, which implies
$\tilde S=0$.
Similarly, substituting
$\nu=1$ and $\nu'=3$ we get $(1-0)\tilde R+(1-1)\tilde S=0$, which
implies $\tilde R=0$. We have shown that indeed $\tilde A=\tilde B=\tilde C=\tilde D=\tilde
E=\tilde R=\tilde S=0$ and so $\langle x(x+y),y(y+2x)\rangle_2=I^2(E)$.
\end{example}

\begin{example}\label{example_for_2_3_xy_xx_yy}Set $\delta=10^{-10}$ and $E=E_1\cup E_2\cup E_3$, with $$E_1=\{(0,0,t)|t\in[0,\delta]\};$$ $$E_2=\{(t^{3/2},t^{3/2},t)|t\in[0,\delta]\};$$ $$E_3=\{(0,t^{3/2},t)|t\in[0,\delta]\}.$$
Then, $\langle x^2,y ^2,xy\rangle_2=I^2(E)$.

Indeed, let $\theta(x,y,z)\in
C^\infty(\mathbb{R}^3\setminus\{\vec 0\})$ be a homogenous function of degree $0$ supported in $\{(x,y,z)\in\mathbb{R}^3|\abs{(x,y)}\leq\abs{z}\}$
such that $\theta|_{\{(x,y,z)\in\mathbb{R}^3|\abs{(x,y)}\leq\frac{1}{2}\abs{z}\}}\equiv1$.
We have
$$x(x-\abs{z}^{3/2}\cdot\theta),y(y-\abs{z}^{3/2}\cdot\theta),x(y-\abs{z}^{3/2}\cdot\theta)\in
C^2(\mathbb{R}^3)$$ and all three vanish on $E$ (see footnote \ref{footnote_on_function_vanishes_on_a_set}). So $x^{2},y^{2},xy\in
I^2(E)$. So it is enough to show that if $f\in C^2(\mathbb{R}^3)$ is such that
$$f(x,y,z)=ax+by+cz+dxz+eyz+gz^{2}+o(x^2+y^2+z^2)$$
for some $a,b,c,d,e,g\in \mathbb{R}$ and $f|_{E}=0$, then $a=b=c=d=e=g=0$.
Let $f$ be as assumed. Since $f|_{E_1}=0$ we have for any $t\in[0,\delta]$: $$0=f(0,0,t)=ct+gt^{2}+o(t^2),$$which
implies $c=g=0$, and so $f(x,y,z)=ax+by+dxz+eyz+o(x^2+y^2+z^2)$. Since $f|_{E_3}=0$ we have for any $t\in[0,\delta]$:
$$0=f(0,t^{3/2},t)=bt^{3/2}+et^{5/2}+o(t^2)=bt^{3/2}+o(t^2),$$which
implies $b=0$, and so $f(x,y,z)=ax+dxz+eyz+o(x^2+y^2+z^2)$. Since $f|_{E_2}=0$ we have for any $t\in[0,\delta]$:
$$0=f(t^{3/2},t^{3/2},t)=at^{3/2}+dt^{5/2}+et^{5/2}+o(t^2)=at^{3/2}+o(t^2),$$which
implies $a=0$, and so $f(x,y,z)=dxz+eyz+o(x^2+y^2+z^2)$. We have for any $t\in(0,\delta]$, $f(0,t^{3/2},t)=f(t^{3/2},t^{3/2},t)=0$.
By Rolle's Theorem there exists $\tilde {x}(t)$ with $\abs{\tilde {x}(t)}\leq
t^{3/2}$
such that $f_{x}(\tilde {x}(t),t^{3/2},t)=0$. On the other hand we
have $f_x(x,y,z)=dz+o((x^2+y^2+z^{2})^{1/2})$. So in particular we have for
any $t\in(0,\delta]$:$$0=f_x(\tilde {x}(t),t^{3/2},t)=dt+o(t),$$which
is only possible if $d=0$, and so $f(x,y,z)=eyz+o(x^2+y^2+z^2)$. Similarly, we have
for any $t\in(0,\delta]$, $f(0,0,t)=f(0,t^{3/2},t)=0$.
By Rolle's Theorem there exists $\tilde {y}(t)$ with $\abs{\tilde {y}(t)}\leq
t^{3/2}$
such that $f_{y}(0,\tilde {y}(t),t)=0$. On the other hand we
have $f_y(x,y,z)=ez+o((x^2+y^2+z^{2})^{1/2})$. So in particular we have for
any $t\in(0,\delta]$:$$0=f_y(0,\tilde {y}(t),t)=et+o(t),$$which
is only possible if $e=0$. We 
conclude that indeed $\langle x^2,y^{2},xy\rangle_2=I^2(E)$.
\end{example}

\begin{example}\label{example_for_2_3_xy_xx_yy+xz}Set $a=10^{-10}$, $c_0=-10^{-1}$ and $$E=\{(x,y,z)\in\mathbb{R}^3|z>0,y^2+xz=z^{2+2a},xy=c_0z^{2+3a}\}\cup\{\vec 0\}.$$
Then, $\langle x^2,y ^2+xz,xy\rangle_2=I^2(E)$.

Indeed, we start by analyzing $E$. The conditions $z>0$ and $xy=c_0z^{2+3a}$ imply that any point $(x,y,z)\in E\setminus\{\vec 0\}$ satisfies $xyz\neq0$. So we can rewrite 
$$E=\{(x,y,z)\in\mathbb{R}^3|z>0,y^2+xz=z^{2+2a},x=\frac{c_0z^{2+3a}}{y}\}\cup\{\vec
0\}$$ $$=\{(x,y,z)\in\mathbb{R}^3|z>0,y^2+\frac{c_0z^{3+3a}}{y}=z^{2+2a},x=\frac{c_0z^{2+3a}}{y}\}\cup\{\vec
0\}$$ $$=\{(x,y,z)\in\mathbb{R}^3|z>0,y^3+c_0z^{3+3a}=z^{2+2a}y,x=\frac{c_0z^{2+3a}}{y}\}\cup\{\vec
0\}$$ $$=\{(x,y,z)\in\mathbb{R}^3|z>0,\big(\frac{y}{z^{1+a}}\big)^{3}-\big(\frac{y}{z^{1+a}}\big)+c_0=0,x=\frac{c_0z^{2+3a}}{y}\}\cup\{\vec
0\}.$$
One can easily see that the
monic polynomial $p(t)=t^3-t+c_{0}$ has 3 distinct non-zero real roots, which we
denote by $\alpha_0<\alpha_1<\alpha_2$. So on $E$ we have $y=\alpha_{\nu}z^{1+a}$ for some $\nu\in\{0,1,2\}$ and $x=\frac{c_0z^{2+3a}}{y}=(\frac{c_0}{\alpha_\nu})z^{1+2a}$ for some $\nu\in\{0,1,2\}$. Combining the above we get \begin{gather}\label{new_descr_of_forgotten_E}E=\{([\frac{c_0}{\alpha_\nu}]
z^{1+2a},\alpha_\nu z^{1+a},z)|\nu\in\{0,1,2\},z\geq0\}.\end{gather}
We start constructing $C^2$ functions that vanish on $E$. First, note that $$y^2+xz-\abs{z}^{2+2a},xy-c_0\abs{z}^{2+3a}\in C^2(\mathbb{R}^3)$$ and both vanish on $E$. So $y^2+xz,xy\in
I^2(E)$. We will now show that $x^2\in I^2(E)$ as well. Let $\gamma_0,\gamma_1,\gamma_2$ be the (unique) solution of the linear (Vandermonde, hence invertible) system \begin{gather}\label{nice_label_forgotten_example}[\frac{c_0}{\alpha_\nu}]^2=\gamma_0+\alpha_{\nu}\gamma_1+\alpha_\nu^2\gamma_2\text{ };\text{
}\nu\in\{0,1,2\}.\end{gather} Let $\theta(x,y,z)\in
C^\infty(\mathbb{R}^3\setminus\{\vec 0\})$ be a homogenous function of degree $0$ supported in $\{(x,y,z)\in\mathbb{R}^3:\abs{(x,y)}\leq\abs{z}\}$
such that $\theta|_{\{(x,y,z)\in\mathbb{R}^3|\abs{(x,y)}\leq\frac{1}{2}\abs{z}\}}\equiv1$.
We have
$$g(x,y,z):=x^2-(\gamma_0\abs{z}^{2+4a}+\gamma_1\abs{z}^{1+3a}y+\gamma_2\abs{z}^{2a}y^2) \cdot \theta\in
C^2(\mathbb{R}^3).$$ We claim that $g$ vanishes on $E\cap B(\epsilon)$ for some small $\epsilon>0$. Indeed, by (\ref{new_descr_of_forgotten_E}) for some small $\epsilon>0$ we have $\theta|_{E\cap B(\epsilon)}\equiv1$. We moreover have on $E$ that $x^2=[\frac{c_0}{\alpha_\nu}]^2z^{2+4a}$ and $y=\alpha_\nu z^{1+a}$ for some $\nu\in\{0,1,2\}$. So on $E\cap B(\epsilon)$ we have by (\ref{nice_label_forgotten_example}) that for any $\nu\in\{0,1,2\}$: $$x^2=[\frac{c_0}{\alpha_\nu}]^2z^{2+4a}=\big(\gamma_0+\alpha_{\nu}\gamma_1+\alpha_\nu^2\gamma_2\big)z^{2+4a}$$ $$=\gamma_0z^{2+4a}+\gamma_1z^{1+3a}y+\gamma_2z^{2a}y^2=(\gamma_0\abs{z}^{2+4a}+\gamma_1\abs{z}^{1+3a}y+\gamma_2\abs{z}^{2a}y^2)
\cdot \theta,$$so indeed $g$ vanishes on $E\cap B(\epsilon)$. We conclude that $x^2\in I^2(E)$. It is thus enough to show that if $f\in C^2(\mathbb{R}^3)$ is such that
$$f(x,y,z)=Ax+by+cz+dy^{2}+eyz+gz^{2}+o(x^2+y^2+z^2)$$
for some $A,b,c,d,e,g\in \mathbb{R}$ and $f|_{E}=0$, then $A=b=c=d=e=g=0$.
Let $f$ be as assumed. By (\ref{new_descr_of_forgotten_E}), for any $z>0$ we have $$0=f([\frac{c_0}{\alpha_0}]
z^{1+2a},\alpha_0 z^{1+a},z)=A[\frac{c_0}{\alpha_0}]
z^{1+2a}+b\alpha_0 z^{1+a}+cz+gz^{2}+o(z^2),$$ which implies $A=b=c=g=0$ (since $c_0,\alpha_0\neq0$), and so $f(x,y,z)=dy^{2}+eyz+o(x^2+y^2+z^2)$. We fix
a $C^\infty$-smooth compactly supported one-variable function $\psi(t)$,
such that $\psi(\alpha_\nu)=\frac{c_0}{\alpha_\nu}$ for $\nu\in\{0,1,2\}$ and set for $z>0$ $$F(t):=f(\psi(t)z^{1+2a},tz^{1+a},z).$$ By (\ref{new_descr_of_forgotten_E}) we have $F(\alpha_0)=F(\alpha_1)=F(\alpha_2)=0$. Recalling $\alpha_0<\alpha_1<\alpha_2$, by Rolle's Theorem there exist $t_1\in(\alpha_0,\alpha_1)$ and $t_2\in(\alpha_0,\alpha_2)$
such that $F'(t_1)=0$ and $F''(t_2)=0$. We have $$F'(t)=\psi'(t)z^{1+2a}\partial_xf(\psi(t)z^{1+2a},tz^{1+a},z)+z^{1+a}\partial_yf(\psi(t)z^{1+2a},tz^{1+a},z);$$ $$F''(t):=\psi''(t)z^{1+2a}\partial_xf(\psi(t)z^{1+2a},tz^{1+a},z)+\big(\psi'(t)\big)^2z^{2+4a}\partial^{2}_{xx}f(\psi(t)z^{1+2a},tz^{1+a},z)$$ $$+2\psi'(t)z^{2+3a}\partial^{2}_{xy}f(\psi(t)z^{1+2a},tz^{1+a},z)+z^{2+2a}\partial^2_{yy}f(\psi(t)z^{1+2a},tz^{1+a},z).$$ Recall $f(x,y,z)=dy^{2}+eyz+o(x^2+y^2+z^2)$, so for $z>\sqrt{x^2+y^2}$ we have $$\partial_xf=o(z);\text{ }\partial_yf=2dy+ez+o(z);\text{ }\partial_{xx}^2f=o(1);\text{
} \partial_{xy}^2f=o(1);\text{ } \partial_{yy}^2f=2d+o(1).$$Substituting and
using the fact that $\abs{t_1}<\max\{\abs{\alpha_0},\abs{\alpha_1}\}$ and
$\abs{\psi'(t)}<\tilde C$, where $\tilde C$ is a constant depending only
on $\alpha_0,\alpha_1$ and $\alpha_2$, we get that for $z>\sqrt{x^2+y^2}$
and $y=o(z)$ we have $$0=F'(t_1)=\psi'(t_{1})z^{1+2a}\cdot o(z)+z^{1+a}(2dy+ez+o(z))=ez^{2+a}+o(z^{2+a}),$$ which is only possible if $e=0$. Similarly,
we now also use the facts that both $t_2$ and $\psi''(t)$ are bounded by
constants depending only on $\alpha_0,\alpha_1$ and $\alpha_2$, to get $$0=F''(t_{2})=\psi''(t_{2})z^{1+2a}\cdot o(z)+\big(\psi'(t_{2})\big)^2z^{2+4a}\cdot
o(1)$$
$$+2\psi'(t_{2})z^{2+3a}\cdot
o(1)+z^{2+2a}\cdot(2d+o(1))=2dz^{2+2a}+o(z^{2+2a}),$$
which is only possible if $d=0$. We conclude that indeed $\langle x^2,y ^2+xz,xy\rangle_2=I^2(E)$.

\end{example}

\begin{example}\label{example_for_2_3_xz_xy_xx_yy}Set $\delta=10^{-10}$ and $E=E_1\cup
E_2$, with $$E_1=\{(t^{3/2},0,t)|t\in[0,\delta]\};$$
$$E_2=\{(t^{3/2},t^{3/2},t)|t\in[0,\delta]\}.$$
Then, $\langle x^2,y ^2,xy,xz\rangle_2=I^2(E)$.

Indeed, let $\theta(x,y,z)\in
C^\infty(\mathbb{R}^3\setminus\{\vec 0\})$ be a homogenous function of degree $0$ supported in $\{(x,y,z)\in\mathbb{R}^3|\abs{(x,y)}\leq\abs{z}\}$
such that $\theta|_{\{(x,y,z)\in\mathbb{R}^3|\abs{(x,y)}\leq\frac{1}{2}\abs{z}\}}\equiv1$.
We have
$$x(x-\abs{z}^{3/2}\cdot\theta),y(y-\abs{z}^{3/2}\cdot\theta),y(x-\abs{z}^{3/2}\cdot\theta),z(x-\abs{z}^{3/2}\cdot\theta)\in
C^2(\mathbb{R}^3)$$ and all four vanish on $E$ (see footnote \ref{footnote_on_function_vanishes_on_a_set}). So $x^{2},y^{2},xy,xz\in
I^2(E)$. It is thus enough to to show that if $f\in C^2(\mathbb{R}^3)$ is such that
$$f(x,y,z)=ax+by+cz+dyz+ez^{2}+o(x^2+y^2+z^2)$$
for some $a,b,c,d,e\in \mathbb{R}$ and $f|_{E}=0$, then $a=b=c=d=e=0$.
Let $f$ be as assumed. Since $f|_{E_1}=0$ we have for any $t\in[0,\delta]$:
$$0=f(t^{3/2},0,t)=at^{3/2}+ct+et^{2}+o(t^2),$$which
implies $a=c=e=0$, and so $f(x,y,z)=by+dyz+o(x^2+y^2+z^2)$.
Since $f|_{E_2}=0$ we have for any $t\in[0,\delta]$:
$$0=f(t^{3/2},t^{3/2},t)=bt^{3/2}+dt^{5/2}+o(t^2)=bt^{3/2}+o(t^2),$$which
implies $b=0$, and so $f(x,y,z)=dyz+o(x^2+y^2+z^2)$. We have
for any $t\in(0,\delta]$, $f(t^{3/2},0,t)=f(t^{3/2},t^{3/2},t)=0$.
By Rolle's Theorem there exists $\tilde {y}(t)$ with $\abs{\tilde {y}(t)}\leq
t^{3/2}$
such that $f_{y}(t^{3/2},\tilde {y}(t),t)=0$. On the other hand we
have $f_y(x,y,z)=dz+o((x^2+y^2+z^{2})^{1/2})$. So in particular we have for
any $t\in(0,\delta]$:$$0=f_y(t^{3/2},\tilde {y}(t),t)=dt+o(t),$$which
is only possible if $d=0$. We 
conclude that indeed $\langle x^2,y^{2},xy,xz\rangle_2=I^2(E)$.
\end{example}

\begin{example}\label{example_for_2_3_yz_xz_xy_xx_yy}Set $\alpha=5/3$, $\beta=3/2$, $\delta=10^{-10}$ and $$E=\{(t^\alpha,t^\beta,t)|t\in[0,\delta]\}.$$
Then, $\langle x^2,y^{2},xy,xz,yz\rangle_2=I^2(E)$.

Indeed, we first present some $C^2$ functions that vanish on $E$. Let $\theta(x,y,z)\in
C^\infty(\mathbb{R}^3\setminus\{\vec 0\})$ be a homogenous function of degree $0$ supported in $\{(x,y,z)\in\mathbb{R}^3|\abs{(x,y)}\leq\abs{z}\}$
such that $\theta|_{\{(x,y,z)\in\mathbb{R}^3|\abs{(x,y)}\leq\frac{1}{2}\abs{z}\}}\equiv1$.
We have
$$(x-\abs{z}^\alpha\cdot\theta)^2,(y-\abs{z}^\beta\cdot\theta)^2,(x-\abs{z}^\alpha\cdot\theta) \cdot (y-\abs{z}^\beta\cdot\theta),(x-\abs{z}^\alpha\cdot\theta)\cdot z,(y-\abs{z}^\beta\cdot\theta)\cdot z\in
C^2(\mathbb{R}^3)$$ and all five vanish on $E$ (see footnote \ref{footnote_on_function_vanishes_on_a_set}). So $x^{2},y^{2},xy,xz,yz\in
I^2(E)$. It is left to show that if $f\in C^2(\mathbb{R}^3)$ is such that
$$f(x,y,z)=ax+by+cz+dz^2+o(x^2+y^2+z^2)$$
for some $a,b,c,d\in \mathbb{R}$ and $f|_{E}=0$, then $a=b=c=d=0$.
Let $f$ be as assumed. Since $f|_E=0$ we have for any $t\in[0,\delta]$: $$0=f(t^\alpha,t^\beta,t)=at^\alpha+bt^\beta+ct+dt^2+o(t^2),$$which
implies $a=b=c=d=0$. We 
conclude that indeed $\langle x^2,y^{2},xy,xz,yz\rangle_2=I^2(E)$.
\end{example}

\begin{example}\label{example_for_2_3_yz_xz_xy_xx}Set $\gamma=0.5005$, $\delta=0.005$ and $$E=\{(x,y,z)\in\mathbb{R}^3|y\geq0,z\geq0,x=(x^2+y^2+z^2)^\gamma,2yz=(x^2+y^2+z^2)^{1+\delta}\}.$$ Then, $\langle x^2,xy,yz,xz\rangle_2=I^2(E)$.

Indeed, we first note that $1<2\gamma<1+\delta$. We start by finding points in $E$. Setting $r=(x^2+y^2+z^2)^{1/2}$, we can write $$E=\{(x,y,z)\in\mathbb{R}^3|y\geq0,z\geq0,x=r^{2\gamma},2yz=r^{2+2\delta}\}.$$
So on $E$ we have $r^2=x^2+y^2+z^2=r^{4\gamma}+y^2+z^2$, or $y^2+z^2=r^2-r^{4\gamma}$.
Additionally we have on $E$, $yz=\frac{1}{2}r^{2+2\delta}$, so $(y+z)^2=r^2-r^{4\gamma}+r^{2+2\delta}$
and $(y-z)^2=r^2-r^{4\gamma}-r^{2+2\delta}$. Since on $E$ both $y$ and $z$
are non negative, we have $y+z=(r^2-r^{4\gamma}+r^{2+2\delta})^{1/2}$, and (for $r$ small enough) $\abs{y-z}=(r^2-r^{4\gamma}-r^{2+2\delta})^{1/2}$.
We can now define $$\phi_{\pm}(r):=\frac{1}{2}(y+z\pm\abs{y-z})=\frac{1}{2}\big[(r^2-r^{4\gamma}+r^{2+2\delta})^{1/2}\pm(r^2-r^{4\gamma}-r^{2+2\delta})^{1/2}\big].$$Note
that $\phi_+(r)=\Theta(r)$ and  $\phi_-(r)=\Theta(r^{1+2\delta})$. Setting $x(r)=r^{2\gamma}$, $y(r)=\phi_+(r)$ and finally $z(r)=\phi_-(r)$
we conclude that \begin{multline}\label{nice_example_label_parametrization_of_a_part_of_E}\text{for any }0<r<10^{-50}\text{ there exists }x(r),y(r),z(r)\text{ such that }(x(r),y(r),z(r))\in E, \\ \text{and moreover  }x(r)=\Theta(r^{2\gamma}),y(r)=\Theta(r)\text{ and }z(r)=\Theta(r^{1+2\delta}).\end{multline}
Recall that $1<2\gamma<1+\delta$, we now note that 
$$x^2-(x^2+y^2+z^2)^{2\gamma},xy-(x^2+y^2+z^2)^{\gamma}y,xz-(x^2+y^2+z^2)^{\gamma}z,zy-\frac{1}{2}(x^2+y^2+z^2)^{1+\delta}\in
C^2(\mathbb{R}^3)$$ and all four vanish on $E$. So $x^{2},xy,xz,yz\in I^2(E)$. It is thus enough to show that if $f\in C^2(\mathbb{R}^3)$ is such that $$f(x,y,z)=ax+by+cz+dy^2+ez^2+o(x^2+y^2+z^2)$$
for some $a,b,c,d,e\in \mathbb{R}$ and $f|_{E}=0$, then $a=b=c=d=e=0$.
Let $f$ be as assumed. By (\ref{nice_example_label_parametrization_of_a_part_of_E}) we have for any $0<r<10^{-50}$: $$0=f(x(r),y(r),z(r))=a\Theta(r^{2\gamma})+b\Theta(r)+c\Theta(r^{1+2\delta})+d\Theta(r^{2})+e\Theta(r^{2+4\delta})+o(r^2),$$ which implies $a=b=c=d=0$, and so $f(x,y,z)=ez^2+o(x^2+y^2+z^2)$. Note that $E$ is symmetric with respect to $y\leftrightarrow z$, so we also have for any $0<r<10^{-50}$: $$0=f(x(r),z(r),y(r))=e(y(r))^{2}+o(r^2)=e\Theta(r^{2})+o(r^2),$$which implies $e=0$. We 
conclude that indeed $\langle x^2,xy,yz,xz\rangle_2=I^2(E)$.

\end{example}

\begin{example}\label{example_for_2_3_yz_xz_xx}Set $\alpha=1+10^{-5}$, $\beta=1+10^{-15}$, $p=1+10^{-25}$, $q=1+10^{-35}$, $\delta=10^{-10^{45}}$ 
and $$E_+=\{(t^p,t^q,t)|t\in[0,\delta]\};$$ $$E_-=\{(-t^p,t^q,t)|t\in[0,\delta]\};$$ $$E_0=\{(t^\alpha,t,t^\beta)|t\in[0,\delta]\};$$ $$E=E_+\cup E_-\cup E_0.$$
Then, $\langle x^2,xy,yz\rangle_2=I^2(E)$.

Indeed, we first present some $C^2$ functions that vanish on $E$. Let $\theta_1(x,y,z)\in C^\infty(\mathbb{R}^3\setminus\{\vec 0\})$ be a homogenous function of degree $0$ supported in $\{(x,y,z)\in\mathbb{R}^3|\abs{(x,y)}\leq\frac{1}{10}\abs{z}\}$ such that $\theta_1|_{\{(x,y,z)\in\mathbb{R}^3|\abs{(x,y)}\leq\frac{1}{20}\abs{z}\}}\equiv1$. Let $\theta_0(x,y,z)\in C^\infty(\mathbb{R}^3\setminus\{\vec 0\})$ be a homogenous
function of degree $0$ supported in $\{(x,y,z)\in\mathbb{R}^3|\abs{(x,z)}\leq\frac{1}{10}\abs{y}\}$
such that $\theta_0|_{\{(x,y,z)\in\mathbb{R}^3|\abs{(x,z)}\leq\frac{1}{20}\abs{y}\}}\equiv1$.
Note that on $E_+\setminus\{\vec 0\}$ and on $E_-\setminus\{\vec 0\}$ we have $\theta_1=1$ and $\theta_0=0$, while on $E_0\setminus\{\vec 0\}$ we have $\theta_1=0$ and $\theta_0=1$. Formally set $\theta_0(\vec 0)=\theta_1(\vec 0)=0$; we have
$$x^2-\abs{z}^{2p}\theta_1-\abs{y}^{2\alpha}\theta_0,xy-x\abs{z}^{q}\theta_1-\abs{y}^{\alpha+1}\theta_0,yz-\abs{z}^{q+1}\theta_1-\abs{y}^{\beta+1}\theta_0\in
C^2(\mathbb{R}^3)$$ and all three vanish on $E$. So $x^{2},xy,yz\in I^2(E)$. It is thus enough to show that if $f\in C^2(\mathbb{R}^3)$ is such that
$$f(x,y,z)=ax+by+cz+dy^2+ez^2+gxz+o(x^2+y^2+z^2)$$
for some $a,b,c,d,e,g\in \mathbb{R}$ and $f|_{E}=0$, then $a=b=c=d=e=g=0$.
Let $f$ be as assumed. Since $f|_{E_+}=0$ we have for any $t\in[0,\delta]$: $$0=f(t^p,t^q,t)=at^p+bt^q+ct+dt^{2q}+et^2+gt^{1+p}+o(t^{2}),$$
which implies $a=b=c=e=0$, and so $f(x,y,z)=dy^2+gxz+o(x^2+y^2+z^2)$. Since $f|_{E_0}=0$ we have for any $t\in[0,\delta]$: $$0=f(t^\alpha,t,t^\beta)=dt^2+gt^{\alpha+\beta}+o(t^{2}),$$which implies $d=0$, and so $f(x,y,z)=gxz+o(x^2+y^2+z^2)$.
As $f|_{E_+\cup E_-}=0$ we have for any $t\in(0,\delta]$, $f(-t^p,t^q,t)=f(t^p,t^q,t)=0$. By Rolle's Theorem there exists $\tilde {x}(t)$ with $\abs{\tilde {x}(t)}\leq t^p$
such that $f_{x}(\tilde {x}(t),t^{q},t)=0$. On the other hand we
have $f_x(x,y)=gz+o((x^2+y^2+z^{2})^{1/2})$. So in particular we have for any $t\in(0,\delta]$:$$0=f_x(\tilde {x}(t),t^{q},t)=gt+o(t),$$which
is only possible if $g=0$. We 
conclude that indeed $\langle x^2,xy,yz\rangle_2=I^2(E)$.
\end{example}

\begin{example}\label{example_for_2_3_xz_xy_xx}$ 
\langle x^2,xy,xz\rangle_2=I^2(\{(x,y,z)\in\mathbb{R}^2|x=(y^2+z^2)^{3/4}\})$.

Indeed, setting $E=\{(x,y,z)\in\mathbb{R}^2|x=(y^2+z^2)^{3/4}\}$ we have $$x^2-(y^2+z^2)^{3/2},xy-y(y^2+z^2)^{3/4},xz-z(y^2+z^2)^{3/4}\in
C^2(\mathbb{R}^3)$$ and all three vanish on $E$. So $x^{2},xy,xz\in I^2(E)$.
It is thus enough to show that if $f\in C^2(\mathbb{R}^3)$ is such that $$f(x,y,z)=ax+by+cz+dyz+ey^2+gz^2+o(x^2+y^2+z^2)$$
for some $a,b,c,d,e,g\in \mathbb{R}$ and $f|_{E}=0$, then $a=b=c=d=e=g=0$.
Let $f$ be as assumed. Then, for any $(y,z)\in\mathbb{R}^2$ we have $$0=f((y^2+z^2)^{3/4},y,z)$$ $$=a(y^2+z^2)^{3/4}+by+cz+dyz+ey^2+gz^2+o((y^2+z^2)^{3/2}+y^2+z^2)$$ $$=a(y^2+z^2)^{3/4}+by+cz+dyz+ey^2+gz^2+o(y^2+z^2),$$which
implies $a=b=c=d=e=g=0$. We  conclude that indeed $\langle x^2,xy,xz\rangle_2=I^2(E)$.
\end{example}

\begin{example}\label{example_for_2_3_xy_xx}Recall the cones and domes of the forms $\Gamma(\Omega,\delta,r)$ and $D(\Omega,\delta)$ as appear Section \ref{section_settings}. Set $a=10^{-15}$, $b=10^{-5}$
and $E=E_y\cup E_z\cup E_{1}$, with $$E_y=\{(x,y,z)\in\mathbb{R}^3|x=\abs{y}^{1+a},z=\abs{y}^{1+a}\}\cap\overline{\Gamma((0,1,0),10^{-10},1)};$$
$$E_z=\{(x,y,z)\in\mathbb{R}^3|x^2=\abs{z}^{2+2a},xy=\abs{z}^{2+b}\}\cap\overline{\Gamma((0,0,1),10^{-10},1)};$$
$$E_1=\{(x,y,z)\in\mathbb{R}^3|x=\abs{z}^{1+a},y=z\}\cap\overline{\Gamma((0,2^{-1/2},2^{-1/2}),10^{-10},1)}.$$

Then, $\langle x^2,xy\rangle_2=I^2(E)$.

Indeed, we start by noting that both $x^2-\abs{y}^{2+2a}$ and $xy-\abs{y}^{2+a}$ vanish on $E_y$; both $x^2-\abs{z}^{2+2a}$ and $xy-\abs{z}^{2+b}$
vanish on $E_z$; both $x^2-\abs{z}^{2+2a}$ and $xy-\abs{z}^{2+a}$
vanish on $E_1$. We define subsets of $S^2$: $$U_y:=D((0,1,0),10^{-9})\text{ };\text{ }U_z:=D((0,0,1),10^{-9})\text{ };\text{ }U_1:=D((0,2^{-1/2},2^{-1/2}),10^{-9})\text{ };$$ $$U_0=S^2\setminus \big(\overline{D((0,1,0),10^{-10})} \cup \overline{D((0,0,1),10^{-10})} \cup \overline{D((0,2^{-1/2},2^{-1/2}),10^{-10})}\big).$$Let
$\tilde\theta_y,\tilde\theta_z,\tilde\theta_1,\tilde\theta_0$ be a $C^\infty$ partition
of unity subordinate to the open cover $U_y,U_z,U_{1},U_0$, i.e., for all $\nu\in\{y,z,1,0\}$:
$\tilde\theta_\nu\in C^{\infty}(S^2)$, $\text{support}(\tilde\theta_\nu)\subset
U_\nu$ and $\tilde\theta_y(\omega)+\tilde\theta_z(\omega)+\tilde\theta_1(\omega)+\tilde\theta_0(\omega)=1$
for any $\omega\in S^2$. For all $\nu\in\{y,z,1,0\}$ define a function $\theta_\nu:\mathbb{R}^3\to\mathbb{R}$
by $\theta_\nu(\vec0)=0$ and $\theta_\nu(x,y,z)=\tilde\theta_\nu\big(\frac{x}{\sqrt{x^2+y^2+z^2}},\frac{y}{\sqrt{x^2+y^2+z^2}},\frac{z}{\sqrt{x^2+y^2+z^2}}\big)$
for any $(x,y,z)\neq\vec0$. Then, we have
$$(x^2-\abs{y}^{2+2a})\theta_y+(x^2-\abs{z}^{2+2a})\theta_z+(x^2-\abs{z}^{2+2a}) \theta_1+x^2\theta_0\in
C^2(\mathbb{R}^3)$$ $$(xy-\abs{y}^{2+a})\theta_y+(xy-\abs{z}^{2+b})\theta_z+(xy-\abs{z}^{2+a})
\theta_1+xy\theta_0\in
C^2(\mathbb{R}^3)$$and both functions vanish on $E$. Their 2-jets are $x^2$ and $xy$
and so $x^{2},xy\in I^2(E)$. It is thus left to show that if $f\in C^2(\mathbb{R}^3)$
is such that $$f(x,y,z)=Ax+By+Cz+Py^{2}+Qz^{2}+Sxz+Tyz+o(x^2+y^2+z^2)$$
for some $A,B,C,P,Q,S,T\in \mathbb{R}$ and $f|_{E}=0$, then $A=B=C=P=Q=S=T=0$.
Let $f$ be as assumed. We present some points in $E$:
\begin{gather}\label{correcting_example_9_15_label_1}\text{for }\sigma\in\{\pm1\}\text{ and for any }0<z<10^{-25},\text{ } (x_\sigma,y_\sigma,z_\sigma):=(\sigma z^{1+a},\sigma z^{1+(b-a)},z)\in E_z\subset E;\end{gather} \begin{gather}\label{correcting_example_9_15_label_2}\text{for any }0<y<10^{-25},\text{ } (\tilde x,\tilde y,\tilde z):=(y^{1+a},y,y^{1+a})\in E_y\subset E;\end{gather} \begin{gather}\label{correcting_example_9_15_label_3}\text{for
any }0<z<10^{-25},\text{ } (x_\#,y_\#,z_\#):=(z^{1+a},z,z)\in
E_1\subset E.\end{gather}
Substituting (\ref{correcting_example_9_15_label_1}) we get that for any $0<z<10^{-25}$ $$0=f(x_{+1},y_{+1},z_{+1})=f(
z^{1+a},z^{1+(b-a)},z)=A
z^{1+a}+Bz^{1+(b-a)}+Cz+Qz^{2}+o(z^2),$$which implies $A=B=C=Q=0$ and so $f(x,y,z)=Py^{2}+Sxz+Tyz+o(x^2+y^2+z^2)$. Substituting (\ref{correcting_example_9_15_label_2}) we get that for any $0<y<10^{-25}$ $$0=f(\tilde x,\tilde y,\tilde z)=f(y^{1+a},y,y^{1+a})=Py^{2}+o(y^2),$$
which implies $P=0$ and so
$f(x,y,z)=Sxz+Tyz+o(x^2+y^2+z^2)$. Substituting (\ref{correcting_example_9_15_label_3})
we get that for any $0<z<10^{-25}$ $$0=f(x_\#,y_\#,z_\#)=f(z^{1+a},z,z)=Tz^{2}+o(z^2),$$which implies $T=0$ and so
$f(x,y,z)=Sxz+o(x^2+y^2+z^2)$. We in particular have \begin{gather}\label{correcting_example_9_15_label_4}f_x(x,y,z)=Sz+o((x^2+y^2+z^2)^{1/2})\text{ };\text{ } f_y(x,y,z)=o((x^2+y^2+z^2)^{1/2}).\end{gather}

Note that (\ref{correcting_example_9_15_label_1}) implies that for any fixed $0<z<10^{-25}$ we have $f(\sigma
z^{1+a},\sigma z^{1+(b-a)},z)=0$ for both $\sigma\in\{\pm1\}$. Define $F(t):=f(tz^{1+a},t z^{1+(b-a)},z)$ and apply Rolle's Theorem; there exists $-1<t(z)<1$ such that $F'(t)=z^{1+a}f_x+z^{1+(b-a)}f_y$ vanishes at $(t(z)z^{1+a},t(z)z^{1+(b-a)},z)$. Since $b>2a$, we showed that \begin{multline}\label{correcting_example_9_15_label_5}\text{for any }0<z<10^{-25}\text{ there exists }t(z)\in(-1,1)\text{ such that }\\(f_x+z^{b-2a}f_y)(t(z)z^{1+a},t(z)z^{1+(b-a)},z)=0.\end{multline}
Finally, combining (\ref{correcting_example_9_15_label_4}) and (\ref{correcting_example_9_15_label_5}) we see that for any $0<z<10^{-25}$ $$0=f_x(t(z)z^{1+a},t(z)z^{1+(b-a)},z)+o((x^2+y^2+z^2)^{1/2})=Sz+o(z),$$which
implies $S=0$ and we conclude that indeed $\langle x^2,xy\rangle_2=I^2(E)$.\end{example}

\begin{example}\label{example_for_2_3_yz_xx}There exists a semi-algebraic
set $\vec 0\in E\subset\mathbb{R}^3$ such that $\langle x^2,yz\rangle_2=I^2(E)$.

Indeed, this follows immediately from the fact that $ Allow(\langle x^2,yz\rangle_2)\neq\emptyset$ and Proposition \ref{prop_ideals_in_P2Rn_generated_by_2_foreign_jets} (the reader is referred to the proof of Proposition \ref{prop_ideals_in_P2Rn_generated_by_2_foreign_jets}
for an easy construction of such an $E$).
\end{example}

\begin{example}\label{example_for_2_3_xz_yz-xx}Set $\gamma=10^{-5}$
and $E=E_1\cup E_2$, with $$E_1=\{(x,y,z)\in\mathbb{R}^3|xz=\abs{z}^{2+\gamma},yz-x^2=\abs{z}^{2+2\gamma},z\geq10^{2}\sqrt{x^2+y^2}\};$$
$$E_2=\{(x,y,z)\in\mathbb{R}^3|xz=\abs{y}^{2+3\gamma},yz-x^2=-10\abs{y}^{2+2\gamma},y\geq10^{2}\sqrt{x^2+z^2}\}.$$
Then, $\langle xz,yz-x^2\rangle_2=I^2(E)$. 

Indeed, we start by analyzing points in $E$. One can easily see that the monic polynomial $p(y)=y^3-10y-1$ has 3 distinct non-zero real roots. We denote them by $\alpha_0<\alpha_1<\alpha_2$. Now, for $0<y<10^{-25}$ and $i\in\{0,1,2\}$
we set $$x_i(y):=\alpha_i y^{1+\gamma}\text{ and }z_i(y):=\alpha_i^{-1}y^{1+2\gamma}.$$We
then have $x_i(y)z_i(y)=\alpha_i y^{1+\gamma}\alpha_i^{-1}y^{1+2\gamma}=y^{2+3\gamma}$; $y\geq10^{2}\sqrt{x_i(y)^2+z_i(y)^2}$
and moreover $$yz_{i}(y)-x_{i}(y)^2+10y^{2+2\gamma}=y\alpha_i^{-1}y^{1+2\gamma}-(\alpha_i y^{1+\gamma})^2+10y^{2+2\gamma}$$ $$=-\alpha_i^{-1}
y^{2+2\gamma}[\alpha_i^3-10\alpha_i-1]=-\alpha_i^{-1}
y^{2+2\gamma}p(\alpha_i)=0.$$Thus, we have that $(x_i(y),y,z_i(y))\in E_2$, and explicitly \begin{gather}\label{label_how_does_E_2_look_like}\bigcup\limits_{i=1}^{3}\{(\alpha_i y ^{1+\gamma},y,\alpha_i^{-1}y^{1+2\gamma})|0<y<10^{-25}\}\subset
E_2.\end{gather} Similarly, for $0<z<10^{-25}$ we set $$\tilde x(z):= z^{1+\gamma}\text{ and }\tilde y(z):=2z^{1+2\gamma}.$$We
then have $\tilde x(z)z=z^{2+\gamma}$;
$z\geq10^{2}\sqrt{\tilde x(z)^2+\tilde y(z)^2}$
and moreover $$\tilde y(z)z-\tilde x(z)^2-z^{2+2\gamma}= 2z^{1+2\gamma}z-(z^{1+\gamma})^2-z^{2+2\gamma}=0.$$Thus, we have that $(\tilde x(z),\tilde y(z),z)\in E_1$,
and explicitly \begin{gather}\label{label_how_does_E_1_look_like}\{(z^{1+\gamma},2z^{1+2\gamma},z)|0<z<10^{-25}\}\subset
E_1.\end{gather} 

We now construct some $C^2$ functions that vanish on $E$ and start computing
$I(E)$. Let $$U_1:=\{(x,y,z)\in S^2:z>10\sqrt{x^2+y^2}\};$$
$$U_2:=\{(x,y,z)\in S^2:y>10\sqrt{x^2+z^2}\};$$
$$U_0:=\{(x,y,z)\in S^2:z<10^{2}\sqrt{x^2+y^2}\text{ and }y<10^{2}\sqrt{x^2+z^2}\}.$$Let
$\tilde\theta_0,\tilde\theta_1,\tilde\theta_2$ be a $C^\infty$ partition
of unity subordinate to the open cover $U_0,U_1,U_2$, i.e., for all $\nu\in\{0,1,2\}$:
$\tilde\theta_\nu\in C^{\infty}(S^2)$, $\text{support}(\tilde\theta_\nu)\subset
U_\nu$ and $\tilde\theta_0(\omega)+\tilde\theta_1(\omega)+\tilde\theta_2(\omega)=1$
for any $\omega\in S^2$. For all $\nu\in\{0,1,2\}$ define a function $\theta_\nu:\mathbb{R}^3\to\mathbb{R}$
by $\theta_\nu(\vec0)=0$ and $\theta_\nu(x,y,z)=\tilde\theta_\nu\big(\frac{x}{\sqrt{x^2+y^2+z^2}},\frac{y}{\sqrt{x^2+y^2+z^2}},\frac{z}{\sqrt{x^2+y^2+z^2}}\big)$
for any $(x,y,z)\neq\vec0$. Then, we have
$$(xz-\abs{z}^{2+\gamma})\theta_1+(xz-\abs{y}^{2+3\gamma})\theta_2+xz\theta_0\in
C^2(\mathbb{R}^3)$$ $$(yz-x^2-\abs{z}^{2+2\gamma})\theta_1+(yz-x^2+10\abs{y}^{2+2\gamma})\theta_2+(yz-x^2)\theta_0\in
C^2(\mathbb{R}^3)$$and both functions vanish on $E$. Their 2-jets are $xz$ and $yz-x^2$
and so $xz,yz-x^2\in I^2(E)$. It is thus enough to show that if $f\in C^2(\mathbb{R}^3)$
is such that $$f(x,y,z)=Ax+By+Cz+Pxy+Qx^{2}+Sy^2+Tz^2+o(x^2+y^2+z^2)$$
for some $A,B,C,P,Q,S,T\in \mathbb{R}$ and $f|_{E}=0$, then $A=B=C=P=Q=S=T=0$.
Let $f$ be as assumed. Then, for any $0<y<10^{-25}$ we have from (\ref{label_how_does_E_2_look_like}) that $$0=f(\alpha_0
y ^{1+\gamma},y,\alpha_0^{-1}y^{1+2\gamma})$$ $$=A\alpha_0
y ^{1+\gamma}+By+C\alpha_0^{-1}y^{1+2\gamma}+P\alpha_0
y ^{1+\gamma}y+Q(\alpha_0
y ^{1+\gamma})^{2}+Sy^2+T(\alpha_0^{-1}y^{1+2\gamma})^2+o(y^2)$$ $$=A\alpha_0
y ^{1+\gamma}+By+C\alpha_0^{-1}y^{1+2\gamma}+Sy^2+o(y^2),$$which implies $A=B=C=S=0$, and so $f(x,y,z)=Pxy+Qx^{2}+Tz^2+o(x^2+y^2+z^2)$. Similarly, for any $0<z<10^{-25}$ we have from (\ref{label_how_does_E_1_look_like})
that $$0=f(z^{1+\gamma},2z^{1+2\gamma},z)=2Pz^{1+\gamma}z^{1+2\gamma}+Q(z^{1+\gamma})^{2}+Tz^2+o(z^2)=Tz^2+o(z^2),$$which implies
$T=0$, and so $f(x,y,z)=Pxy+Qx^{2}+o(x^2+y^2+z^2)$. We fix a $C^\infty$-smooth compactly supported one-variable function $\psi(t)$, such that $\psi(\alpha_i)=\alpha_i^{-1}$ for $i\in\{0,1,2\}$ and set for $y>0$ $$F(t):=f(y^{1+\gamma} t,y,y^{1+2\gamma}\psi(t)).$$ Then, for any $0<y<10^{-25}$ we have from (\ref{label_how_does_E_2_look_like})
that  $F(\alpha_i)=f(\alpha_iy^{1+\gamma}
,y,\alpha_i^{-1}y^{1+2\gamma})=0$ for $i\in\{0,1,2\}$. Recalling $\alpha_0<\alpha_1<\alpha_2$, by Rolle's Theorem there exist $t_1\in(\alpha_0,\alpha_1)$ and $t_2\in(\alpha_0,\alpha_2)$ such that $F'(t_1)=0$ and $F''(t_2)=0$. We have $$F'(t)=y^{1+\gamma}\partial_x f(y^{1+\gamma}
t,y,y^{1+2\gamma}\psi(t))+y^{1+2\gamma}\psi'(t)\partial_z
f(y^{1+\gamma}
t,y,y^{1+2\gamma}\psi(t));$$ $$F''(t):=(y^{1+\gamma})^{2}\partial_{xx}^2
f(y^{1+\gamma}
t,y,y^{1+2\gamma}\psi(t))+2y^{1+\gamma}y^{1+2\gamma}\psi'(t)\partial_{xz}^2
f(y^{1+\gamma}
t,y,y^{1+2\gamma}\psi(t))$$ $$+(y^{1+2\gamma}\psi'(t))^{2}\partial_{zz}^2
f(y^{1+\gamma}
t,y,y^{1+2\gamma}\psi(t))+y^{1+2\gamma}\psi''(t)\partial_zf(y^{1+\gamma}
t,y,y^{1+2\gamma}\psi(t)).$$ Recall $f(x,y,z)=Pxy+Qx^{2}+o(x^2+y^2+z^2)$. For any $0<y<10^{-25}$ and while $y\geq10^{2}\sqrt{x^2+z^2}$ we calculate explicitly $$\partial_xf=Py+2Qx+o(y);\text{ }\partial_zf=o(y); \text{ }\partial_{xx}^2f=2Q+o(1);\text{ } \partial_{xz}^2f=o(1);\text{ } \partial_{zz}^2f=o(1).$$Substituting and using the fact that $\abs{t_1}<\max\{\abs{\alpha_0},\abs{\alpha_1}\}$ and $\abs{\psi'(t)}<\tilde C$, where $\tilde C$ is a constant depending only on $\alpha_0,\alpha_1$ and $\alpha_2$, we get $$0=F'(t_1)=y^{1+\gamma}(Py+2Qy^{1+\gamma}
t_{1}+o(y))+y^{1+2\gamma}\psi'(t_{1})\cdot o(y)$$ $$=y^{1+\gamma}(Py+2Qy^{1+\gamma}
t_1+o(y))=y^{1+\gamma}(Py+o(y)),$$which is only possible if $P=0$. Similarly, we now also use the facts that both $t_2$ and $\phi''(t)$ are bounded by constants depending only on $\alpha_0,\alpha_1$ and $\alpha_2$, to get $$0=F''(t_{2})=(y^{1+\gamma})^{2}(2Q+o(1))+2y^{1+\gamma}y^{1+2\gamma}\psi'(t_{2})(o(1))$$ $$+(y^{1+2\gamma}\psi'(t_{2}))^{2}(o(1))+y^{1+2\gamma}\psi''(t_{2})(o(y))$$ $$=y^{2+2\gamma}(2Q+o(1)),$$
which is only possible if $Q=0$. We conclude that indeed $\langle xz,yz-x^2\rangle_2=I^2(E)$.
\end{example}

\begin{example}\label{example_for_2_3_xx_xy+xz+yz}Set $a=10^{-5}$, $b=10^{-15}$
and $E=E_y\cup E_z$, with $$E_y=\{(x,y,z)\in\mathbb{R}^3|x^2=\abs{y}^{2+2a},xy+xz+yz=\abs{y}^{2+2b},y\geq10^{3}\sqrt{x^2+z^2}\};$$
$$E_z=\{(x,y,z)\in\mathbb{R}^3|x^2=\abs{z}^{2+2a},xy+xz+yz=\abs{z}^{2+2b},z\geq10^{3}\sqrt{x^2+y^2}\}.$$
Then, $\langle x^2,xy+xz+yz\rangle_2=I^2(E)$.

Indeed, we start by analyzing points in $E$. For $0<z<10^{-25}$ and $\sigma\in\{\pm1\}$ we set $$x_\sigma(z):=\sigma z^{1+a}\text{ and }y_\sigma(z):=\frac{z^{2+2b}-\sigma z^{2+a}}{z+\sigma z^{1+a}}=z^{1+2b}(\frac{1-\sigma z^{a-2b}}{1+\sigma z^a}).$$We then have $x_\sigma(z)^2=z^{2+2a}$; $z\geq10^{3}\sqrt{x_\sigma(z)^2+y_\sigma(z)^2}$ and moreover $$x_\sigma(z)y_\sigma(z)+x_\sigma(z)z+y_\sigma(z)z=\frac{z^{2+2b}-\sigma
z^{2+a}}{z+\sigma z^{1+a}}(\sigma
z^{1+a}+z)+\sigma
z^{1+a}z=z^{2+2b}.$$Thus, we have that $(x_\sigma(z),y_\sigma(z),z)\in E_z$. Note that $$y_\sigma(z)=z^{1+2b}(1-\sigma z^{a-2b})(1-\sigma z^a+O(z^{2a}))$$ $$=z^{1+2b}(1-\sigma z^{a-2b}-\sigma z^a+O(z^{2a-2b}))=z^{1+2b}-\sigma z^{1+a}+O(z^{1+a+2b}).$$So we showed that \begin{multline}\label{fancy_label_points_on_E_for_J5}\text{for any }0<z<10^{-25}\text{ and any }\sigma\in\{\pm1\}\text{ there exist }x_\sigma(z):=\sigma z^{1+a}\text{ and }y_\sigma(z)\\\text{such that } (x_\sigma(z),y_\sigma(z),z)\in E\text{ and moreover }y_\sigma(z)=\Theta(z^{1+2b}).\end{multline}

We now construct some $C^2$ functions that vanish on $E$ and start computing $I(E)$. Let $$U_y:=\{(x,y,z)\in S^2:y>10\sqrt{x^2+z^2}\};$$
$$U_z:=\{(x,y,z)\in S^2:z>10\sqrt{x^2+y^2}\};$$
$$U_0:=\{(x,y,z)\in S^2:y<10^{2}\sqrt{x^2+z^2}\text{ and }z<10^{2}\sqrt{x^2+y^2}\}.$$Let $\tilde\theta_y,\tilde\theta_z,\tilde\theta_0$ be a $C^\infty$ partition of unity subordinate to the open cover $U_y,U_z,U_0$, i.e., for all $\nu\in\{y,z,0\}$: $\tilde\theta_\nu\in C^{\infty}(S^2)$, $\text{support}(\tilde\theta_\nu)\subset U_\nu$ and $\tilde\theta_y(\omega)+\tilde\theta_z(\omega)+\tilde\theta_0(\omega)=1$ for any $\omega\in S^2$. For all $\nu\in\{y,z,0\}$ define a function $\theta_\nu:\mathbb{R}^3\to\mathbb{R}$ by $\theta_\nu(\vec0)=0$ and $\theta_\nu(x,y,z)=\tilde\theta_\nu\big(\frac{x}{\sqrt{x^2+y^2+z^2}},\frac{y}{\sqrt{x^2+y^2+z^2}},\frac{z}{\sqrt{x^2+y^2+z^2}}\big)$ for any $(x,y,z)\neq\vec0$. Then, we have
$$(x^2-y^{2+2a})\theta_y+(x^2-z^{2+2a})\theta_z+x^2\theta_0\in
C^2(\mathbb{R}^3)$$ $$(xy+xz+yz-y^{2+2b})\theta_y+(xy+xz+yz-z^{2+2b})\theta_z+(xy+xz+yz)\theta_0\in
C^2(\mathbb{R}^3)$$and both functions vanish on $E$. Their 2-jets are $x^2$ and $xy+xz+yz$ and so $x^{2},xy+xz+yz\in I^2(E)$. It is thus enough to show that if $f\in C^2(\mathbb{R}^3)$ is such that $$f(x,y,z)=Ax+By+Cz+Py^{2}+Qz^{2}+Sxy+Txz+o(x^2+y^2+z^2)$$
for some $A,B,C,P,Q,S,T\in \mathbb{R}$ and $f|_{E}=0$, then $A=B=C=P=Q=S=T=0$.
Let $f$ be as assumed. Then, for $\sigma\in\{\pm1\}$ and $0<z<10^{-25}$ we have by (\ref{fancy_label_points_on_E_for_J5})$$0=f(x_\sigma(z),y_\sigma(z),z)$$ $$=A\sigma
z^{1+a}+B\Theta(z^{1+2b})+Cz+P\Theta(z^{1+2b})^{2}+Qz^{2}+S\sigma
z^{1+a}\Theta(z^{1+2b})+T\sigma
z^{1+a}z+o(z^2)$$ $$=A\sigma\Theta(z^{1+a})+B\Theta(z^{1+2b})+C\Theta(z)+Q\Theta(z^{2})+o(z^{2}),$$which implies $A=B=C=Q=0$. Since $E$ is symmetric with respect to $y\leftrightarrow z$, we also have $P=0$, and so $f(x,y,z)=Sxy+Txz+o(x^2+y^2+z^2)$. 

Recall the construction of (\ref{fancy_label_points_on_E_for_J5}) and for the sake of convenience denote $\sigma\in\{\pm\}$ instead of $\sigma\in\{\pm1\}$. Fixing $0<z<10^{-25}$ we have that $y_+(z)- y_-(z)=-2z^{1+a}+o(z^{1+a})\neq0$ so we may define a linear one-variable real valued function $\psi(t):=z^{1+a}+\frac{2z^{1+a}}{y_+-y_-}(t-y_+)$. We note that $\psi(y_+)=z^{1+a}=x_+$ and $\psi(y_-)=z^{1+a}-2z^{1+a}=-z^{1+a}=x_-$, i.e., $\psi(y_\sigma)=x_\sigma$ for $\sigma\in\{\pm1\}$. We also note that $\frac{2z^{1+a}}{y_+-y_-}=\frac{2z^{1+a}}{-2z^{1+a}+o(z^{1+a})}=-1+o(1)$, so we can write $\psi(t)=z^{1+a}+\lambda(z)(t- y_+)$ with $\lambda(z)=-1+o(1)$. For $t\in\mathbb{R}$
 define $F(t)=f(\psi(t),t,z)$. By (\ref{fancy_label_points_on_E_for_J5})
we have $F(t)=f(\psi(t),t,z)=0$ for $t\in\{y_-,y_+\}$, so by Rolle's Theorem there exists $t'\in(y_+,y_-)$ such that $$0=\frac{d}{dt}F(t')=(\psi'(t')\partial_x f+\partial_y f)(\psi(t'),t',z).$$ Note that both $t'=o(z)$ and $\psi(t')=o(z)$. Recalling that $f(x,y,z)=Sxy+Txz+o(x^2+y^2+z^2)$, we have for points of the form $\{(\psi(t'),t',z)\}_{0<z<10^{-25}}$ that $\partial_x f=Sy+Tz+o((x^2+y^2+z^2)^{1/2})=Tz+o(z)$ and $\partial_y f=Sx+o((x^2+y^2+z^2)^{1/2})=o(z)$. Since $\psi'(t)=\lambda(z)=-1+o(z)$ we found that for arbitrary small $z$ there exists a point $(\psi(t'),t',z)$ such that $$0=\frac{d}{dt}F(t')=\lambda(z)(Tz+o(z))+o(z)=-Tz+o(z),$$which is only possible if $T=0$. Since $E$ is symmetric with respect to $y\leftrightarrow
z$, we also have $S=0$, and we conclude that indeed $\langle x^2,xy+xz+yz\rangle_2=I^2(E)$.
\end{example}

\begin{example}\label{example_for_2_3_xy_yz_xz}$\langle xy,yz,xz\rangle_2=I^2(\{(x,y,z)\in\mathbb{R}^2|xy=yz=xz=0\})$.

Indeed, setting $E=\{(x,y,z)\in\mathbb{R}^2|xy=yz=xz=0\}$ we have $xy,yz,xz\in
C^2(\mathbb{R}^3)$ and all three vanish on $E$. So $xy,yz,xz\in I^2(E)$.
It is thus enough to show that if $f\in C^2(\mathbb{R}^3)$ is such that $$f(x,y,z)=ax+by+cz+dx^2+ey^2+gz^2+o(x^2+y^2+z^2)$$
for some $a,b,c,d,e,g\in \mathbb{R}$ and $f|_{E}=0$, then $a=b=c=d=e=g=0$.
Let $f$ be as assumed. Then, for any $x>0$ we have $$0=f(x,0,0)=ax+dx^2+o(x^2),$$which
is only possible if $a=d=0$. As $E$ is symmetric with respect to permutations of the coordinates we also have $b=c=e=g=0$. We  conclude that indeed $\langle xy,yz,xz\rangle_2=I^2(E)$.
\end{example}

\begin{example}\label{example_for_2_3_xy_yz}$\langle xy,yz\rangle_2=I^2(\{(x,y,z)\in\mathbb{R}^2|xy=yz=0\})$.

Indeed, setting $E=\{(x,y,z)\in\mathbb{R}^2|xy=yz=0\}$ we have $xy,yz\in C^2(\mathbb{R}^3)$ and both vanish on $E$. So $xy,yz\in I^2(E)$. It is thus enough to show that if $f\in C^2(\mathbb{R}^3)$ is such that $$f(x,y,z)=ax+by+cz+dxz+ex^2+gy^2+hz^2+o(x^2+y^2+z^2)$$ for some $a,b,c,d,e,g,h\in \mathbb{R}$ and $f|_{E}=0$, then $a=b=c=d=e=g=h=0$. Let $f$ be as assumed. Then, for any $x>0$ we have $$0=f(x,0,0)=ax+ex^2+o(x^2),$$which is only possible if $a=e=0$. As $E$ is symmetric with respect to $x\leftrightarrow z$ we also have  $c=h=0$, and $f(x,y,z)=by+dxz+gy^2+o(x^2+y^2+z^2)$. For any $y>0$ we have $$0=f(0,y,0)=by+gy^2+o(y^2),$$which
is only possible if $b=g=0$ and so $f(x,y,z)=dxz+o(x^2+y^2+z^2)$. Now, for any $x>0$ we have $$0=f(x,0,x)=dx^2+o(x^2),$$which is only possible if $d=0$. We conclude that indeed $\langle xy,yz\rangle_2=I^2(E)$.
\end{example}

\begin{example}\label{example_for_2_3_xy_z(x+y)}Set $\delta=10^{-5}$, $\gamma=10^{-15}$ and $E=E_1\cup E_2\cup E_3$, with $$E_1=\{x=z=0\};$$ $$E_2=\{y=z=0\};$$
$$E_3=\{z>0\}\cap\{ xy=\abs{z}^{2+2\delta}\}\cap\{x+y=\abs{z}^{1+\gamma}\}.$$ Then, $\langle xy,z(x+y)\rangle_2=I^2(E)$.

Indeed,
we first note that
$$z(x+y-\abs{z}^{1+\gamma}), xy-\abs{z}^{2+2\gamma}\in
C^2(\mathbb{R}^3)$$ and both vanish on $E$. So $z(x+y),xy\in I^2(E)$.
It is thus enough to show that if $f\in C^2(\mathbb{R}^3)$ is such that $$f(x,y,z)=ax+by+cz+d(x+y)^2+ex^2+gz^2+hxz+o(x^2+y^2+z^2)$$
for some $a,b,c,d,e,g,h\in \mathbb{R}$ and $f|_{E}=0$, then $a=b=c=d=e=g=h=0$.
Let $f$ be as assumed. Since $f|_{E_1}=0$ we have for any $y>0$:
$$0=f(0,y,0)=by+dy^{2}+o(y^{2}),$$
which implies $b=d=0$, and so $f(x,y,z)=ax+cz+ex^2+gz^2+hxz+o(x^2+y^2+z^2)$. Since $f|_{E_2}=0$ we have for any $x>0$:
$$0=f(x,0,0)=ax+ex^{2}+o(x^{2}),$$
which implies $a=e=0$, and so $f(x,y,z)=cz+gz^2+hxz+o(x^2+y^2+z^2)$. 

We now find some points in $E_3=\{z>0\}\cap\{ xy=\abs{z}^{2+2\delta}\}\cap\{x+y=\abs{z}^{1+\gamma}\}$. On $E_3$ we have $(x-y)^2=(x+y)^2-4xy=\abs{z}^{2+2\gamma}-4\abs{z}^{2+2\delta}=\abs{z}^{2+2\gamma}(1-4\abs{z}^{2\delta-2\gamma})$, so when $0<z<10^{-50}$ we have $\abs{x-y}=(\abs{z}^{2+2\gamma}(1-4\abs{z}^{2\delta-2\gamma}))^{1/2}=z^{1+\gamma}(1-4z^{2\delta-2\gamma})^{1/2}$. We also have $x+y=z^{1+\gamma}$. We can now define $$\phi_{\pm}(z):=\frac{1}{2}(x+y\pm\abs{x-y})=\frac{1}{2}\big[z^{1+\gamma}\pm z^{1+\gamma}(1-4z^{2\delta-2\gamma})^{1/2}\big].$$Note
that $\phi_+(z)=o(\abs{z})$ and  $\phi_-(z)=o(\abs{z})$ and moreover, for $0<z<10^{-50}$ we have that $\phi_+(z)\neq\phi_-(z)$. Tracing these steps back one can easily verify that for
$0<z<10^{-50}$ we have $(\phi_+(z),\phi_-(z),z)\in E_3$. Also note that $E_3$ is symmetric with respect to $x\leftrightarrow y$. We conclude that \begin{multline}\label{nice_example_label_parametrization_of_a_part_of_E_new_in_a_different_case}\text{for
any }0<z<10^{-50}\text{ there exists  }\phi_+(z)\neq \phi_-(z)\text{ such that both } \\ (\phi_+(z),\phi_-(z),z)\in E_{3}\text{ and }(\phi_-(z),\phi_+(z),z)\in
E, \\ \text{and moreover both }\phi_+(z)=o(z)\text{ and }\phi_-(z)=o(z).\end{multline}Recall we have $f(x,y,z)=cz+gz^2+hxz+o(x^2+y^2+z^2)$ for some $c,g,h\in\mathbb{R}$ such that $f|_E=0$ and we need to show that $c=g=h=0$. By (\ref{nice_example_label_parametrization_of_a_part_of_E_new_in_a_different_case}) we have for any $0<z<10^{-50}$ $$0=f(\phi_+(z),\phi_-(z),z)=cz+gz^2+h\phi_+(z)z+o(z^2)=cz+gz^2+o(z^2),$$which implies $c=g=0$, and so $f(x,y,z)=hxz+o(x^2+y^2+z^2)$.
Fixing any $0<z<10^{-50}$, by (\ref{nice_example_label_parametrization_of_a_part_of_E_new_in_a_different_case}) we have $f(\phi_+(z),\phi_-(z),z)=f(\phi_-(z),\phi_+(z),z)=0$. For $t\geq0$  define $$F(t)=f(t\phi_+(z)+(1-t)\phi_-(z),t\phi_-(z)+(1-t)\phi_+(z),z).$$ By Rolle's Theorem there exists $t'\in[0,1]$ such that $$0=\frac{d}{dt}F(t')=(\phi_+(z)-\phi_-(z))(f_x-f_y)(t'\phi_+(z)+(1-t')\phi_-(z),t'\phi_-(z)+(1-t')\phi_+(z),z).$$ We thus found a point $(\tilde x(z),\tilde y(z),z):=(t'\phi_+(z)+(1-t')\phi_-(z),t'\phi_-(z)+(1-t')\phi_+(z),z)$ with $\tilde x(z)=o(z)$ and $\tilde y(z)=o(z)$ such that $(f_x-f_y)(\tilde x(z),\tilde y(z),z)=0$. Recall we also have $f(x,y,z)=hxz+o(x^2+y^2+z^2)$, so $(f_x-f_y)(x,y,z)=hz+o((x^2+y^2+z^2)^{1/2})$. Consequently, $$0=(f_x-f_y)(\tilde x(z),\tilde y(z),z)=hz+o(\abs{z}),$$which
is only possible if $h=0$. We 
conclude that indeed $\langle xy,z(x+y)\rangle_2=I^2(E)$.
\end{example}

\begin{example}\label{example_for_2_3_xy+yz_xz+yz}Set $E=E_x\cup E_y\cup E_z\cup E_0$, with $$E_x=\{(x,y,z)\in\mathbb{R}^3|y=z=0\};$$ $$E_y=\{(x,y,z)\in\mathbb{R}^3|x=z=0\};$$ $$E_z=\{(x,y,z)\in\mathbb{R}^3|x=y=0\};$$ $$E_0=\{(x,y,z)\in\mathbb{R}^3|x+y=x+z=0\}.$$
Then, $\langle
xy+yz,xz+yz\rangle_2=I^2(E)$.

Indeed, we have $xy+yz,xz+yz\in
C^2(\mathbb{R}^3)$ and both vanish on $E$. So $xy+yz,xz+yz\in I^2(E)$. It is thus enough
to show that if $f\in C^2(\mathbb{R}^3)$ is such that $$f(x,y,z)=ax+by+cz+dyz+ex^2+gy^2+hz^2+o(x^2+y^2+z^2)$$
for some $a,b,c,d,e,g,h\in \mathbb{R}$ and $f|_{E}=0$, then $a=b=c=d=e=g=h=0$.
Let $f$ be as assumed. Then, for any $x>0$ we have $$0=f(x,0,0)=ax+ex^2+o(x^2),$$which
is only possible if $a=e=0$. Similarly, for any $y>0$ we have $$0=f(0,y,0)=by+gy^2+o(y^2),$$which
is only possible if $b=g=0$.
Lastly, for any $z>0$ we have $$0=f(0,0,z)=cz+hz^2+o(z^2),$$which
is only possible if $c=h=0$, and so $f(x,y,z)=dyz+o(x^2+y^2+z^2)$. Now, for any $x>0$ we have $$0=f(x,-x,-x)=dx^{2}+o(x^2),$$which
is only possible if $d=0$. We conclude that indeed $\langle
xy+yz,xz+yz\rangle_2=I^2(E)$.
\end{example}

\end{document}